\documentclass[11pt]{report}
\usepackage{indentfirst}
\usepackage{color}
\usepackage{geometry,amsthm,amsmath,amsfonts,amssymb,times,enumerate}
\usepackage{scalefnt}
\usepackage{newlfont,setspace}
\usepackage{graphics}
\usepackage{graphicx}
\usepackage{layout}
\usepackage[colorlinks,linkcolor=red,urlcolor=green,filecolor=green,citecolor=blue]{hyperref}
\usepackage{hyperref}
\usepackage{url}
\usepackage{fancyhdr}

\providecommand{\MR}{\relax\ifhmode\unskip\space\fi MR }
\providecommand{\MRhref}[2]{  \href{http://www.ams.org/mathscinet-getitem?mr=#1}{#2}}
\providecommand{\href}[2]{#2}
\newtheoremstyle{break}
{12pt}{12pt}{\normalfont}{}{\bfseries}{.}{\newline}{}
\theoremstyle{break}
\newtheorem{theorem}{Theorem}[section]
\newtheorem{proposition}[theorem]{Proposition}
\newtheorem{lemma}[theorem]{Lemma}
\newtheorem{corollary}[theorem]{Corollary}
\newtheorem{definition}[theorem]{Definition}
\newtheorem{example}[theorem]{Example}

\newtheorem{remark}[theorem]{Remark}
\newenvironment{proof}[1][Proof]{  
	\par\noindent\textbf{#1.}\par 
}{\hfill$\square$\par }
\makeatletter

\begin{document}

\begin{titlepage}
	\centering
	
	\vspace*{2cm}
	
	{\LARGE \textbf{MOHAMED BOUDIAF UNIVERSITY OF M'SILA}} \\
	\vspace{0.5cm}
	{\Large \textbf{FACULTY OF MATHEMATICS AND COMPUTER SCIENCES}} \\
	\vspace{0.5cm}
	{\large \textbf{DEPARTMENT OF MATHEMATICS}} \\
	
	\vspace{2cm}
	
	{\Huge \textbf{Study of Composition Operators}} \\
	\vspace{0.5cm}
	{\Huge \textbf{in Certain Functional Spaces}} \\
	
	\vspace{2cm}
	
	{\Large \textbf{A Thesis Submitted in Partial Fulfillment}} \\
	\vspace{0.3cm}
	{\Large \textbf{of the Requirements for the Degree of}} \\
	\vspace{0.3cm}
	{\Large \textbf{MAGISTER}} \\
	\vspace{0.3cm}
	{\Large \textbf{in Functional Analysis}} \\
	
	\vspace{2cm}
	
	{\Large \textbf{By}} \\
	\vspace{0.5cm}
	{\Large \textbf{Mahdi Tahar BRAHIMI}} \\
	
	\vspace{2cm}
	
	{\Large \textbf{Supervised by}} \\
	\vspace{0.5cm}
	{\Large \textbf{Professor Madani MOUSSAI}} \\
	
	\vspace{1cm}
	
	{\Large \textbf{Academic Year: 2007/2008}}
	
\end{titlepage}

 {\thispagestyle{empty}}
\setstretch{1.5}
\begin{abstract}
	In this thesis we study three problems\,,\,the first is the superposition of the operators and their properties\,,\, such as boundedness,\,continuity,\,regularity and the inequalities of the norms of the composition of functions in some functional spaces\,,\, the second is to generalize some results of the composition of more than two functions\,,\, and the third is to give a generalization of Peetre's theorem.
\end{abstract}
\setstretch{1}
\tableofcontents
\chapter*{ \center{Notations}}

\addcontentsline{toc}{section}{\textbf{Notations}}
\setstretch{2}
\begin{itemize}
\item For all $\alpha = (\alpha_1, \alpha_2, ..., \alpha_n) \in \mathbb{N}^n$%
, $x = (x_1, x_2, ..., x_n) \in \mathbb{R}^n$, we have: 
\begin{equation*}
D^{\alpha}f = \frac{\partial^{|\alpha|} f}{\partial x_1^{\alpha_1} \cdots
\partial x_n^{\alpha_n}}, \quad \text{where} \quad |\alpha| = \alpha_1 +
\alpha_2 + ... + \alpha_n,
\end{equation*}

\item $C = C(\mathbb{R}^n)$: Denotes the space of continuous functions. 
\begin{equation*}
C(\mathbb{R}^n) = \{ u : \mathbb{R}^n \to \mathbb{R} ; \forall x_0 \in 
\mathbb{R}^n, \forall \varepsilon > 0, \exists \delta > 0, \forall x \in 
\mathbb{R}^n, \|x - x_0\| < \delta \Rightarrow |u(x) - u(x_0)| < \varepsilon
\}
\end{equation*}

\begin{itemize}
\item $C^r(\mathbb{R}^n) = \{ f \in C(\mathbb{R}^n) ; D^\alpha f \in C(%
\mathbb{R}^n), \text{ for all } |\alpha| \leq r \}$

\item $\mathcal{E}(\mathbb{R}^n) = C^\infty(\mathbb{R}^n) = \{ f \in C(%
\mathbb{R}^n) ; D^\alpha f \in C(\mathbb{R}^n), \text{ for all } \alpha \in 
\mathbb{N}^n \}$

\item $\text{supp}(f) = \overline{\{ x \in \mathbb{R}^n ; f(x) \neq 0 \}}$:
Support of the function $f$

\item $\mathcal{D}(\mathbb{R}^n) = C_0^\infty(\mathbb{R}^n) = \{ f \in
C^\infty(\mathbb{R}^n) ; \text{supp}(f) \text{ is compact} \}$

\item $\mathcal{S}(\mathbb{R}^n)$: The Schwartz space and $\mathcal{S}%
^{\prime }(\mathbb{R}^n)$ the space of tempered distributions
\end{itemize}

$\mathcal{S}(\mathbb{R}^{n})=\left\{ \varphi \in C^{\infty }(\mathbb{R}%
^{n});\quad \Vert \varphi \Vert _{\mathcal{S}(\mathbb{R}^{n})}\sim p_{k,m}(\varphi
)=\sup_{x\in \mathbb{R}^{n}}\left( \underset{\alpha \in \mathbb{N}%
	^{n},|\alpha |\leq m}{\sup }\left( \sup\limits_{\beta \in \mathbb{N}%
	^{n},|\beta |\leq k}|x^{\beta }\partial ^{\alpha }\varphi (x)|\right)
\right) <\infty \right\} $ \vspace{10pt}

$\mathcal{S}^{\prime }(\mathbb{R}^{n})=\left\{ T:\mathcal{S}\rightarrow 
\mathbb{R},\exists k\in \mathbb{N},\exists m\in \mathbb{N},\exists
C_{k,m}>0,\forall \varphi \in \mathcal{S}(\mathbb{R}^{n});|\langle T,\varphi
\rangle |<C_{k,m}p_{k,m}(\varphi )\right\} $

\item We define the norm of the Lebesgue space $L^{p}(\mathbb{R}^{n})$ by 
\begin{eqnarray*}
\Vert f\Vert _{L^{p}(\mathbb{R}^{n})} &=&\left( \int_{\mathbb{R}%
^{n}}|f(x)|^{p}\,\mathrm{d}x\right) ^{\frac{1}{p}},\quad \text{if }p\geq 1 \\
\Vert f\Vert _{L^{\infty }(\mathbb{R}^{n})} &=&\inf \{c>0;|f(x)|\leq c,\text{
(a.e.)}\}=\text{ess sup}_{x\in \mathbb{R}^{n}}|f(x)|
\end{eqnarray*}

\item For all $f,g\in \mathcal{S}(\mathbb{R}^{n})$, the convolution $f\ast g$
satisfies $f\ast g\in \mathcal{S}(\mathbb{R}^{n})$, such that 
\begin{equation*}
\forall \xi \in \mathbb{R}^{n};f\ast g(\xi )=\int_{\mathbb{R}^{n}}\tau
_{t}f(\xi )g(t)\,\mathrm{d}t=\int_{\mathbb{R}^{n}}f(u)\tau _{u}g(\xi )\,%
\mathrm{d}u,\quad \text{where }\tau _{u}g(\xi )=g(\xi -u)
\end{equation*}

\item If $f\in \mathcal{S}(\mathbb{R}^{n})$, we define the Fourier transform
and its inverse by 
\begin{eqnarray*}
\mathcal{F}f(\xi ) &=&\hat{f}(\xi )\sim \int_{\mathbb{R}^{n}}e^{-i\langle
x,\xi \rangle }f(x)\,\mathrm{d}x,\quad \text{for all }\xi \in \mathbb{R}^{n}
\\
\mathcal{F}^{-1}f(\xi ) &=&\check{f}(\xi )\sim (2\pi )^{-n}\int_{\mathbb{R}%
^{n}}e^{i\langle x,\xi \rangle }f(x)\,\mathrm{d}x,\quad \text{for all }\xi
\in \mathbb{R}^{n} \\
(\mathcal{F}T)(f) &=&\hat{T}(f)=T(\hat{f}),\quad (\mathcal{F}^{-1}T)(f)=%
\check{T}(f)=T(\check{f}),\quad \text{for all }T\in \mathcal{S}^{\prime }(%
\mathbb{R}^{n})
\end{eqnarray*}

\item $[s]$: Denotes the integer part of $s\in \mathbb{R}$.

\item $p^{\prime }$: Denotes the conjugate exponent of $p\geq 1$, defined by 
$\frac{1}{p}+\frac{1}{p^{\prime }}=1$.

\item $\mathcal{P}_{m}(\mathbb{R}^{n})$: The set of polynomials of degree at
most $m$, $m\in \mathbb{N}$, such that 
\begin{equation*}
\mathcal{P}_{m}(\mathbb{R}^{n})=\left\{ P\in \mathcal{S}^{\prime }(\mathbb{R}%
^{n});P(x)=\sum_{|\alpha |\leq m}a_{\alpha }x^{\alpha },x\in \mathbb{R}%
^{n}\right\} ,\quad \mathcal{P}_{-1}(\mathbb{R}^{n})=\{0\}
\end{equation*}

\begin{itemize}
\item $\mathcal{P}(\mathbb{R}^{n})=\mathcal{P}_{\infty }(\mathbb{R}^{n})$:
The subspace of $\mathcal{S}^{\prime }(\mathbb{R}^{n})$ of polynomials or
multinomials.

\item $[f]$: The equivalence class of a distribution $f\in \mathcal{S}%
^{\prime }(\mathbb{R}^{n})$ modulo $\mathcal{P}_{\infty }(\mathbb{R}^{n})$.

\item We agree to take $x^{\alpha }=x_{1}^{\alpha _{1}}x_{2}^{\alpha
_{2}}\cdots x_{n}^{\alpha _{n}}$, $\alpha !=\alpha _{1}!\alpha _{2}!\cdots
\alpha _{n}!$, $t^{\alpha }x=(t^{\alpha _{1}}x_{1},t^{\alpha
_{2}}x_{2},\ldots ,t^{\alpha _{n}}x_{n})$, $(t\geq 0)$, and $\langle
x,y\rangle =x_{1}y_{1}+x_{2}y_{2}+\cdots +x_{n}y_{n}$.
\end{itemize}

\item We introduce the space $\mathcal{S}_{0}(\mathbb{R}^{n})$ defined by
functions in $\mathcal{S}(\mathbb{R}^{n})$ with zero moments 
\begin{eqnarray*}
\mathcal{S}_{0}(\mathbb{R}^{n}) &=&\left\{ f\in \mathcal{S}(\mathbb{R}%
^{n});\partial ^{\alpha }\mathcal{F}f(0)=\int_{\mathbb{R}^{n}}x^{\alpha
}f(x)\,\mathrm{d}x=0,\text{ for all }\alpha \in \mathbb{N}^{n}\right\} \\
\widehat{\mathcal{S}}_{0}(\mathbb{R}^{n}) &=&\left\{ \hat{f};f\in \mathcal{S}%
_{0}(\mathbb{R}^{n})\right\} =\left\{ f\in \mathcal{S}(\mathbb{R}%
^{n});\partial ^{\alpha }f(0)=0,\text{ for all }\alpha \in \mathbb{N}%
^{n}\right\} \\
\mathcal{S}_{0}^{\prime }(\mathbb{R}^{n}) &=&\mathcal{S}^{\prime }(\mathbb{R}%
^{n})/\mathcal{S}_{0}^{\perp }(\mathbb{R}^{n})=\mathcal{S}^{\prime }(\mathbb{%
R}^{n})/\mathcal{P}(\mathbb{R}^{n})
\end{eqnarray*}%

\item For each $s\in \mathbb{R}^{+}\setminus \mathbb{N}$, we define the
Holder space by  
\begin{eqnarray*}
\mathcal{C}^{s}(\mathbb{R}^{n}) &=&\left\{ f\in L^{\infty };\Vert f\Vert _{%
\mathcal{C}^{s}(\mathbb{R}^{n})}=\Vert f\Vert _{\infty }+\sup_{x,y\in 
\mathbb{R}^{n},x\neq y}\frac{|f(x)-f(y)|}{|x-y|^{s}}<\infty \right\} ,\quad 
\text{if }0<s<1 \\
\mathcal{C}^{s}(\mathbb{R}^{n}) &=&\left\{ f\in \mathcal{C}^{[s]}(\mathbb{R}%
^{n});\Vert f\Vert _{\mathcal{C}^{s}(\mathbb{R}^{n})}=\sum_{|\beta |\leq
\lbrack s]}\Vert D^{\beta }f\Vert _{\mathcal{C}^{\alpha }(\mathbb{R}%
^{n})}<\infty \right\} ,\quad \text{if }s=[s]+\alpha ,0<\alpha <1
\end{eqnarray*}%
For each $s\in \mathbb{N}$, we define the Zygmund space by 
\begin{eqnarray*}
\mathcal{C}^{1}(\mathbb{R}^{n}) &=&\left\{ f\in C^{0}(\mathbb{R}^{n});\Vert
f\Vert _{\mathcal{C}^{1}(\mathbb{R}^{n})}=\sup_{h\neq 0,x\in \mathbb{R}^{n}}%
\frac{|u(x+h)+u(x-h)-2u(x)|}{|h|}<\infty \right\} \\
\mathcal{C}^{s}(\mathbb{R}^{n}) &=&\left\{ f\in \mathcal{C}^{1}(\mathbb{R}%
^{n});D^{\alpha }f\in \mathcal{C}^{1},\text{ for all }|\alpha |\leq
s-1\right\} ,\quad \text{if }s\neq 1
\end{eqnarray*}%

\item $A\lesssim B$: Means that for two parametric expressions $A$ and $B$,
there exists an independent constant $c>0$, such that $A\leq cB$.

\begin{itemize}
\item $A\sim B$: If $A\lesssim B$ and $B\lesssim A$, then there exist $%
c_{1},c_{2}>0$, such that $c_{1}B\leq A\leq c_{2}B$.
\end{itemize}
\end{itemize}
\setstretch{1.2}
\chapter*{ \center{Introduction}}

\addcontentsline{toc}{section}{\textbf{Introduction}}

\begin{quote}
The main results of this thesis are related to three problems.
\end{quote}

\begin{enumerate}
\item The first problem has been studied by several authors, G. Bourdaud, ~%
\cite{Bo-Cr}, ~\cite{Bo-Cr-Si-1}, ~\cite{Bo-Cr-Si-2}, W. Sickel ~\cite{SIC},
D. Kateb ~\cite{KAT}, S. Igari ~\cite{IGA}, and consists in solving the
superposition operator problem by finding a necessary and sufficient
condition for a composition operator $T_G: E \to E$, such that $T_G(f) = G
\circ f$, where $G$ is a real-valued function, to satisfy $T_G(E) \subset E$%
. Several researchers have tried to find non-trivial operators (associated
with non-affine functions) verifying the condition $T_G(E) \subset E$, for a
given functional space $E$.

\item The second problem consists in giving a general formula for the
derivative of the composition of several functions, using a basic inequality
introduced in ~\cite{Bo-Cr-Si-1} to generalize some results, studying the
properties of composition operators (bounded in particular), such as
differentiability, and verifying norm inequalities for the composition of
more than two functions.

\item The third problem consists in giving an extension to the famous
theorem of Peetre, based on the works introduced in ~\cite{Bo-Cr-Si-1} and ~%
\cite{Pee}.
\end{enumerate}

\begin{quote}
Our work is organized into four chapters:
\end{quote}

\begin{itemize}
\item In the first chapter, we present known functional spaces, such as
Besov spaces and their properties. We define the spaces $BV_p(\mathbb{R})$, $%
BV_p^1(\mathbb{R})$ from the spaces of functions of bounded p-variation for $%
p \geq 1$ and their norms.

\item In the second chapter, we study the composition operator problem in
the spaces $BV_p^1(\mathbb{R})$, $1 \leq p < \infty$, based on the works
introduced in ~\cite{Bo-Cr-Si-1}, ~\cite{Bo-Cr-Si-2}, as well as some
multiplication and continuity properties, and we will extend some results
concerning a basic inequality for the composition of several functions.
Thus, we obtain a basic algorithm giving the derivative of the composition
of several functions.

\item In the third chapter, we study functional calculus in homogeneous
Besov spaces $\dot{B}_p^{s,q}(\mathbb{R}^n)$, presenting the necessary and
sufficient conditions for composition operators.

\item In the fourth chapter, we present a new functional space $BV_p^\alpha(%
\mathbb{R})$ defined thanks to the space $\mathcal{V}_p^\alpha(\mathbb{R})$,
introduced in ~\cite{Pee}, and we generalize Peetre's Theorem (\ref{3.17})
to the spaces $BV_p^\alpha(\mathbb{R})$, $p \in ]1, +\infty[$, $0 \leq
\alpha < 1$, as well as application examples to affirm the stated results.
\end{itemize}
\setstretch{1.5}
\chapter{\textbf{Some preliminary results}}

\pagenumbering{arabic}\pagestyle{plain} \setcounter{page}{1}

\begin{quote}
In this chapter, we will recall essential notions, namely the
Littlewood-Paley series as well as finite differences, which will allow us
to construct equivalent norms for Besov spaces; finally, we will give an
overview of functions of bounded p-variation, as well as their properties
that will be very useful later.
\end{quote}

\section{\textbf{Littlewood-Paley series}}

\begin{quote}
In this paragraph, we give a definition of the partition of unity in $%
C_0^\infty(\mathbb{R}^n)$, followed by an example, and for more details see ~%
\cite{ULLR}.
\end{quote}

\subsection{\textbf{Partition of unity}}

\begin{quote}
Let the sequence of reals $A = \{A_j\}_{j \in \mathbb{N}} \subset \mathbb{R}%
^+$, such that there exist $0 < \lambda_0 \leq \lambda_1$ with 
\begin{equation*}
\lambda_0 A_j \leq A_{j+1} \leq \lambda_1 A_j, \quad j \in \mathbb{N}.
\end{equation*}
Then there exists $\ell_0 \in \mathbb{N}$ such that 
\begin{equation*}
2A_j \leq A_k, \quad \text{for all } j, k \text{ and } j + \ell_0 \leq k.
\end{equation*}
We define $\Omega_A = \{\Omega_{j,A}\}_{j \in \mathbb{N}}$, a covering of $%
\mathbb{R}^n$, associated with $A$, such that 
\begin{equation*}
\Omega_{j,A} = \left\{ 
\begin{array}{ccc}
\left\{ \xi \in \mathbb{R}^n ; |\xi| \leq A_{j+\ell_0} \right\}, & \text{if }
& j = 0,1,...,\ell_0-1 \\ 
\left\{ \xi \in \mathbb{R}^n ; A_{j-\ell_0} \leq |\xi| \leq A_{j+\ell_0}
\right\}, & \text{if } & j \geq \ell_0%
\end{array}
\right.
\end{equation*}
\end{quote}

\begin{definition}
\label{1.1} We say that the sequence of functions $\varphi_A =
\{\varphi_{j,A}\}_{j \in \mathbb{N}} \subset C_0^\infty(\mathbb{R}^n)$ is a
partition of unity with respect to $\Omega_A$ if the following conditions
are satisfied:

\begin{enumerate}
\item[(i)] $\varphi_{j,A}(\xi) \geq 0, \quad \xi \in \mathbb{R}^n, \quad j
\in \mathbb{N}$

\item[(ii)] $\text{supp}(\varphi_{j,A}) \subset \Omega_{j,A}, \quad j \in 
\mathbb{N}$

\item[(iii)] For all $\alpha \in \mathbb{N}^n$, there exists a constant $%
c_\alpha > 0$, such that 
\begin{equation*}
|D^\alpha \varphi_{j,A}(\xi)| \leq c_\alpha (1 + |\xi|^2)^{-|\alpha|/2},
\quad \xi \in \mathbb{R}^n, \quad j \in \mathbb{N}
\end{equation*}

\item[(iv)] There exists a constant $c_\varphi > 0$, such that 
\begin{equation*}
0 < \sum_{j=0}^\infty \varphi_{j,A}(\xi) = c_\varphi < \infty, \quad \xi \in 
\mathbb{R}^n.
\end{equation*}
\end{enumerate}
\end{definition}

For $A_j = 2^j$, $\ell_0 = 1$, $c_\varphi = 1$, we have a dyadic partition
of unity.\\ The partition $\varphi_A = \{\varphi_{j,A}\}_{j \in \mathbb{N}}$
is called inhomogeneous and can be replaced by the partition $%
\{\varphi_{j,A}\}_{j \in \mathbb{Z}}$, called homogeneous.

\begin{example}
\label{1.2} Let $K > 1$, consider the covering $\{C_p\}_{-1}^{+\infty}$,
defined by 
\begin{equation*}
\left\{ 
\begin{array}{ccc}
C_p & = & \left\{ \xi \in \mathbb{R}^n ; K^{-1} 2^p \leq |\xi| \leq K
2^{p+1} \right\} \\ 
C_{-1} & = & \overline{B}(0, K) = \left\{ \xi \in \mathbb{R}^n ; |\xi| \leq
K \right\}%
\end{array}
\right.
\end{equation*}
$\{C_p\}_{-1}^{+\infty}$ is a uniformly finite covering for $\mathbb{R}^n$,
i.e., the set $\{q \in \mathbb{N} ; C_q \cap C_p \neq \varnothing \}$ is
finite.

We can construct $\{\phi_0, \psi_\nu\}$, $\nu \in \mathbb{N}$, an
inhomogeneous Littlewood-Paley decomposition of unity such that 
\begin{equation*}
\psi_\nu(\xi) = \phi_\nu(\xi) - \phi_{\nu-1}(\xi), \quad \text{and} \quad
\phi_\nu(\xi) = \phi_0(2^{-\nu} \xi),
\end{equation*}
where $\phi_0 \in C_0^\infty(\mathbb{R}^n)$ is real, and $\phi_0 \equiv 1$
on the closed ball $\overline{B}(0,1)$, $\text{supp}(\phi_0) \subset 
\overline{B}(0,K)$, such that 
\begin{equation*}
\left( \phi_0 + \sum_{\nu \in \mathbb{N}} \psi_\nu \right) u = u, \quad u
\in \mathcal{S}^{\prime }(\mathbb{R}^n).
\end{equation*}
Indeed, such a decomposition exists because we have the following lemma.
\end{example}

\begin{lemma}
\label{1.3} There exist $\varphi, \psi \in C_0^\infty(\mathbb{R}^n)$, with $%
\text{supp}(\psi) \subset C_{-1}$, $\text{supp}(\varphi) \subset C_0$, such
that 
\begin{eqnarray*}
\psi(\xi) + \sum_{p=0}^\infty \varphi(2^{-p}\xi) &=& 1, \\
\psi(\xi) + \sum_{p=0}^N \varphi(2^{-p}\xi) &=& \psi(2^{-(N+1)}\xi).
\end{eqnarray*}
\end{lemma}

\begin{proof}
Let $\theta \in C_0^\infty(\mathbb{R}^n)$, with $0 \leq \theta \leq 1$, $%
\text{supp}(\theta) \subset C_0$, $\theta(\xi) = 1$ for $1 \leq |\xi| \leq 2$%
.

Define $s(\xi) = \sum_{p=-\infty}^\infty \theta(2^{-p}\xi)$, $\xi \in 
\mathbb{R}^n \setminus \{0\}$.

Since $\{C_p\}_{-1}^{+\infty}$ is a uniformly finite covering for $\mathbb{R}%
^n$, then $s \in C^\infty(\mathbb{R}^n \setminus \{0\})$.\\ Define $\varphi
\in C_0^\infty(\mathbb{R}^n)$ by $$\varphi(\xi) = \frac{\theta(\xi)}{s(\xi)},$$
such that $s(2^{-p}\xi) = s(\xi)$.

For $|\xi| \geq K$, $p \leq -1$, we have $2^{-p}|\xi| = 2^{|p|}|\xi| \geq
2^{|p|}K \geq 2K$, and $2^{-p}\xi \notin C_0$, hence $\theta(2^{-p}\xi) = 0$%
. Therefore, if $|\xi| \geq K$, $\xi \in \mathbb{R}^n \setminus C_{-1}$,
then 
\begin{equation*}
\sum_{p=0}^\infty \varphi(2^{-p}\xi) = \sum_{p=-\infty}^\infty
\varphi(2^{-p}\xi) = \sum_{p=-\infty}^\infty \frac{\theta(2^{-p}\xi)}{%
s(2^{-p}\xi)} = \frac{\sum_{p=-\infty}^\infty \theta(2^{-p}\xi)}{%
\sum_{p=-\infty}^\infty \theta(2^{-p}\xi)} = 1.
\end{equation*}

Take $\psi(\xi) = 1 - \sum_{p=0}^\infty \varphi(2^{-p}\xi)$, then $\psi \in
C_0^\infty(\mathbb{R}^n)$, with $\text{supp}(\psi) \subset C_{-1}$.\\ For all $%
j \in \mathbb{Z}$, let $\psi_j(\xi) = \varphi(2^{-j}|\xi|)$ and $\psi_0(\xi)
= 1 - \sum_{k=1}^\infty \psi_k(\xi)$.\\ Then we have the following properties: 
\begin{eqnarray*}
\psi_j \text{ is even such that } \text{supp}(\psi_j) &\subset& \overline{C}%
_j, \quad \text{for all } j \in \mathbb{N} \\
\sum_{j=0}^\infty \psi_j(\xi) &=& 1, \quad \text{for all } \xi \in \mathbb{R}%
^n \\
C_m \cap \text{supp}(\psi_j) &=& \varnothing, \quad \text{if } |m - j| > 1 \\
\psi_{j-1}(\xi) + \psi_j(\xi) + \psi_{j+1}(\xi) &=& 1, \quad \text{for all }
\xi \in \text{supp}(\psi_j), \quad j \in \mathbb{N}.
\end{eqnarray*}

To construct a homogeneous partition on $\mathbb{Z}$, we introduce the
sequence $\Phi_j \in \mathcal{S}(\mathbb{R}^n)$, such that: 
\begin{eqnarray*}
\hat{\Phi}_j(\xi) &=& \frac{\psi_j(\xi)}{\sum_{k \in \mathbb{Z}} \psi_k(\xi)}
\\
\Phi_j(\xi) &=& 2^{jn} \Phi_0(2^j \xi), \quad \text{for all } \xi \in 
\mathbb{R}^n \\
\int_{\mathbb{R}^n} x^k \Phi_j(x) \,\mathrm{d}x &=& \hat{\Phi}_j^{(k)}(0) =
0, \quad (k \in \mathbb{N}^n).
\end{eqnarray*}
We have 
\begin{equation*}
\sum_{j \in \mathbb{Z}} \hat{\Phi}_j(\xi) = \sum_{j \in \mathbb{Z}} \left( 
\frac{\psi_j(\xi)}{\sum_{k \in \mathbb{Z}} \psi_k(\xi)} \right) = \frac{%
\sum_{j \in \mathbb{Z}} \psi_j(\xi)}{\sum_{k \in \mathbb{Z}} \psi_k(\xi)} =
1, \quad (\xi \in \mathbb{R}^n \setminus \{0\}).
\end{equation*}
Since $1 - \sum_{j=1}^\infty \hat{\Phi}_j$ is infinitely differentiable and
compactly supported,\\ we can choose a function $\Psi \in \mathcal{S}(\mathbb{R%
}^n)$, such that $\hat{\Psi} = 1 - \sum_{j=1}^\infty \hat{\Phi}_j$, where $%
\hat{\Psi} \neq 0$ on $|\xi| \leq 1$, and thus: $$\hat{\Psi} +
\sum_{j=1}^\infty \hat{\Phi}_j = 1.$$

Set $\psi(\xi) = \hat{\Psi}(\xi)$ and $\varphi(2^{-j}\xi) = \hat{\Phi}%
_j(\xi) $, then we have 
\begin{equation}
\psi(\xi) + \sum_{j=1}^\infty \varphi(2^{-j}\xi) = 1, \quad (\xi \in \mathbb{%
R}^n).  \label{[1.1]}
\end{equation}
We have $\psi(2^{-N}\xi) + \sum_{p=0}^\infty \varphi(2^{-p-N}\xi) =
\psi(\xi) + \sum_{p=0}^{N-1} \varphi(2^{-p}\xi) + \sum_{p=N}^\infty
\varphi(2^{-p}\xi) = 1$, hence 
\begin{eqnarray*}
\sum_{p=0}^\infty \varphi(2^{-p-N}\xi) &=& \sum_{p=N}^\infty
\varphi(2^{-p}\xi), \\
\psi(2^{-N}\xi) &=& \psi(\xi) + \sum_{p=0}^{N-1} \varphi(2^{-p}\xi).
\end{eqnarray*}
By induction, if $\varphi(2^{-N}\xi) = \psi(2^{-(N+1)}\xi) - \psi(2^{-N}\xi)$%
, and since $\psi(2^{-(N+1)}\xi) = \psi(\xi) + \sum_{p=0}^N
\varphi(2^{-p}\xi)$, then we have 
\begin{eqnarray*}
\psi(2^{-(N+2)}\xi) &=& \psi(\xi) + \sum_{p=0}^{N+1} \varphi(2^{-p}\xi) \\
&=& \psi(2^{-(N+1)}\xi) + \varphi(2^{-(N+1)}\xi),
\end{eqnarray*}
hence $\varphi(2^{-(N+1)}\xi) = \psi(2^{-(N+2)}\xi) - \psi(2^{-(N+1)}\xi)$.
\end{proof}

\subsection{\textbf{Decomposition of a function }$f \in \mathcal{S}^{\prime
}(\mathbb{R}^n)$}

Since $\Phi_j$ is real and even, $\hat{\Phi}_j$ is real and even and its
support is $C_j$, hence $\Phi_j \in \mathcal{S}_0(\mathbb{R}^n)$ and thus $%
\Phi_j \ast f$ is defined for all $f \in \mathcal{S}_0^{\prime }(\mathbb{R}%
^n)$.

Multiplying both sides of equality (\ref{1.1}) by $u \in \mathcal{S}(%
\mathbb{R}^n)$ and applying the inverse Fourier transform, we obtain 
\begin{equation*}
u = \Psi \ast u + \sum_{j \in \mathbb{N}}^\infty \Phi_j \ast u, \quad \text{%
for all } u \in \mathcal{S}(\mathbb{R}^n),
\end{equation*}
and we deduce its dual relation 
\begin{eqnarray*}
f &=& \Psi \ast f + \sum_{j \in \mathbb{N}} \Phi_j \ast f, \quad \text{for
all } f \in \mathcal{S}^{\prime }(\mathbb{R}^n), \\
u &=& \sum_{j \in \mathbb{Z}} \Phi_j \ast u, \quad \text{for all } u \in 
\mathcal{S}_0(\mathbb{R}^n), \\
f &=& \sum_{j \in \mathbb{Z}} \Phi_j \ast f, \quad \text{for all } f \in 
\mathcal{S}_0^{\prime }(\mathbb{R}^n).
\end{eqnarray*}

For all $j, k \in \mathbb{Z}$, we define the continuous operators $\Delta_k$%
, $Q_j$ by 
\begin{equation*}
\Delta_k: \mathcal{S}^{\prime }(\mathbb{R}^n) \to \mathcal{S}^{\prime }(%
\mathbb{R}^n) \quad \text{and} \quad Q_j: \mathcal{S}^{\prime }(\mathbb{R}%
^n) \to \mathcal{S}^{\prime }(\mathbb{R}^n),
\end{equation*}
such that 
\begin{eqnarray*}
Q_j f &=& (\psi(2^{-j} \cdot))^\vee \ast f = 2^{jn} \overset{\vee}{\psi}(2^j
\cdot) \ast f, \quad \text{for } j = 0,1,2,... \\
\Delta_k f &=& (\varphi(2^{-k} \cdot))^\vee \ast f = 2^{kn} \overset{\vee}{%
\varphi}(2^k \cdot) \ast f, \quad \text{for } k = 1,2,...
\end{eqnarray*}
We have $\widehat{\Delta_k f}(\xi) = \hat{\Phi}_k(\xi) \hat{f}(\xi)$, hence $%
\Delta_k f(\xi) = \Phi_k(\xi) \ast f(\xi)$, and we set $Q_0 f(\xi) = F(\xi)$.

For all $f \in \mathcal{S}^{\prime }(\mathbb{R}^n)$ and $k \in \mathbb{N}$,
(convergence in $\mathcal{S}^{\prime }(\mathbb{R}^n)$), we have 
\begin{eqnarray*}
Q_k f &=& \sum_{j \in \mathbb{Z}, j \leq k} \Delta_j f, \\
f &=& \sum_{j \in \mathbb{Z}} \Delta_j f = F + \sum_{j=1}^\infty \Delta_j f,
\quad \text{for all } f \in \mathcal{S}_0^{\prime }(\mathbb{R}^n),
\end{eqnarray*}
such that $\widehat{\Delta_j f} \neq 0$ on $A_j$ and $\hat{F}(\xi) \neq 0$
for $|\xi| \leq 1$.

If $\xi = 0$, then the expression $\sum_{j \in \mathbb{Z}} \hat{\Phi}_j(0) =
0$ implies (\ref{1.1}), because if $\phi^{(k)} = 0$ for all $k \in \mathbb{%
N}^n$, then (\ref{1.1}) is verified for all $\phi \in \widehat{\mathcal{S}}%
_0(\mathbb{R}^n)$.

\subsection{\textbf{Decomposition of }$\Delta_k(f \cdot g)$}

We compute 
\begin{eqnarray*}
\Delta_k(f \cdot g) &=& \Delta_k\left( \sum_{j \in \mathbb{N}} \sum_{\ell
\in \mathbb{N}} \Delta_j f \cdot \Delta_\ell g \right) \\
&=& \sum_{j \in \mathbb{N}} \sum_{\ell \in \mathbb{N}} \Delta_k(\Delta_j f
\cdot \Delta_\ell g) \\
&=& \left( \prod_{k,1} + \prod_{k,2} + \prod_{k,3} \right) (f,g),
\end{eqnarray*}
where 
\begin{eqnarray*}
\prod\nolimits_{k,1}(f,g) &=& \Delta_{k(1)}(f \cdot g) = \Delta_k(\widetilde{%
\Delta}_k f \cdot Q_{k+1} g) \\
\prod\nolimits_{k,2}(f,g) &=& \Delta_{k(2)}(f \cdot g) = (Q_{k+1} f \cdot 
\widetilde{\Delta}_k g) = \prod\nolimits_{k,2}(g,f) \\
\prod\nolimits_{k,3}(f,g) &=& \Delta_{k(3)}(f \cdot g) = \sum_{j=k}^\infty
\Delta_k(\Delta_j f \cdot \overline{\Delta}_j g),
\end{eqnarray*}
with $\widetilde{\Delta}_k = \sum_{j=k-2}^{k+4} \Delta_j$, and $\overline{%
\Delta}_k = \sum_{j=k-1}^{k+1} \Delta_j$ (support calculations).

\section{\textbf{Finite difference operators }$\Delta_h^m$}

\begin{definition}
\label{1.4} Let $x, h \in \mathbb{R}^n$, $m \in \mathbb{N}$, and $f$ an
arbitrary function, introduce the finite difference operator $\Delta_h$ such
that 
\begin{eqnarray*}
\Delta_h f &=& \tau_{-h} f - f, \quad \text{where } \tau_h f(x) = f(x - h),
\\
\Delta_h^1 f(x) &=& \Delta_h f(x) = f(x + h) - f(x).
\end{eqnarray*}
The operators $\Delta_h^m f$ are defined by the recurrence relation 
\begin{equation}
\Delta_h^m f(x) = \Delta_h(\Delta_h^{m-1} f(x)), \quad m \geq 2.
\label{[1.2]}
\end{equation}
We deduce that 
\begin{equation*}
\Delta_h^2 f(x) = \Delta_h(\Delta_h^1 f(x)) = f(x + 2h) - 2f(x + h) + f(x).
\end{equation*}
\end{definition}

\begin{lemma}
\label{1.5} Let $x, h \in \mathbb{R}^n$, $m \in \mathbb{N}$, and $f$ an
arbitrary function defined on a subset of $\mathbb{R}^n$, then the general
term of $\Delta_h^m f(x)$ can be written as 
\begin{equation*}
\Delta_h^m f(x) = \sum_{\ell=0}^m \binom{m}{\ell} (-1)^{\ell \pm m}
\tau_{-\ell h} f(x) = \sum_{\ell=0}^m \binom{m}{\ell} (-1)^\ell \tau_{(\ell
- m)h} f(x).
\end{equation*}
\end{lemma}

\begin{proof}\text{ \ }
\vspace{-20pt}	
\begin{itemize}

\item[(i)] Since we have $\binom{m}{\ell-1} + \binom{m}{\ell} = \binom{m+1}{%
\ell}$, and $(-1)^{m-\ell} = (-1)^{m+\ell} = -(-1)^{m+\ell-1} =
-(-1)^{m+\ell+1}$, we deduce by induction that if $\Delta_h^m f(x) =
\sum_{\ell=0}^m \binom{m}{\ell} (-1)^{m+\ell} \tau_{-\ell h} f(x)$, then
according to (\ref{1.2}), 
\begin{eqnarray*}
\Delta_h^{m+1} f(x) &=& \Delta_h(\Delta_h^m f(x)) = \Delta_h^m f(x + h) -
\Delta_h^m f(x) \\
&=& \sum_{k=0}^m \binom{m}{k} (-1)^{m+k} f(x + (k+1)h) - \sum_{k=0}^m \binom{%
m}{k} (-1)^{m+k} f(x + kh) \\
&=& \sum_{\ell=1}^{m+1} \binom{m}{\ell-1} (-1)^{m+\ell-1} f(x + \ell h) -
\sum_{k=0}^m \binom{m}{k} (-1)^{m+k} f(x + kh) \\
&=& f(x + (m+1)h) - (-1)^m f(x) + \sum_{\ell=1}^m \left( \binom{m}{\ell-1}
(-1)^{m+\ell-1} - \binom{m}{\ell} (-1)^{m+\ell} \right) f(x + \ell h) \\
&=& f(x + (m+1)h) - (-1)^m f(x) + \sum_{\ell=1}^m \binom{m+1}{\ell}
(-1)^{m+\ell+1} f(x + \ell h) \\
&=& \sum_{\ell=0}^{m+1} \binom{m+1}{\ell} (-1)^{m+1+\ell} f(x + \ell h).
\end{eqnarray*}

\item[(ii)] From (i) and since $\binom{m}{\ell} = \binom{m}{m-\ell}$, then
we have 
\begin{equation*}
\Delta_h^m f(x) = \sum_{\ell=0}^m \binom{m}{m-\ell} (-1)^{m+\ell} f(x + \ell
h).
\end{equation*}
Since $(-1)^{m-\ell} = (-1)^{m+\ell}$, and for $0 \leq \ell \leq m$, we have 
$0 \leq m-\ell \leq m$, so for $m-\ell = k$, we have 
\begin{equation*}
\Delta_h^m f(x) = \sum_{\ell=0}^m \binom{m}{m-\ell} (-1)^{m-\ell} f(x + \ell
h) = \sum_{k=0}^m \binom{m}{k} (-1)^k f(x + (m-k)h).
\end{equation*}
Thus, we agree to set 
\begin{equation*}
\Delta_h^m f(x) = \sum_{k=0}^m \binom{m}{k} (-1)^k f(x + (m-k)h) =
\sum_{\ell=0}^m \binom{m}{\ell} (-1)^{m \mp \ell} \tau_{-\ell h} f(x).
\end{equation*}
\end{itemize}
\end{proof}

\begin{definition}
\label{1.6} Let $h \in \mathbb{R}^n$, $m \in \mathbb{N}$, $t > 0$, and $f
\in L^p(\mathbb{R}^n)$, we define 
\begin{equation}
\omega_p^m(t, f) = \sup_{|h| \leq t, h \in \mathbb{R}^n} \| \Delta_h^m f
\|_{L^p(\mathbb{R}^n)},  \label{[1.3]}
\end{equation}
as the $m$-th order modulus of continuity of $f$ in $L^p(\mathbb{R}^n)$.
\end{definition}

This modulus of continuity is used to find equivalent norms.

\begin{remark}
\label{1.7}

\begin{description}
\item[a)] The continuity of a function $f$ at $x$ is defined by 
\begin{equation*}
\lim_{h \to 0} |\Delta_h^1 f(x)| = \lim_{h \to 0} |f(x+h) - f(x)| \to 0.
\end{equation*}

\item[b)] The differentiability of a function $f$ at $x$ is described by 
\begin{equation*}
\lim_{h \to 0} \frac{|\Delta_h^1 f(x)|}{|h|} \leq c < \infty, \quad (c > 0).
\end{equation*}

\item[c)] The moduli of continuity defined in (\ref{1.3}) converge in the $%
L^p$ norm, are monotone with respect to the $\sup$ operator, and are $m$%
-differentiable (regularity of order $m$).

\item[d)] $\omega_p^r(t, f) = 0$ if and only if $f$ is a polynomial of
degree $\leq r-1$.

\item[e)] For all $s>0$, we have an equivalent norm for the Holder-Zygmund
space $\mathcal{C}^{s}(\mathbb{R}^{n})$ 
\begin{equation*}
\Vert f\Vert _{\mathcal{C}^{s}(\mathbb{R}^{n})}\sim \sup_{x\in \mathbb{R}%
^{n}}|f(x)|+\sup_{x\in \mathbb{R}^{n},0<|h|\leq 1}|h|^{-s}|\Delta
_{h}^{k}f(x)|,\quad \text{for all }k\in \mathbb{N},k>s.
\end{equation*}
\end{description}
\end{remark}

\section{\textbf{Norms in Besov spaces}}

\begin{quote}
In this paragraph, we give the definitions of norms for Besov spaces, using
the \textbf{Littlewood-Paley} theory and \textbf{finite differences}.
\end{quote}

\subsection{\textbf{By Littlewood-Paley theory}}

\begin{definition}
\label{1.8} Let $s \in \mathbb{R}$, $p, q \in [1, +\infty]$, then we say
that a function $f$ belongs to $\dot{B}_p^{s,q}(\mathbb{R}^n)$, the
homogeneous Besov space, if and only if $f \in \mathcal{S}^{\prime }(\mathbb{%
R}^n) / \mathcal{P}(\mathbb{R}^n)$ and $f = \sum_{j \in \mathbb{Z}} \Delta_j
f$ such that 
\begin{equation*}
\| f \|_{\dot{B}_p^{s,q}(\mathbb{R}^n)} = \left\{ 
\begin{array}{cl}
\displaystyle \left( \sum_{j \in \mathbb{Z}} \left( 2^{sj} \| \Phi_j \ast f
\|_p \right)^q \right)^{1/q} < \infty, & \text{for } 1 \leq q < \infty \\ 
\displaystyle \sup_{j \in \mathbb{Z}} \left( 2^{sj} \| \Phi_j \ast f \|_p
\right) < \infty, & \text{for } q = \infty%
\end{array}
\right.
\end{equation*}
\end{definition}

\begin{proposition}[~\cite{ULLR}]
\label{1.9}  For all $f \in \dot{B}_p^{s,q}(\mathbb{R}^n)$ and
all $\lambda > 0$, there exist two constants $0 < c_1 \leq c_2$, such that 
\begin{equation}
c_1 \| f \|_{\dot{B}_p^{s,q}(\mathbb{R}^n)} \leq \lambda^{n/p - s} \|
f(\lambda \cdot) \|_{\dot{B}_p^{s,q}(\mathbb{R}^n)} \leq c_2 \| f \|_{\dot{B}%
_p^{s,q}(\mathbb{R}^n)}.  \label{[1.4]}
\end{equation}
\end{proposition}

\begin{proof}\text{ \  }
	
\vspace{-20pt}	
\begin{itemize}

\item Let $g(x) = f(\lambda x)$, $\gamma = s - n/p$, $\lambda = 2^N$, $u =
2^N x$, then $\mathrm{d}u = 2^{nN} \mathrm{d}x$, and we have 
\begin{eqnarray*}
\Phi_j \ast g(x) &=& \int_{\mathbb{R}^n} \Phi_j(t) g(x-t) \,\mathrm{d}t =
\int_{\mathbb{R}^n} 2^j \Phi_0(2^j t) f(2^N (x-t)) \,\mathrm{d}t \\
&=& \int_{\mathbb{R}^n} 2^{j-N} \Phi_0(2^{j-N} u) f(2^N x - u) \,\mathrm{d}u
= \int_{\mathbb{R}^n} \Phi_{j-N}(u) f(2^N x - u) \,\mathrm{d}u \\
&=& \Phi_{j-N} \ast f(2^N x).
\end{eqnarray*}
Then 
\begin{eqnarray*}
\| g \|_{\dot{B}_p^{s,q}(\mathbb{R}^n)} &=& \left( \sum_{j \in \mathbb{Z}}
\left( 2^{js} \left( \int_{\mathbb{R}^n} |\Phi_{j-N} \ast f(2^N x)|^p \,%
\mathrm{d}x \right)^{1/p} \right)^q \right)^{1/q} \\
&=& \left( \sum_{j \in \mathbb{Z}} \left( 2^{js} 2^{-nN/p} \left( \int_{%
\mathbb{R}^n} |\Phi_{j-N} \ast f(u)|^p \,\mathrm{d}u \right)^{1/p} \right)^q
\right)^{1/q} \\
&=& (2^N)^{\gamma} \left( \sum_{k = j-N \in \mathbb{Z}} \left( 2^{ks} \|
\Phi_k \ast f \|_p \right)^q \right)^{1/q} = \lambda^\gamma \| f \|_{\dot{B}%
_p^{s,q}(\mathbb{R}^n)}.
\end{eqnarray*}

\item Consider $\lambda > 0$ such that $2^{N-1} < \lambda \leq 2^N$, $(N \in 
\mathbb{Z})$. To prove (\ref{1.4}), we prove one of its inequalities
because if $\| f(\lambda \cdot) \|_{\dot{B}_p^{s,q}(\mathbb{R}^n)} \leq c \|
f \|_{\dot{B}_p^{s,q}(\mathbb{R}^n)}$, then 
\begin{equation*}
\| f \|_{\dot{B}_p^{s,q}(\mathbb{R}^n)} = \| f(\lambda^{-1} (\lambda \cdot))
\|_{\dot{B}_p^{s,q}(\mathbb{R}^n)} \leq c^{\prime }\| f(\lambda \cdot) \|_{%
\dot{B}_p^{s,q}(\mathbb{R}^n)}.
\end{equation*}
We have $f = \sum_{j \in \mathbb{Z}} \Delta_j f$, so $\Delta_k f = \Phi_k
\ast f = \sum_{j \in \mathbb{Z}} \Phi_k \ast \Delta_j f$, but $\text{supp}%
(\Delta_k f) \subset [2^{k-1}, 2^{k+1}]$, hence 
\begin{equation*}
\Delta_k f = \sum_{j \in \mathbb{Z}} \Phi_k \ast \Delta_j f = \Phi_k \ast
\Delta_{k-1} f + \Phi_k \ast \Delta_{k+1} f + \Phi_k \ast \Delta_k f.
\end{equation*}
By induction, if $\Delta_k f$, $\Delta_{k-1} f$, $\Phi_0$ are decreasing,
then $\Phi_k$, $\Phi_k \ast \Delta_{k-1} f$, $\Phi_k \ast \Delta_k f$ are
decreasing, and thus $\Phi_k \ast \Delta_{k+1} f$ is decreasing as well as $%
\Delta_{k+1} f$, whence 
\begin{eqnarray*}
\Delta_k f(\lambda (x-t)) &=& \Phi_k \ast \Delta_{k-1} f + \Phi_k \ast
\Delta_{k+1} f + \Phi_k \ast \Delta_k f (\lambda (x-t)) \\
&\leq& \Phi_k \ast \Delta_{k-1} f + \Phi_k \ast \Delta_{k+1} f + \Phi_k \ast
\Delta_k f (2^{N-1} (x-t)) \\
&=& \Delta_k f (2^{N-1} (x-t)),
\end{eqnarray*}
and $f(\lambda (x-t)) = \sum_{j \in \mathbb{Z}} \Delta_j f (\lambda (x-t))
\leq \sum_{j \in \mathbb{Z}} \Delta_j f (2^{N-1} (x-t)) = f(2^{N-1} (x-t))$.

We deduce that 
\begin{eqnarray*}
f(2^N (x-t)) &\leq& f(\lambda (x-t)) \leq f(2^{N-1} (x-t)), \\
\Phi_j \ast f(2^{N-1} \cdot) &\leq& \Phi_j \ast f(\lambda \cdot) \leq \Phi_j
\ast f(2^N \cdot), \\
|\Phi_j \ast f(\lambda \cdot)| &\leq& \max \left( |\Phi_j \ast f(2^N
\cdot)|, |\Phi_j \ast f(2^{N-1} \cdot)| \right), \\
\| \Phi_j \ast f(\lambda \cdot) \|_p &\leq& \max \left( \| \Phi_j \ast f(2^N
\cdot) \|_p, \| \Phi_j \ast f(2^{N-1} \cdot) \|_p \right),
\end{eqnarray*}
hence 
\begin{eqnarray*}
\| g \|_{\dot{B}_p^{s,q}(\mathbb{R}^n)} &\leq& \max \left( \| f(2^N \cdot)
\|_{\dot{B}_p^{s,q}(\mathbb{R}^n)}, \| f(2^{N-1} \cdot) \|_{\dot{B}_p^{s,q}(%
\mathbb{R}^n)} \right) \\
&\leq& 2^\gamma (2^{N-1})^\gamma \| f \|_{\dot{B}_p^{s,q}(\mathbb{R}^n)}
\leq C \lambda^\gamma \| f \|_{\dot{B}_p^{s,q}(\mathbb{R}^n)}, \quad C =
2^\gamma.
\end{eqnarray*}
Thus, we can write $\| \cdot \|_{\dot{B}_p^{s,q}(\mathbb{R}^n)} \sim
\lambda^{n/p - s} \| (\lambda \cdot) \|_{\dot{B}_p^{s,q}(\mathbb{R}^n)}$.

If $\lambda = 2^m$, then $\| f(2^m \cdot) \|_{\dot{B}_p^{s,q}(\mathbb{R}^n)}
= (2^m)^{s - n/p} \| f \|_{\dot{B}_p^{s,q}(\mathbb{R}^n)}$.

If $q = 2$, then $\| f(\lambda \cdot) \|_{\dot{B}_p^{s,2}(\mathbb{R}^n)} =
\| f(\lambda \cdot) \|_{\dot{H}^{s,p}(\mathbb{R}^n)} = \lambda^{s - n/p} \|
f \|_{\dot{H}^{s,p}(\mathbb{R}^n)} = \lambda^{s - n/p} \| f \|_{\dot{B}%
_p^{s,2}(\mathbb{R}^n)}$.

We also have $\| f(\lambda \cdot) \|_{L^p(\mathbb{R}^n)} = \lambda^{-n/p} \|
f \|_{L^p(\mathbb{R}^n)}$.
\end{itemize}
\end{proof}

\begin{definition}
\label{1.10} Let $s \in \mathbb{R}$, $p, q \in [1, +\infty]$, then we say
that a function $f$ belongs to $B_p^{s,q}(\mathbb{R}^n)$, the
non-homogeneous Besov space, if and only if $f \in \mathcal{S}^{\prime }(%
\mathbb{R}^n)$ and $f = Q_0 f + \sum_{j \geq 1} \Delta_j f$, such that 
\begin{equation*}
\| f \|_{B_p^{s,q}(\mathbb{R}^n)} = \left\{ 
\begin{array}{cl}
\displaystyle \| \Psi \ast f \|_p + \left( \sum_{j \in \mathbb{N}} \left(
2^{sj} \| \Phi_j \ast f \|_p \right)^q \right)^{1/q} < \infty, & \text{for }
1 \leq q < \infty \\ 
\| \Psi \ast f \|_p + \sup_{j \in \mathbb{N}} \left( 2^{sj} \| \Phi_j \ast f
\|_p \right) < \infty, & \text{for } q = \infty%
\end{array}
\right.
\end{equation*}
\end{definition}

\subsection{\textbf{By finite differences}}

\begin{lemma}[~\cite{ULLR}]
\label{1.11}  For all $0 < s < 1$, we have 
\begin{equation*}
\sum_{j \in \mathbb{Z}} 2^{js} |\Phi_j(x)| \leq c x^{-(1+s)}, \quad \text{%
for all } x > 0, \quad (c > 0).
\end{equation*}
\end{lemma}

\begin{proof}\text{  }
\vspace{-20pt}	
\begin{itemize}
\item If $x = 2^{-N}$, then $\Phi_0 \in \mathcal{S}$ implies $|\Phi_0(x)|
\leq c x^{-2}$, whence 
\begin{equation*}
\sum_{j=N+1}^\infty 2^{js} |\Phi_j(x)| = \sum_{j=N+1}^\infty 2^{j(s+1)}
|\Phi_0(2^j x)| \leq \sum_{n=N+1}^\infty 2^{n(\alpha-1)} x^{-2} = \frac{1}{%
1-2^{1-\alpha}} \cdot 2^{N(\alpha-1)} x^{-2} = c x^{-(1+s)}.
\end{equation*}
Since $\Phi_0$ is bounded, then 
\begin{equation*}
\sum_{j=-\infty}^N 2^{js} |\Phi_j(x)| \leq c \sum_{j=-\infty}^N 2^{j(1+s)} =
c 2^{N(1+s)} = c x^{-(1+s)}.
\end{equation*}

\item If $2^{-N} < x \leq 2^{-N+1}$, $N \in \mathbb{Z}$, and by the decrease
of $\Phi_j$, we have 
\begin{eqnarray*}
\sum_{j \in \mathbb{Z}} 2^{js} |\Phi_j(x)| &\leq& \sum_{j \in \mathbb{Z}}
2^{js} |\Phi_j(2^{-N})| \\
&\leq& c 2^{N(1+s)} = (2^{1+s} c) 2^{(N-1)(1+s)} \leq c^{\prime -(1+s)}.
\end{eqnarray*}
\end{itemize}
\end{proof}

\begin{lemma}[~\cite{ULLR}]
\label{1.12}  Let $1 \leq p \leq \infty$, $k \in \mathbb{N}$, $f
\in \mathcal{S}(\mathbb{R}^n)$, then

\begin{itemize}
\item[(i)] $\| \tau_{-h} f - f \|_p \leq |h| \| f^{\prime }\|_p$

\item[(ii)] $\| \Delta_h^k f \|_p \leq |h|^k \| f^{(k)} \|_p, \quad (k \geq
2)$
\end{itemize}
\end{lemma}

\begin{proof}\text{ }
	
\vspace{-20pt}	
	
\begin{itemize}

\item[(i)] Using the inequality $|f(x) - f(y)| \leq |x-y|^{1-1/p} \|
f^{\prime }\|_p$, and since $p \geq 1$ gives $0 \leq 1-1/p \leq 1$, taking $%
h = x - y$, we obtain 
\begin{equation*}
\| \tau_{-h} f - f \|_p \leq |h|^{1-1/p} \| f^{\prime }\|_p \leq |h| \|
f^{\prime }\|_p.
\end{equation*}

\item[(ii)] We obtain by induction 
\begin{eqnarray*}
\| \Delta_h^{k+1} f \|_p &=& \| \Delta_h(\Delta_h^k f) \|_p \leq |h| \|
(\Delta_h^k f)^{\prime }\|_p = |h| \| \Delta_h^k f^{\prime }\|_p \\
&\leq& |h| |h|^k \| (f^{\prime (k)} \|_p \leq |h|^{k+1} \| f^{(k+1)} \|_p.
\end{eqnarray*}
\end{itemize}
\end{proof}

\begin{proposition}[~\cite{ULLR}]
\label{1.13}  Let $h \in \mathbb{R}^n$, $0 < s < 1$, $1 \leq p
\leq \infty$, $1 \leq q < \infty$, then $f \in B_p^{s,q}(\mathbb{R}^n)$ if
and only if $f \in L^p(\mathbb{R}^n)$, such that 
\begin{equation}
\int_{\mathbb{R}^n} \left( |h|^{-s} \| \tau_{-h} f - f \|_p \right)^q \frac{%
\mathrm{d}h}{|h|^n} < \infty,  \label{[1.5]}
\end{equation}
and $f \in \dot{B}_p^{s,q}(\mathbb{R}^n)$ if and only if $f \in \mathcal{S}%
^{\prime }(\mathbb{R}^n) / \mathcal{P}(\mathbb{R}^n)$, such that (\ref{1.5}%
) is verified.
\end{proposition}

\begin{proof}
First, take the simplest case $q = n = 1$, because the case $q = \infty$, as
well as the general case of $\mathbb{R}^n$, can easily be deduced using the
interpolation theorem of two Besov spaces which gives another Besov space,
as well as the Marcinkiewicz theorem.

\begin{itemize}
\item[(1)] Let $f = F + \sum_{j \in \mathbb{N}} \Delta_j f \in B_p^{s,1}(%
\mathbb{R})$, then $f \in L^p$, such that $\sum_{j \in \mathbb{N}} \|
\Delta_j f \|_p < \infty$, and thus by Lemma \ref{1.12}, we have 
\begin{eqnarray*}
\int_0^{2^{-j}} |h|^{-s} \| \Delta_h^1(\Delta_j f) \|_p \frac{\mathrm{d}h}{%
|h|} &\leq& \int_0^{2^{-j}} |h|^{1-s} \| (\Delta_j f)^{\prime }\|_p \frac{%
\mathrm{d}h}{|h|} \\
&\leq& c 2^j \| (\Delta_j f)^{\prime }\|_p \int_0^{2^{-j}} |h|^{-s} \mathrm{d%
}h = c 2^{js} \| (\Delta_j f)^{\prime }\|_p,
\end{eqnarray*}
and 
\begin{eqnarray*}
\int_{2^{-j}}^\infty |h|^{-s} \| \Delta_h^1(\Delta_j f) \|_p \frac{\mathrm{d}%
h}{|h|} &\leq& 2 \| (\Delta_j f)^{\prime }\|_p \int_{2^{-j}}^\infty |h|^{-s} 
\frac{\mathrm{d}h}{|h|} \leq c 2^{js} \| (\Delta_j f)^{\prime }\|_p,
\end{eqnarray*}
hence 
\begin{equation*}
\int_0^\infty |h|^{-s} \| \Delta_h^1(f) \|_p \frac{\mathrm{d}h}{|h|} =
\int_0^\infty |h|^{-s} \| \Delta_h^1\left( \sum_{j \in \mathbb{N}} \Delta_j
f \right) \|_p \frac{\mathrm{d}h}{|h|} \leq c \sum_{j \in \mathbb{N}} 2^{js}
\| (\Delta_j f)^{\prime }\|_p < \infty.
\end{equation*}

\item[(2)] Let $f \in L^p$, and $\int_0^\infty |h|^{-s} \| \Delta_h^1(f)
\|_p \frac{\mathrm{d}h}{|h|} < \infty$, then $f \ast \Psi \in L^p$.

We have $\Phi_j \ast f(x) = \int_{\mathbb{R}} \Phi_j(y) (f(x-y) - f(x)) \,%
\mathrm{d}y$, hence by Minkowski's inequality 
\begin{eqnarray*}
\| \Phi_j \ast f \|_p &=& \left\| \int_{-\infty}^0 \Phi_j(y) (f(x-y) - f(x))
\,\mathrm{d}y + \int_0^\infty \Phi_j(y) (f(x-y) - f(x)) \,\mathrm{d}y
\right\|_p \\
&\leq& 2 \int_0^\infty |\Phi_j(y)| \| \tau_{-y} f - f \|_p \,\mathrm{d}y.
\end{eqnarray*}
Using Lemma \ref{1.11}, we obtain $f \in B_p^{s,1}(\mathbb{R})$, because 
\begin{eqnarray*}
\sum_{j \in \mathbb{N}} 2^{js} \| \Phi_j \ast f \|_p &\leq& 2 \int_0^\infty
\sum_{j \in \mathbb{N}} 2^{js} |\Phi_j(y)| \| \tau_{-y} f - f \|_p \,\mathrm{%
d}y \\
&\leq& c \int_0^\infty |y|^{-(1+s)} \| \tau_{-y} f - f \|_p \,\mathrm{d}y =
c \int_0^\infty |h|^{-s} \| \Delta_h^1(f) \|_p \frac{\mathrm{d}h}{|h|} <
\infty.
\end{eqnarray*}
Note that (\ref{1.5}) is equivalent to $\int_{|h| \leq 1} \left( |h|^{-s}
\| \tau_{-h} f - f \|_p \right)^q \frac{\mathrm{d}h}{|h|^n}$, because 
\begin{equation*}
\int_{\mathbb{R}^n} \left( |h|^{-s} \| \tau_{-h} f - f \|_p \right)^q \frac{%
\mathrm{d}h}{|h|^n} = \int_{|h| < 1} (...) \,\mathrm{d}h + \int_{|h| \geq 1}
(...) \,\mathrm{d}h,
\end{equation*}
and the integral $\int_{|h| \geq 1} \left( |h|^{-s} \| \tau_{-h} f - f \|_p
\right)^q \frac{\mathrm{d}h}{|h|^n}$ is transformed by a change of variable
into $$\int_{|h| \leq 1} \left( |h|^{-s} \| \tau_{-h} f - f \|_p \right)^q 
\frac{\mathrm{d}h}{|h|^n}.$$
\end{itemize}
\end{proof}

\begin{theorem}[~\cite{SIC}]
\label{1.14}  Let $h \in \mathbb{R}^n$, $s > 0$, $1 \leq q \leq
\infty$, $1 \leq p < \infty$, $M \in \mathbb{N}$, $M \leq s < M+1$, then $f
\in \dot{B}_p^{s,q}(\mathbb{R}^n)$ if and only if $f \in \mathcal{S}^{\prime
}(\mathbb{R}^n) / \mathcal{P}(\mathbb{R}^n)$, and such that 
\begin{equation*}
\| f \|_{\dot{B}_p^{s,q}(\mathbb{R}^n)} \sim \left( \int_{\mathbb{R}^n}
|h|^{-sq} \| \Delta_h^{M+1} f \|_p^q \frac{\mathrm{d}h}{|h|^n} \right)^{1/q}
< \infty.
\end{equation*}
\end{theorem}

\begin{quote}
The term $\int_{\mathbb{R}^n} ... \mathrm{d}h$ can be replaced by $\int_{|h|
< \varepsilon} ... \mathrm{d}h$, for any $\varepsilon > 0$, and in general
we take for $\varepsilon = 1$, the term $\int_{|h| < 1} ... \mathrm{d}h$ in
the sense of equivalent semi-norms.
\end{quote}

\begin{proof}
As in the proof of Proposition \ref{1.13}, suppose $f \in B_p^{s,1}(%
\mathbb{R})$, $0 < s < 1$, thus $f \in L^p$ and $$\int_0^\infty |h|^{-s} \|
\Delta_h^1 f \|_p \frac{\mathrm{d}h}{|h|} < \infty.$$

Since $\Psi \ast f \in L^p$, then for all $j \in \mathbb{N}$, we have $\|
\Phi_j \ast f \|_p \leq c \| \Delta_h^1 f \|_p$.

By Lemma \ref{1.12}, and for $2^{-(j+1)} \leq h \leq 2^{-j}$, we obtain 
\begin{eqnarray*}
\int_{2^{-(j+1)}}^{2^{-j}} |h|^{-s} \| \Delta_h^1 f \|_p \frac{\mathrm{d}h}{%
|h|} &\geq& c \| \Phi_j \ast f \|_p \int_{2^{-(j+1)}}^{2^{-j}} |h|^{-s} 
\frac{\mathrm{d}h}{|h|} = c 2^{js} \| \Phi_j \ast f \|_p,
\end{eqnarray*}
hence 
\begin{eqnarray*}
\sum_{j=0}^\infty 2^{js} \| \Phi_j \ast f \|_p &\leq& c \sum_{j=0}^\infty
\int_{2^{-(j+1)}}^{2^{-j}} |h|^{-s} \| \Delta_h^1 f \|_p \frac{\mathrm{d}h}{%
|h|} 
\leq c \int_0^1 |h|^{-s} \| \Delta_h^1 f \|_p \frac{\mathrm{d}h}{|h|} <
\infty.
\end{eqnarray*}
-- If $s \geq 1$, then $M \leq s < M+1$, whence $0 \leq s-M < 1$, and thus 
\begin{eqnarray*}
\| f \|_{\dot{B}_p^{s,q}(\mathbb{R})} &\sim& \int_0^\infty |h|^{-(s-M)} \|
\Delta_h^1 f \|_p \frac{\mathrm{d}h}{|h|} 
\sim \int_0^\infty |h|^{-s} \| \Delta_h^{M+1} f \|_p \frac{\mathrm{d}h}{|h|%
}.
\end{eqnarray*}
\end{proof}

\begin{theorem}[~\cite{SIC} ]
\label{1.15} Let $h \in \mathbb{R}^n$, $0 < q \leq \infty$, $0 <
p < \infty$, $M \in \mathbb{N}$, and denote $\sigma_p$ by 
\begin{equation*}
\sigma_p = \max \left( 0, \frac{n}{p} - n \right).
\end{equation*}

\begin{itemize}
\item[(i)] Suppose $\sigma_p < s < M$, then 
\begin{equation*}
\| f \|_{\dot{B}_p^{s,q}(\mathbb{R}^n)} \sim \left( \int_{\mathbb{R}^n}
|h|^{-sq} \| \Delta_h^M f \|_p^q \frac{\mathrm{d}h}{|h|^n} \right)^{1/q}.
\end{equation*}

\item[(ii)] If $0 < q \leq \infty$, $0 < p < \infty$, $s > \sigma_p$, $M
\leq s < M+1$, then 
\begin{equation*}
\| f \|_{\dot{B}_p^{s,q}(\mathbb{R}^n)} \sim \left( \int_{\mathbb{R}^n}
|h|^{-sq} \| \Delta_h^{M+1} f \|_p^q \frac{\mathrm{d}h}{|h|^n} \right)^{1/q}.
\end{equation*}
\end{itemize}
\end{theorem}

\begin{proof}
For the proof, see for example ~\cite{Trie}, paragraph 3.5.3.
\end{proof}

\section{\textbf{Functions of bounded p-variation}}

\begin{quote}
Throughout this paragraph, $I$ denotes an interval of $\mathbb{R}$.
\end{quote}

\subsection{\textbf{General notions}}

\begin{quote}
In this subsection, we give the definitions of functions of bounded
p-variation $\mathcal{V}_p(I)$ and the spaces $\mathcal{BV}_p(I)$, $BV_p(I)$%
, of their equivalence classes with the equivalence relation of equality
almost everywhere, as well as the space of their primitives $BV_p^1(I)$, and
their properties that we will use later.
\end{quote}

\begin{definition}
\label{1.16} Let $p \in [1, +\infty[$, then a function $f: I \to \mathbb{R}$
is said to be of bounded p-variation, or briefly p-v b, if for all finite
strict real sequences $t_0 < t_1 < ... < t_N$ of $I$, there exists $c > 0$,
such that 
\begin{eqnarray*}
\sum_{k=1}^N | f(t_k) - f(t_{k-1}) |^p &\leq& c^p, \quad \text{or
equivalently} \\
\sup_{\{t_k\} \subset I} \left[ \sum_{k=1}^N | f(t_k) - f(t_{k-1}) |^p %
\right] &<& \infty.
\end{eqnarray*}
We denote the set of these functions by $\mathcal{V}_p(I)$, ($\mathcal{V}_p$
if $I = \mathbb{R}$), and the minimum of such constants $c$ with respect to $%
f$ by $\nu_p(f, I)$, ($\nu_p(f)$ if $I = \mathbb{R}$) 
\begin{equation*}
\nu_p(f, I) = \inf_c \left\{ c > 0 : \sum_{k=1}^N | f(t_k) - f(t_{k-1}) |^p
\leq c^p, \{t_i\}_{1 \leq i \leq n} \subset I \right\}.
\end{equation*}
\end{definition}

Definition \ref{1.16} is equivalent to the fact that for any family of
disjoint intervals $I_k = [a_k, b_k] \subset I$, we have 
\begin{equation*}
\left( \sum_{I_k} | f(a_k) - f(b_k) |^p \right)^{1/p} \leq c < \infty, \quad
(c > 0).
\end{equation*}
The space $\mathcal{V}_1(I)$, or simply $\mathcal{V}(I)$, is called the
space of functions of bounded variation on $I$, and $\mathcal{V}_\infty(I)$
is a Banach space for the norm 
\begin{equation*}
\| f \|_{\mathcal{V}_\infty(I)} = \nu_\infty(f, I) = \sup_{x \in I} |f(x)|.
\end{equation*}

\begin{definition}
\label{1.17} A function $f: A \to \mathbb{R}^n$, $A \subset \mathbb{R}^n$,
is said to be $\gamma$-Lipschitz, $\gamma \geq 0$, of order $\alpha$, $0 <
\alpha \leq 1$, if and only if 
\begin{equation*}
| f(x) - f(y) | \leq \gamma |x - y|^\alpha, \quad \text{for all } x, y \in A.
\end{equation*}
\end{definition}

The set of these functions is denoted by $\text{Lip}_\alpha(A)$, ($\text{Lip}%
_\alpha$ if $A = \mathbb{R}^n$). We also say that a function is Lipschitz if
it is $\gamma$-Lipschitz for some $\gamma \geq 0$, and we equip $\text{Lip}%
_\alpha(A)$ with the following norm 
\begin{equation*}
\| f \|_{\text{Lip}_\alpha(A)} = \sup_{x,y \in A, x \neq y} \frac{| f(x) -
f(y) |}{|x - y|^\alpha}.
\end{equation*}
A function $f: A \to \mathbb{R}^n$ is said to be locally $\gamma$-Lipschitz
if at every point in $A$ there exists a neighborhood where $f$ is $\gamma$%
-Lipschitz. Note that 
\begin{equation*}
B_\infty^{s,\infty}(\mathbb{R}) = \mathcal{C}^s(\mathbb{R}) = \text{Lip}_s(%
\mathbb{R}), \quad \text{if } 0 < s < 1.
\end{equation*}

\begin{proposition}[~\cite{Bo-Cr-Si-1}]
\label{1.18}  For all $x, y \in I$ and $p \in [1, +\infty[$%
, each element of $\mathcal{V}_p(I)$ is a bounded function; moreover, $%
\mathcal{V}_p(I)$ becomes a Banach space if equipped with the norm 
\begin{equation}
\| f \|_{\mathcal{V}_p(I)} = \sup_{x \in I} |f(x)| + \nu_p(f, I).
\label{[1.6]}
\end{equation}
\end{proposition}

\begin{proof}
By a sequence with only two terms, we obtain $| f(t_k) - f(t_{k-1}) |^p \leq
(\nu_p(f, I))^p$, for all $t_k, t_{k-1} \in I$.

If we take $t_{k-1} = 0$ and $t_k = x$, then $| f(x) - f(0) | \leq \nu_p(f,
I)$, hence 
\begin{eqnarray*}
| f(x) | &=& | f(x) - f(0) + f(0) | \\
&\leq& | f(x) - f(0) | + | f(0) | \\
&\leq& \nu_p(f, I) + | f(0) | = C < \infty.
\end{eqnarray*}
Thus, each function in $\mathcal{V}_p(I)$ is bounded, and the norm (\ref%
{[1.6]}) satisfies all conditions making $\mathcal{V}_p(I)$ a Banach space.
\end{proof}

\subsection{\textbf{Functions of bounded p-variation as distributions}}

\begin{definition}[~\cite{Bo-Cr-Si-1} ]
\label{1.19} Let $p \in [1, +\infty]$, we denote by $%
\mathcal{BV}_p(I)$ the set of functions $f: \mathbb{R} \to \mathbb{R}$ such
that there exists a function $g \in \mathcal{V}_p(I)$ that coincides with $f$
almost everywhere, 
\begin{eqnarray*}
\mathcal{BV}_p(I) &=& \left\{ f: I \to \mathbb{R} ; \exists g \in \mathcal{V}%
_p(I), \text{ such that } f = g \ (a.e.) \right\}, \\
\varepsilon_p(f, I) &=& \inf \left\{ \nu_p(g, I) ; g \in \mathcal{V}_p(I), 
\text{ such that } g = f \ (a.e.) \right\}.
\end{eqnarray*}
\end{definition}

\begin{description}
\item[--] We denote by $BV_p(I)$ the quotient set with respect to the
equivalence relation "equality in $\mathcal{BV}_p(I)$ almost everywhere",
such that 
\begin{eqnarray*}
\dot{f} &=& \left\{ g \in \mathcal{BV}_p(I) ; g = f \ (a.e.) \right\}, \\
BV_p(I) &=& \left\{ \dot{f} ; f \in \mathcal{BV}_p(I) \right\} = \mathcal{BV}%
_p(I) / a.e.
\end{eqnarray*}

\item[--] If $h \in BV_p(I)$, we denote by $\varepsilon_p(h, I)$ the number $%
\varepsilon_p(f, I)$, for any representative $f$ of $h$.
\end{description}

\begin{definition}
\label{1.20} Let $f: I \to \mathbb{R}$ be a function having discontinuities
only of the first kind, then we say that $f$ is normalized if 
\begin{eqnarray*}
f(x) &=& \frac{1}{2} (f(x^+) + f(x^-)), \quad \text{for all } x \in 
\mathring{I}, \\
f(x) &=& \lim_{y \to x, y \in \mathring{I}} f(y), \quad \text{for all } y
\in I \cap \partial I,
\end{eqnarray*}
where $\partial I = \overline{I} \setminus \mathring{I}$ is the boundary of $%
I$, $\overline{I}$ is the closure, and $\mathring{I}$ is the interior, such
that 
\begin{equation*}
f(x^+) = \lim_{h > 0, h \to 0} f(x+h), \quad \text{and} \quad f(x^-) =
\lim_{h > 0, h \to 0} f(x-h).
\end{equation*}
\end{definition}

\begin{proposition}[~\cite{Bo-Cr-Si-1} ]
\label{1.21} If $f$ is a function in $\mathcal{V}_p(I)$,
then the function $\tilde{f}$ defined by 
\begin{eqnarray*}
\tilde{f}(x) &=& \frac{1}{2} (f(x^+) + f(x^-)), \quad \text{for all } x \in 
\mathring{I}, \\
\tilde{f}(x) &=& \lim_{y \to x, y \in \mathring{I}} f(y), \quad \text{for
all } x \in I \cap \partial I,
\end{eqnarray*}
is normalized, and belongs to $\mathcal{V}_p(I)$, and satisfies the
following inequalities 
\begin{equation*}
\nu_p(\tilde{f}, I) \leq \nu_p(f, I), \quad \text{and} \quad \sup_I |\tilde{f%
}| \leq \sup_I |f|.
\end{equation*}
\end{proposition}

\begin{proof}
For the proof, see ~\cite{Bo-Cr-Si-1}, ~\cite{Bo-Cr-Si-2}.
\end{proof}

\begin{proposition}[~\cite{Bo-Cr-Si-1}]
\label{1.22}  Let $p \in [1, +\infty]$, and $f \in BV_p(%
\mathbb{R})$, then $f$ has a unique normalized representative $\tilde{f} \in 
\mathcal{V}_p$, such that $\varepsilon_p(f) = \nu_p(\tilde{f})$.
\end{proposition}

We thus consider the space $BV_p(\mathbb{R})$ as a Banach space of
distributions, equipped with the following norm 
\begin{eqnarray*}
\| f \|_{BV_p(\mathbb{R})} &=& \varepsilon_p(f) + \| f \|_\infty = \nu_p(%
\tilde{f}) + \sup_{x \in \mathbb{R}} |\tilde{f}(x)|, \quad \text{if } p <
\infty, \\
\| f \|_{BV_\infty(\mathbb{R})} &=& \| f \|_\infty = \sup_{x \in \mathbb{R}}
|\tilde{f}(x)|.
\end{eqnarray*}

\begin{definition}
\label{1.23} Let $p \in [1, +\infty]$, then a function $f: I \to \mathbb{R}$
belongs to $BV_p^1(I)$ if there exist $\alpha, x_0 \in \mathbb{R}$ and $g
\in BV_p(I)$, such that for all $x \in I$, we have 
\begin{equation}
f(x) = \alpha + \int_{x_0}^x g(t) \,\mathrm{d}t.  \label{[1.7]}
\end{equation}
\end{definition}

If (\ref{1.7}) is verified, then $f$ is a Lipschitz continuous function,
and we equip $BV_p^1(I)$ with the norm $\| f \|_{BV_p^1(I)} = |f(x_0)| + \|
f^{\prime }\|_{BV_p(I)}$, for which $BV_p^1(I)$ becomes a Banach space, and
for each point $x_0 \in I$, we obtain an equivalent norm.

\begin{proposition}[~\cite{Bo-Cr-Si-2}]
\label{1.24}  For any interval $I$ of $\mathbb{R}$, the
space $\mathcal{V}_p(I)$, $p \geq 1$, is a Banach algebra for pointwise
multiplication of functions, such that 
\begin{equation*}
\| f g \|_{\mathcal{V}_p(I)} \leq \| f \|_{\mathcal{V}_p(I)} \| g \|_{%
\mathcal{V}_p(I)}, \quad \text{for all } f, g \in \mathcal{V}_p(I).
\end{equation*}
\end{proposition}

\begin{proof}
Let $x_0 < x_1 < ... < x_n$ be a finite sequence in $I$ and $f, g \in 
\mathcal{V}_p(I)$. Since $p \geq 1$, and using Minkowski's inequality, we
obtain 
\begin{eqnarray*}
\left( \sum_{j=1}^N | f g(x_j) - f g(x_{j-1}) |^p \right)^{1/p} &\leq&
\left( \sum_{j=1}^N | f(x_j) (g(x_j) - g(x_{j-1})) |^p \right)^{1/p} \\
&+& \left( \sum_{j=1}^N | g(x_{j-1}) (f(x_j) - f(x_{j-1})) |^p \right)^{1/p}
\\
&\leq& \sup_I |f| \left( \sum_{j=1}^N | g(x_j) - g(x_{j-1}) |^p
\right)^{1/p} + \sup_I |g| \left( \sum_{j=1}^N | f(x_j) - f(x_{j-1}) |^p
\right)^{1/p} \\
&\leq& \sup_I |f| \nu_p(g, I) + \sup_I |g| \nu_p(f, I),
\end{eqnarray*}
hence 
\begin{eqnarray*}
\nu_p(f g, I) &=& \sup_I \left( \sum_{j=1}^N | f g(x_j) - f g(x_{j-1}) |^p
\right)^{1/p} \leq \sup_I |f| \nu_p(g, I) + \sup_I |g| \nu_p(f, I),
\end{eqnarray*}
and thus 
\begin{eqnarray*}
\| f g \|_{\mathcal{V}_p(I)} &=& \nu_p(f g, I) + \sup_I |f| \sup_I |g| \\
&\leq& \sup_I |f| \nu_p(g, I) + \sup_I |g| \nu_p(f, I) + \sup_I |f| \sup_I
|g| \\
&\leq& \sup_I |f| \nu_p(g, I) + \sup_I |g| \nu_p(f, I) + \sup_I |f| + \sup_I
|g| + \nu_p(g, I) \nu_p(f, I) \\
&=& \left( \sup_I |f| + \nu_p(f, I) \right) \left( \sup_I |g| + \nu_p(g, I)
\right) = \| f \|_{\mathcal{V}_p(I)} \| g \|_{\mathcal{V}_p(I)}.
\end{eqnarray*}
Therefore, $\| f g \|_{\mathcal{V}_p(I)} \leq \| f \|_{\mathcal{V}_p(I)} \|
g \|_{\mathcal{V}_p(I)}$, for all $f, g \in \mathcal{V}_p(I)$.
\end{proof}

\chapter{\textbf{Composition of Operators in }$BV_p^1(I)$}

In this chapter, we give a reminder of some basic notions concerning
composition operators, then we present the properties of the spaces $%
BV_p^1(I)$, where $I$ is an interval of $\mathbb{R}$, then we present some
fundamental results of functional calculus on these spaces, then we present
two fundamental Theorems \ref{2.9} and \ref{2.13}, due to the works of 
\textbf{~\cite{Bo-Cr-Si-1}}.

Finally, we give our contribution, namely Theorem \ref{2.15} which is a
generalization of a fundamental inequality (Theorem \ref{2.9}), based on
Lemma \ref{2.14} which is an algorithm giving the derivative of the
composition of $n$ functions.

\section{\textbf{Reminder}}

\begin{definition}
\label{2.1} Let $E$ be a functional space and let $\phi$ be a real-valued
function, we define the composition operator $T_\phi$ associated with $\phi$
by 
\begin{equation*}
T_\phi(f) = \phi \circ f, \quad \text{for all } f \in E.
\end{equation*}
\end{definition}

In general, $T_\phi$ is nonlinear, and the \textbf{P}roblem of \textbf{S}%
uperposition \textbf{O}perators (\textbf{P.S.O}) for $E$ consists in finding
the set $S(E)$ of real-valued functions $\phi$ such that $T_\phi(E)
\subseteq E$

\begin{center}
$S(E) = \left\{ \phi: \mathbb{R} \to \mathbb{R} ; T_\phi(E) \subseteq E
\right\}$.
\end{center}

If $T_\phi(E) \subseteq E$, then we say that the composition operator $$
T_\phi: E \to E$$ operates on the functional space $E$.

\begin{remark}[~\cite{Bo-Cr-Si-1} ]
\label{2.2} We agree to say that a superposition operator 
$T_f: E \to E$ satisfies the norm composition inequality property for a
normed space $E$ if it verifies 
\begin{equation*}
\| T_f(g) \|_E \leq c_f (1 + \| g \|_E), \quad (c_f > 0), \quad \text{for
all } g \in E.
\end{equation*}
Note that if a composition operator satisfies the norm inequality property
in the normed space $E$, this implies that it operates on $E$.
\end{remark}

\section{\textbf{Properties of the spaces }$BV_p^1(I)$}

\begin{quote}
In all that follows, $I$ denotes an interval of $\mathbb{R}$.\\ All results in
this paragraph are due to ~\cite{Bo-Cr-Si-1} and ~\cite{Bo-Cr-Si-2}.
\end{quote}

\begin{proposition}[~\cite{Bo-Cr-Si-2}]
\label{2.3}  If $f \in BV_p^1(I)$, then $f$ can be
extended uniquely to a function in $BV_p^1(\overline{I})$ with the same norm
as $f$ in $BV_p^1(I)$.
\end{proposition}

\begin{proof}
For the proof, see ~\cite{Bo-Cr-Si-2}.
\end{proof}

By the above proposition, and exploiting affine transformations, the study
of the spaces $BV_p^1(I)$ can be reduced to the following three cases: 
\begin{equation*}
I = \mathbb{R}, \quad I = [0, +\infty[, \quad I = [0, 1].
\end{equation*}

\begin{proposition}
\label{2.4} The homogeneity property 
\begin{equation*}
\| f(\lambda \cdot)^{\prime }\|_{BV_p} = \lambda \| f^{\prime }\|_{BV_p},
\quad \text{for all } \lambda > 0
\end{equation*}
is verified for each function $f$ in $BV_p^1(\mathbb{R})$ or $BV_p^1([0,
+\infty[)$.
\end{proposition}

Thus, the spaces $BV_p^1(\mathbb{R})$ and $BV_p^1([0, +\infty[)$ can be seen
as analogous to the homogeneous Sobolev space 
\begin{equation*}
\dot{W}^{1,p}(\mathbb{R}) = \{ f : f^{\prime p}(\mathbb{R}) \},
\end{equation*}
equipped with the semi-norm $\| f^{\prime }\|_p$, while the usual Sobolev
space $W^{1,p}(\mathbb{R})$ is non-homogeneous: $$W^{1,p}(\mathbb{R}) = \dot{W%
}^{1,p}(\mathbb{R}) \cap L^p(\mathbb{R}).$$

\begin{proposition}[~\cite{Bo-Cr-Si-2}]\text{ \ }
\label{2.5} 
\vspace{-20pt}
\begin{itemize}
\item[(i)] $BV_p^1(I) \cap L^p(I) = BV_p^1(I)$ if and only if $I$ is bounded.

\item[(ii)] $BV_p^1(I) \cap L^p(I)$ injects continuously into $BV_p(I)$.
\end{itemize}
\end{proposition}

\begin{proof}
For the proof, see ~\cite{Bo-Cr-Si-2}.
\end{proof}

\begin{theorem}[ ~\cite{Bo-Cr-Si-2}]
\label{2.6} $BV_p^1(I) \cap L^p(I)$ is a Banach algebra.
\end{theorem}

\begin{proof}
For the proof, see ~\cite{Bo-Cr-Si-1} and ~\cite{Bo-Cr-Si-2}.
\end{proof}

\begin{proposition}[~\cite{Bo-Cr-Si-2}]
\label{2.7}  Let $I$ be a compact interval of $\mathbb{R}$%
, then we have the following multiplication properties 
\begin{equation*}
\text{If } f, g \in BV_p^1(\mathbb{R}), \text{ and } \text{supp}(g)
\subseteq I, \text{ then } f g \in BV_p^1(\mathbb{R}),
\end{equation*}
moreover, there exists $c > 0$, such that for all $f, g \in BV_p^1(\mathbb{R}%
)$, $\text{supp}(g) \subseteq I$, we have 
\begin{equation*}
\| f g \|_{BV_p^1(\mathbb{R})} \leq c \| f \|_{BV_p^1(\mathbb{R})} \| g
\|_{BV_p^1(\mathbb{R})}.
\end{equation*}
\end{proposition}

\begin{proof}
For the proof, see ~\cite{Bo-Cr-Si-1} and ~\cite{Bo-Cr-Si-2}.
\end{proof}

\section{\textbf{Solution of the S.O.P  }$BV_p^1(I)$}

\begin{quote}
Theorems \ref{2.9} and \ref{2.13} are due to ~\cite{Bo-Cr-Si-1}.
\end{quote}

\begin{lemma}[~\cite{GIR}]
\label{2.8}  Let $a, b, c \in \mathbb{R}$, $a < b < c$, and let $%
h $ be a measurable function defined on $[a, c]$ with real values such that 
\begin{eqnarray*}
\int_a^b h(x) \,\mathrm{d}x &\geq& 0, \\
\int_b^c h(x) \,\mathrm{d}x &<& 0,
\end{eqnarray*}
then there exist $u, v \in ]a, c[$, such that $h(u) h(v) \leq 0$.
\end{lemma}

\begin{proof}
Applying Bonnet's theorem (second mean value theorem for integrals) to the
function $h$ on the interval $[a, b]$, then there exists an average value $%
\langle h \rangle_{[a,b]} = u \in ]a, b[$, such that 
\begin{equation*}
\int_a^b h(x) \,\mathrm{d}x = h(u) (b - a).
\end{equation*}
Since $\int_a^b h(x) \,\mathrm{d}x \geq 0$, and $b - a \geq 0$, then $h(u)
\geq 0$. Similarly, for the interval $[b, c]$, there exists an average value 
$\langle h \rangle_{[b,c]} = v \in ]b, c[$, such that 
\begin{equation*}
\int_b^c h(x) \,\mathrm{d}x = h(v) (c - b).
\end{equation*}
Since $\int_b^c h(x) \,\mathrm{d}x \leq 0$, and $c - b \geq 0$, then $h(v)
\leq 0$, whence the existence of $u, v$ in $]a, c[$, such that $h(u) h(v)
\leq 0$.
\end{proof}

\begin{theorem}[~\cite{Bo-Cr-Si-1},\, (Basic inequality)]
\label{2.9}   If $p \in [1, +\infty[$, $%
t_0 \in I$, $\alpha \in \mathbb{R}$, $h \in \mathcal{V}_p(I)$, such that 
\begin{eqnarray*}
g(t) &=& \alpha + \int_{t_0}^t h(x) \,\mathrm{d}x, \quad f \in \mathcal{V}%
_p(g(I)), \text{ then} \\
\nu_p((f \circ g) \cdot h, I) &\leq& \nu_p(f, g(I)) \cdot (\sup_I |h| +
2^{1/p} \cdot \nu_p(h, I)) + \nu_p(h, I) \cdot \sup_{g(I)} |f|,
\label{[2.1]} \\
\| (f \circ g) \cdot h \|_{\mathcal{V}_p(I)} &\leq& 2^{1/p} \| f \|_{%
\mathcal{V}_p(g(I))} \| h \|_{\mathcal{V}_p(I)}.  \label{[2.2]}
\end{eqnarray*}
\end{theorem}

\begin{proof}
1) We have 
\begin{eqnarray*}
\sum_{j=1}^N | (f \circ g) \cdot h(t_j) - (f \circ g) \cdot h(t_{j-1}) | &=&
\sum_{k=0}^{N-1} | (f \circ g) \cdot h(t_{k+1}) - (f \circ g) \cdot h(t_k) |
\\
&=& \sum_{k=0}^{N-1} \left| [ f \circ g(t_{k+1}) - f \circ g(t_k) ] \cdot
h(t_k) + f \circ g(t_{k+1}) \cdot (h(t_{k+1}) - h(t_k)) \right|,
\end{eqnarray*}
hence by Minkowski's inequality 
\begin{eqnarray*}
\left( \sum_{j=1}^N | (f \circ g) \cdot h(t_j) - (f \circ g) \cdot
h(t_{j-1}) |^p \right)^{1/p} &\leq& A_1^{1/p} + B_1^{1/p},
\end{eqnarray*}
such that 
\begin{eqnarray*}
A_1 &=& \sum_{k=0}^{N-1} \left| [ f \circ g(t_{k+1}) - f \circ g(t_k) ]
\cdot h(t_k) \right|^p, \\
B_1 &=& \sum_{k=0}^{N-1} \left| f \circ g(t_{k+1}) \cdot [ h(t_{k+1}) -
h(t_k) ] \right|^p.
\end{eqnarray*}
Since $| f \circ g(t_{k+1}) | \leq \sup_{g(I)} |f|$, for all $k \in \mathbb{N%
}$, then $$B_1^{1/p} \leq (\sup_{g(I)} |f|) \cdot \left( \sum_{k=0}^{N-1} |
h(t_{k+1}) - h(t_k) |^p \right)^{1/p} \leq (\sup_{g(I)} |f|) \cdot \nu_p(h,
I).$$

There exists a sequence $0 = n_0 < n_1 < ... < n_J = N$, such that for all
indices $j = 1, ..., J$, we have

\begin{itemize}
\item[(i)] The restriction $s_j$ of $(g(t_k))_{0 \leq k \leq N}$ to $%
\{k\}_{n_{j-1} \leq k \leq n_j, k \in \mathbb{N}}$ is monotone.

\item[(ii)] The restriction of $(g(t_k))_{0 \leq k \leq N}$ to $%
\{k\}_{n_{j-1} \leq k \leq n_j + 1}$ is not monotone.
\end{itemize}

Thus, $[0, N] = [n_0, n_1 - 1] \cup [n_1, n_2 - 1] \cup ... \cup [n_{J-1},
n_J - 1]$.

Set $A_1 = A_2 + B_2$, such that 
\begin{eqnarray*}
A_2 &=& \sum_{j=1}^{J-1} \sum_{k = n_{j-1}}^{n_j - 1} \left| f(g(t_{k+1})) -
f(g(t_k)) \right|^p \cdot | h(t_k) |^p, \\
B_2 &=& \sum_{k = n_{J-1}}^{n_J - 1} \left| f(g(t_{k+1})) - f(g(t_k))
\right|^p \cdot | h(t_k) |^p.
\end{eqnarray*}
- If $J \neq 1$, then $B_2 \leq \nu_p^p(f, g(I)) \cdot \sup_I |h|^p$, and $%
A_2 \leq \nu_p^p(f, g(I)) \cdot \sum_{j=1}^{J-1} | h(t_{k_j}) |^p$.

There always exists some $a_j \in ] t_{k_j}, t_{n_j + 1} [$, such that $%
h(a_j) h(t_{k_j}) \leq 0$. If $h(t_{k_j}) = 0$ or if $h$ changes sign on the
interval $] t_{k_j}, t_{n_j + 1} [$, the result is direct; otherwise, we can
consider a function $\hat{h}$ defined from functions of the form $\pm h$
such that 
\begin{equation*}
\int_{t_{k_j}}^b \hat{h}(x) \,\mathrm{d}x \geq 0 \quad \text{and} \quad
\int_b^{t_{n_j + 1}} \hat{h}(x) \,\mathrm{d}x < 0, \quad \text{where } b \in
] t_{k_j}, t_{n_j + 1} [,
\end{equation*}
then according to Lemma \ref{2.8}, there exist $a_j^{\prime }, a_j^{\prime
\prime }\in ] t_{k_j}, t_{n_j + 1} [$, such that $\hat{h}(a_j^{\prime }) 
\hat{h}(a_j^{\prime \prime }) < 0$, hence $$\forall j = 1, ..., J-1, \exists
a_j \in ] t_{k_j}, t_{n_j + 1} [: \hat{h}(a_j) \hat{h}(t_{k_j}) \leq 0,$$
such that 
\begin{equation*}
\sum_{j=1}^{J-1} | h(t_{k_j}) |^p = \sum_{j=1}^{J-1} | \hat{h}(t_{k_j}) |^p
\quad \text{and} \quad \nu_p^p(h, I) = \nu_p^p(\hat{h}, I).
\end{equation*}
Let $M = \max \{ m \in \mathbb{N} : 2m + 1 \leq J \}$, then 
\begin{eqnarray*}
\sum_{j=1}^{J-1} | h(t_{k_j}) |^p &=& \left\{ 
\begin{array}{cl}
\sum_{j=1}^{2M} | h(t_{k_j}) |^p, & \text{if } J-1 \text{ is even}, \\ 
\sum_{j=1}^{2M} | h(t_{k_j}) |^p + | h(t_{k_{J-1}}) |^p, & \text{if } J-1 
\text{ is odd},%
\end{array}
\right. \\
\text{hence} \quad \sum_{j=1}^{2M} | h(t_{k_j}) |^p &\leq& \sum_{l=1}^M |
h(t_{k_{2l}}) - h(a_{2l}) |^p + \sum_{l=1}^M | h(t_{k_{2l-1}}) - h(a_{2l-1})
|^p, \\
\text{and} \quad | h(t_{k_{J-1}}) |^p &\leq& | h(t_{k_{J-1}}) - h(a_{J-1})
|^p.
\end{eqnarray*}
By the inequalities $t_{n_{j-1}} \leq t_{k_j} < t_{n_j}$, and $t_{k_j} < a_j
< t_{n_j + 1}$, for all $j = 1, ..., J-1$, we deduce that $a_j < t_{n_j + 1}
\leq t_{n_{j+1}} \leq t_{k_{j+2}}$. - If $J \geq 4$ and $j = 1, ..., J-3$,
then the intervals $[t_{k_{2l}}, a_{2l}]$ are pairwise disjoint for $l = 1,
..., M$, and the intervals $[t_{k_{2l-1}}, a_{2l-1}]$ are pairwise disjoint
if $J$ is odd, for $l = 1, ..., M+1$. - If $J$ is even, then 
\begin{eqnarray*}
\sum_{j=1}^{J-1} | \hat{h}(t_{k_j}) |^p &\leq& \sum_{l=1}^M | \hat{h}%
(t_{k_{2l}}) - \hat{h}(t_{k_{2l+2}}) |^p + \sum_{l=1}^M | \hat{h}%
(t_{k_{2l-1}}) - \hat{h}(t_{k_{2l+1}}) |^p \\
&\leq& 2 \nu_p^p(\hat{h}, I),
\end{eqnarray*}
hence $\sum_{j=1}^{J-1} | h(t_{k_j}) |^p \leq 2 \nu_p^p(h, I)$.

Since $(s_J)$ is monotone and $p \geq 1$, and given the above inequalities,
we have 
\begin{eqnarray*}
\nu_p((f \circ g) \cdot h, I) &\leq& \nu_p(f, g(I)) \cdot \left( 2
\nu_p^p(h, I) + \sup_I |h|^p \right)^{1/p} + \nu_p(h, I) \cdot \sup_{g(I)}
|f|, \\
\text{and} \quad \left( 2 \nu_p^p(h, I) + \sup_I |h|^p \right)^{1/p} &\leq&
2^{1/p} \nu_p(h, I) + \sup_I |h|,
\end{eqnarray*}
hence: 
\begin{equation*}
\nu_p((f \circ g) \cdot h, I) \leq \nu_p(f, g(I)) \cdot \left( 2^{1/p}
\nu_p(h, I) + \sup_I |h| \right) + \nu_p(h, I) \cdot \sup_{g(I)} |f|.
\end{equation*}
2) Inequality (\ref{2.2}) follows directly from inequality (\ref{2.1}).
\end{proof}

\begin{theorem}[~\cite{Mau}]
\label{2.10}  Let $f_1(x)$ and $f_2(x)$ be two functions defined,
continuous and differentiable on a compact set of $[a, b]$, such that $%
f_1^{\prime }(x) = f_2^{\prime }(x)$, (a.e.), then $f_1(x) - f_2(x)$ is
constant.
\end{theorem}

\begin{proof}
This theorem is due to De la Valle de-Poussin, and for the proof, see ~\cite%
{Mau}.
\end{proof}

\begin{theorem}[~\cite{Hein} ]
\label{2.11} Let $\Omega \subset \mathbb{R}^n$ be an open set,
and let $f: \Omega \to \mathbb{R}^n$ be a Lipschitz function, then $f$ is
differentiable almost everywhere on $\Omega$; moreover, the following
assertions are verified:

\begin{itemize}
\item[(i)] $f$ is differentiable almost everywhere on the set $L(f)$ such
that 
\begin{eqnarray*}
L(f) &=& \left\{ x \in \Omega : \text{Lip } f(x) < \infty \right\}, \\
\text{where} \quad \text{Lip } f(x) &=& \lim_{y \to x} \sup_{y \in \Omega} 
\frac{| f(x) - f(y) |}{|x - y|}.
\end{eqnarray*}

\item[(ii)] Every function $f \in W^{1,p}(\Omega)$, $p \in [1, +\infty]$, is
differentiable (a.e.).
\end{itemize}
\end{theorem}

\begin{proof}
This theorem of Rademacher is a generalization of Lebesgue's theorem to the
case $$n = m = 1,\quad \Omega = ]a, b[$$ for functions of bounded variation, and
for the proof, see ~\cite{Hein}; assertion (i) is due to a theorem of
Stepanov, and (ii) to a theorem of Calderon.
\end{proof}

\begin{proposition}[~\cite{Hein}]
\label{2.12}  Let $f: [a, b] \to \mathbb{R}$ be a Lipschitz
function, then 
\begin{equation*}
f(b) - f(a) = \int_a^b f^{\prime }(t) \,\mathrm{d}t,
\end{equation*}
moreover, $f: [a, b] \to \mathbb{R}$ is absolutely continuous if and only if
there exists a function $g \in L^1[a, b]$ such that $f(x) = f(a) + \int_a^x
g(t) \,\mathrm{d}t$, $x \in [a, b]$, and in this case the derivative $\frac{%
\mathrm{d}f}{\mathrm{d}x}$ exists for almost every $x \in [a, b]$ such that $%
\frac{\mathrm{d}f}{\mathrm{d}x} = g \in L^1[a, b]$.
\end{proposition}

\begin{theorem}[~\cite{Bo-Cr-Si-1}]
\label{2.13}  Let $1 \leq p < \infty$, then the following
assertions are verified:

\begin{itemize}
\item[(i)] If $f, g \in BV_p^1(\mathbb{R})$, then $f \circ g \in BV_p^1(%
\mathbb{R})$, and 
\begin{equation*}
\| f \circ g \|_{BV_p^1(\mathbb{R})} \leq \| f \|_{BV_p^1(\mathbb{R})}
\left( 1 + 2^{1/p} \| g \|_{BV_p^1(\mathbb{R})} \right).
\end{equation*}

\item[(ii)] Let $f: \mathbb{R} \to \mathbb{R}$ be a Borel measurable
function, then:\\ \begin{center} The operator $T_f$ operates on $BV_p^1(\mathbb{R})$ if and
only if $f \in BV_p^1(\mathbb{R})$.\end{center} 
\end{itemize}
\end{theorem}

This Theorem \ref{2.13} is fundamental because it solves the composition
operator problem such that 
\begin{equation*}
S(BV_p^1(\mathbb{R})) = BV_p^1(\mathbb{R}) = \left\{ \phi: \mathbb{R} \to 
\mathbb{R} ; T_\phi(BV_p^1(\mathbb{R})) \subseteq BV_p^1(\mathbb{R})
\right\}.
\end{equation*}

\begin{proof}
\begin{description}
\item 

\item[(i)] Since $f, g \in BV_p^1(\mathbb{R})$, then $f, g$ are continuous
and Lipschitz, and by Rademacher's theorem, they are differentiable almost
everywhere, and we thus have $(f \circ g)^{\prime }(x) = (f^{\prime }\circ
g) \cdot g^{\prime }(x)$. By De la Vallee Poussin's Theorem \ref{2.10}, the
quantity $(f \circ g)(x) - \int_0^x f^{\prime }(g(t)) g^{\prime }(t) \,%
\mathrm{d}t$ is constant, and applying Proposition \ref{2.12}, we obtain 
\begin{equation*}
(f \circ g)(x) = (f \circ g)(0) + \int_0^x f^{\prime }(g(t)) g^{\prime }(t)
\,\mathrm{d}t.
\end{equation*}
Since $f^{\prime }, g^{\prime }\in \mathcal{V}_p(\mathbb{R})$, then we can
apply Theorem \ref{2.9}, and we conclude that $(f^{\prime }\circ g)
g^{\prime }\in \mathcal{V}_p(\mathbb{R})$, such that 
\begin{eqnarray*}
\| (f^{\prime }\circ g) g^{\prime }\|_{\mathcal{V}_p(\mathbb{R})} &\leq&
2^{1/p} \| f^{\prime }\|_{\mathcal{V}_p(\mathbb{R})} \| g^{\prime }\|_{%
\mathcal{V}_p(\mathbb{R})} < \infty.
\end{eqnarray*}
We have 
\begin{eqnarray*}
\| f \circ g \|_{BV_p^1(\mathbb{R})} &\approx& | f \circ g(x_0) | + | (%
\widetilde{f \circ g})^{\prime }|_{\mathcal{V}_p(\mathbb{R})} \\
&\lesssim& | f \circ g(x_0) | + 2^{1/p} \| (\tilde{f})^{\prime }\|_{\mathcal{%
V}_p(\mathbb{R})} \| (\tilde{g})^{\prime }\|_{\mathcal{V}_p(\mathbb{R})}.
\end{eqnarray*}
Set $g(x_0) = y_0$, then we have 
\begin{eqnarray*}
\| f \circ g \|_{BV_p^1(\mathbb{R})} &\lesssim& | \tilde{f}(y_0) | + 2^{1/p}
\| \tilde{f}^{\prime }\|_{\mathcal{V}_p(\mathbb{R})} \| \tilde{g}^{\prime
}\|_{\mathcal{V}_p(\mathbb{R})} \\
&\lesssim& | \tilde{f}(y_0) | + \| \tilde{f}^{\prime }\|_{\mathcal{V}_p(%
\mathbb{R})} + 2^{1/p} \| \tilde{f}^{\prime }\|_{\mathcal{V}_p(\mathbb{R})}
\cdot \left( \| \tilde{g}^{\prime }\|_{\mathcal{V}_p(\mathbb{R})} + | g(y_0)
| \right) \\
&+& 2^{1/p} \left( | g(y_0) | + \| \tilde{g}^{\prime }\|_{\mathcal{V}_p(%
\mathbb{R})} \right) | f(y_0) | \\
&\lesssim& \left( | f(y_0) | + \| \tilde{f}^{\prime }\|_{\mathcal{V}_p(%
\mathbb{R})} \right) \cdot \left[ 1 + 2^{1/p} \left( | g(y_0) | + \| \tilde{g%
}^{\prime }\|_{\mathcal{V}_p(\mathbb{R})} \right) \right] \\
&\approx& \| f \|_{BV_p^1(\mathbb{R})} \cdot \left( 1 + 2^{1/p} \| g
\|_{BV_p^1(\mathbb{R})} \right),
\end{eqnarray*}
hence $\| f \circ g \|_{BV_p^1(\mathbb{R})} \lesssim \| f \|_{BV_p^1(\mathbb{%
R})} \left( 1 + 2^{1/p} \| g \|_{BV_p^1(\mathbb{R})} \right)$.

\item[(ii)] 

\begin{itemize}
\item[--] If $f \in BV_p^1(\mathbb{R})$, then according to (i), 
\begin{eqnarray*}
\| f \circ g \|_{BV_p^1(\mathbb{R})} &\lesssim& \| f \|_{BV_p^1(\mathbb{R})}
\left( 1 + 2^{1/p} \| g \|_{BV_p^1(\mathbb{R})} \right) < \infty, \quad 
\text{for all } g \in BV_p^1(\mathbb{R}), \\
\text{hence } f \circ g &\in& BV_p^1(\mathbb{R}), \quad \text{for all } g
\in BV_p^1(\mathbb{R}),
\end{eqnarray*}
and thus $f$ operates on $BV_p^1(\mathbb{R})$.

\item[--] If $f$ operates on $BV_p^1(\mathbb{R})$, then $T_f(BV_p^1(\mathbb{R%
})) \subseteq BV_p^1(\mathbb{R})$.\\ We have $$\text{id}_\mathbb{R}(x) = x =
x_0 + \int_{x_0}^x \mathrm{d}t = x_0 + \int_{x_0}^x 1 \,\mathrm{d}t,$$ such:
that $$1 \in BV_p(\mathbb{R}).$$ Moreover: 
\begin{eqnarray*}
\| \text{id}_\mathbb{R} \|_{BV_p^1(\mathbb{R})} &\approx& | f(x_0) | + \| (%
\text{id}_\mathbb{R})^{\prime }\|_{BV_p(\mathbb{R})} \\
&\approx& | x_0 | + \| 1 \|_{BV_p(\mathbb{R})} \\
&\approx& | x_0 | + \nu_p(1) + \sup_{x \in \mathbb{R}} |1| \\
&=& | x_0 | + 0 + 1 < \infty,
\end{eqnarray*}
hence $\text{id}_\mathbb{R} \in BV_p^1(\mathbb{R})$, and thus $f = f \circ 
\text{id}_\mathbb{R} = T_f(\text{id}_\mathbb{R}) \in BV_p^1(\mathbb{R})$.
\end{itemize}
\end{description}
\end{proof}

\section{\textbf{Statement of results}}

\begin{quote}
In this paragraph, we present Theorem \ref{2.15}, which represents our
contribution, namely a generalization of the basic inequality introduced in
Theorem \ref{2.9} (~\cite{Bo-Cr-Si-1}) by Lemma \ref{2.14}, giving the
derivative of $n$ functions, followed by Example \ref{2.16} to affirm the
obtained results, and we agree to take
\end{quote}

\begin{eqnarray*}
\overset{i=m}{\underset{i=n}{{\Large \circ}}} g_i &=& \left\{ 
\begin{array}{cl}
g_n \circ g_{n+1} \circ \cdots \circ g_k \circ \cdots \circ g_m, & \text{if }
m \geq n \\ 
\text{id}, & \text{if } m < n%
\end{array}
\right. \\
&\text{and}& \\
\prod_{i=n}^{i=m} A_i &=& \left\{ 
\begin{array}{cl}
A_n \times A_{n+1} \times \cdots \times A_m, & \text{if } m \geq n \\ 
1, & \text{if } m < n%
\end{array}
\right.
\end{eqnarray*}

\begin{lemma}
\label{2.14} (Basic algorithm) Let $(I_k)_{2 \leq k \leq n}$ be a sequence
of intervals of $\mathbb{R}$, and let $(g_k)_{1 \leq k \leq n}$ be a
sequence of differentiable functions such that $g_k: I_k \to I_{k-1}$, then
for all $n \geq 2$, 
\begin{equation*}
\left[ \overset{i=n}{\underset{i=1}{{\Large \circ}}} g_i \right]^{\prime }=
g_n^{\prime }\times \prod_{i=1}^{n-1} \left[ g_i^{\prime }\circ \left( 
\overset{j=n}{\underset{j=i+1}{{\Large \circ}}} g_j \right) \right] =
\prod_{i=1}^n \left[ g_i^{\prime }\circ \left( \overset{j=n}{\underset{j=i+1}%
{{\Large \circ}}} g_j \right) \right].
\end{equation*}
\end{lemma}

\begin{proof}
For $n = 3$, we have 
\begin{eqnarray*}
(g_1 \circ g_2 \circ g_3)^{\prime }&=& (g_1 \circ (g_2 \circ g_3))^{\prime
}= (g_2 \circ g_3)^{\prime }\times (g_1^{\prime }\circ (g_2 \circ g_3)) \\
&=& g_3^{\prime }\times (g_2^{\prime }\circ g_3) \times (g_1^{\prime }\circ
(g_2 \circ g_3)) = g_3^{\prime }\times (g_1^{\prime }\circ g_2 \circ g_3)
\times (g_2^{\prime }\circ g_3).
\end{eqnarray*}
By induction, for $n$ terms, 
\begin{equation*}
(g_1 \circ g_2 \circ \cdots \circ g_n)^{\prime }= g_n^{\prime }\times
(g_1^{\prime }\circ g_2 \circ \cdots \circ g_n) \times \cdots \times
(g_k^{\prime }\circ g_{k+1} \circ \cdots \circ g_n) \times \cdots \times
g_{n-1}^{\prime }\circ g_n.
\end{equation*}
We deduce that for $n+1$ terms, taking as the $n$-th term $(g_n \circ
g_{n+1})$, 
\begin{eqnarray*}
(g_1 \circ \cdots \circ (g_n \circ g_{n+1}))^{\prime }&=& (g_n \circ
g_{n+1})^{\prime }\times (g_1^{\prime }\circ \cdots \circ (g_n \circ
g_{n+1})) \times \cdots \times g_{n-1}^{\prime }\circ (g_n \circ g_{n+1}) \\
&=& g_{n+1}^{\prime }\times (g_n^{\prime }\circ g_{n+1}) \times (g_1^{\prime
}\circ \cdots \circ g_n \circ g_{n+1}) \times \cdots \times g_{n-1}^{\prime
}\circ (g_n \circ g_{n+1}) \\
&=& g_{n+1}^{\prime }\times \left[ g_n^{\prime }\circ g_{n+1} \times
\prod_{i=1}^{n-1} g_i^{\prime }\circ \left( \overset{j=n+1}{\underset{j=i+1}{%
{\Large \circ}}} g_j \right) \right] \\
&=& g_{n+1}^{\prime }\times \left[ \prod_{i=1}^n g_i^{\prime }\circ \left( 
\overset{j=n+1}{\underset{j=i+1}{{\Large \circ}}} g_j \right) \right].
\end{eqnarray*}
\end{proof}

\begin{theorem}
\label{2.15} Let $p \in [1, +\infty[$, $n \geq 2$, $h \in \mathcal{V}_p(I)$,
such that for all $t \in I$, $t_0 \in I$, $\alpha \in \mathbb{R}$, 
\begin{equation*}
\overset{i=n}{\underset{i=1}{{\Large \circ}}} g_i(t) = \alpha + \int_{t_0}^t
h(x) \,\mathrm{d}x \quad \text{and} \quad f \in \mathcal{V}_p \left( \overset%
{i=n}{\underset{i=1}{{\Large \circ}}} g_i(I) \right), \text{ then}
\end{equation*}

\begin{itemize}
\item[(i)] $\displaystyle \left\| \left[ \overset{i=n}{\underset{i=1}{%
{\Large \circ}}} g_i \right]^{\prime }\right\|_{\mathcal{V}_p(I)} = \left\|
g_n^{\prime }\times \prod_{i=1}^{n-1} \left[ g_i^{\prime }\circ \left( 
\overset{j=n}{\underset{j=i+1}{{\Large \circ}}} g_j \right) \right]
\right\|_{\mathcal{V}_p(I)} \leq 2^{\frac{n-1}{p}} \| g_n^{\prime }\|_{%
\mathcal{V}_p(I)} \times \prod_{k=1}^{n-1} \| g_k^{\prime }\|_{\mathcal{V}_p
\left( \overset{n}{\underset{\ell=k+1}{\circ}} g_\ell(I) \right)}$.

\item[(ii)] $\displaystyle \left\| f \circ \left( \overset{i=n}{\underset{i=1%
}{{\Large \circ}}} g_i \right) \times h \right\|_{\mathcal{V}_p(I)} \leq 2^{%
\frac{n}{p}} \| f \|_{\mathcal{V}_p \left( \overset{i=n}{\underset{i=1}{%
{\Large \circ}}} g_i(I) \right)} \times \| g_n^{\prime }\|_{\mathcal{V}%
_p(I)} \times \prod_{k=1}^{n-1} \| g_k^{\prime }\|_{\mathcal{V}_p \left( 
\overset{n}{\underset{\ell=k+1}{\circ}} g_\ell(I) \right)}$.
\end{itemize}
\end{theorem}

\begin{proof}
\begin{description}
\item 

\item[(i)] According to Lemma \ref{2.14}, 
\begin{equation*}
\left\| \left[ \overset{i=n}{\underset{i=1}{{\Large \circ}}} g_i \right]%
^{\prime }\right\|_{\mathcal{V}_p(I)} = \left\| \left[ \prod_{i=1}^{n-1}
g_i^{\prime }\circ \left( \overset{j=n}{\underset{j=i+1}{{\Large \circ}}}
g_j \right) \right] \times g_n^{\prime }\right\|_{\mathcal{V}_p(I)},
\end{equation*}
hence 
\begin{eqnarray*}
\left\| \left[ \overset{i=n}{\underset{i=1}{{\Large \circ}}} g_i \right]%
^{\prime }\right\|_{\mathcal{V}_p(I)} &=& \left\| \left[ \left(
\prod_{i=1}^{n-1} g_i^{\prime }\circ \left( \overset{j=n-1}{\underset{j=i+1}{%
{\Large \circ}}} g_j \right) \right) \circ g_n \right] \times g_n^{\prime
}\right\|_{\mathcal{V}_p(I)} \\
&\leq& 2^{1/p} \left\| \prod_{i=1}^{n-1} g_i^{\prime }\circ \left( \overset{%
j=n-1}{\underset{j=i+1}{{\Large \circ}}} g_j \right) \right\|_{\mathcal{V}%
_p(g_n(I))} \times \| g_n^{\prime }\|_{\mathcal{V}_p(I)}, \quad \text{see } (%
\ref{2.2}).
\end{eqnarray*}
And 
\begin{eqnarray*}
\left\| \prod_{i=1}^{n-1} g_i^{\prime }\circ \left( \overset{j=n-1}{\underset%
{j=i+1}{{\Large \circ}}} g_j \right) \right\|_{\mathcal{V}_p(g_n(I))} &=&
\left\| \left[ \prod_{i=1}^{n-2} g_i^{\prime }\circ \left( \overset{j=n-1}{%
\underset{j=i+1}{{\Large \circ}}} g_j \right) \right] \times g_{n-1}^{\prime
}\circ \left( \overset{j=n-1}{\underset{j=n}{{\Large \circ}}} g_j \right)
\right\|_{\mathcal{V}_p(g_n(I))} \\
&=& \left\| \left[ \left( \prod_{i=1}^{n-2} g_i^{\prime }\circ \left( 
\overset{j=n-2}{\underset{j=i+1}{{\Large \circ}}} g_j \right) \right) \circ
g_{n-1} \right] \times g_{n-1}^{\prime }\right\|_{\mathcal{V}_p(g_n(I))},
\quad \text{since } \overset{j=n-1}{\underset{j=n}{{\Large \circ}}} g_j = 
\text{id} \\
&\leq& 2^{1/p} \left\| \prod_{i=1}^{n-2} g_i^{\prime }\circ \left( \overset{%
j=n-2}{\underset{j=i+1}{{\Large \circ}}} g_j \right) \right\|_{\mathcal{V}%
_p(g_{n-1}(g_n(I)))} \times \| g_{n-1}^{\prime }\|_{\mathcal{V}_p(g_n(I))}.
\end{eqnarray*}
Step by step, we obtain 
\begin{eqnarray*}
\left\| \left[ \overset{i=n}{\underset{i=1}{{\Large \circ}}} g_i \right]%
^{\prime }\right\|_{\mathcal{V}_p(I)} &=& \left\| g_n^{\prime }\times
\prod_{i=1}^{n-1} \left[ g_i^{\prime }\circ \left( \overset{j=n}{\underset{%
j=i+1}{{\Large \circ}}} g_j \right) \right] \right\|_{\mathcal{V}_p(I)} \\
&\leq& 2^{\frac{n-1}{p}} \| g_n^{\prime }\|_{\mathcal{V}_p(I)} \times
\prod_{k=1}^{n-1} \| g_k^{\prime }\|_{\mathcal{V}_p \left( \overset{n}{%
\underset{\ell=k+1}{\circ}} g_\ell(I) \right)}.
\end{eqnarray*}
On the other hand, by induction, we have 
\begin{eqnarray*}
\left\| \left[ \overset{i=n+1}{\underset{i=1}{{\Large \circ}}} g_i \right]%
^{\prime }\right\|_{\mathcal{V}_p(I)} &=& \left\| \left[ \left( \overset{i=n}%
{\underset{i=1}{{\Large \circ}}} g_i \right) \circ g_{n+1} \right]^{\prime
}\right\|_{\mathcal{V}_p(I)} = \left\| \left[ \left( \overset{i=n}{\underset{%
i=1}{{\Large \circ}}} g_i \right)^{\prime }\circ g_{n+1} \right] \times
g_{n+1}^{\prime }\right\|_{\mathcal{V}_p(I)} \\
&\leq& 2^{1/p} \left\| \left( \overset{i=n}{\underset{i=1}{{\Large \circ}}}
g_i \right)^{\prime }\right\|_{\mathcal{V}_p(g_{n+1}(I))} \times \|
g_{n+1}^{\prime }\|_{\mathcal{V}_p(I)}, \quad \text{according to } (\ref%
{[2.2]}) \text{ of Theorem } \ref{2.9} \\
&\leq& 2^{\frac{n-1}{p}} 2^{1/p} \| g_n^{\prime }\|_{\mathcal{V}%
	_p(g_{n+1}(I))} \times \prod_{k=1}^{n-1} \| g_k^{\prime }\|_{\mathcal{V}_p
	\left( \overset{n}{\underset{\ell=k+1}{\circ}} g_\ell(g_{n+1}(I)) \right)}
\times \| g_{n+1}^{\prime }\|_{\mathcal{V}_p(I)} \\
\end{eqnarray*}

\begin{eqnarray*}
&=& 2^{\frac{n}{p}} \| g_{n+1}^{\prime }\|_{\mathcal{V}_p(I)} \times \left[
\prod_{k=1}^{n-1} \| g_k^{\prime }\|_{\mathcal{V}_p \left( \overset{n}{%
\underset{\ell=k+1}{\circ}} g_\ell(g_{n+1}(I)) \right)} \times \|
g_n^{\prime }\|_{\mathcal{V}_p(g_{n+1}(I))} \right] \\
&=& 2^{\frac{n}{p}} \| g_{n+1}^{\prime }\|_{\mathcal{V}_p(I)} \times
\prod_{k=1}^n \| g_k^{\prime }\|_{\mathcal{V}_p \left( \overset{n+1}{%
\underset{\ell=k+1}{\circ}} g_\ell(I) \right)}.
\end{eqnarray*}

\item[(ii)] Since $h(t) = \left( \overset{i=n}{\underset{i=1}{{\Large \circ}}%
} g_i(t) \right)^{\prime }$, and according to (\ref{2.2}), we have 
\begin{equation*}
\left\| \left[ f \circ \left( \overset{i=n}{\underset{i=1}{{\Large \circ}}}
g_i \right) \right] \times h \right\|_{\mathcal{V}_p(I)} \leq 2^{1/p} \| f
\|_{\mathcal{V}_p \left( \overset{i=n}{\underset{i=1}{{\Large \circ}}}
g_i(I) \right)} \times \left\| \left( \overset{i=n}{\underset{i=1}{{\Large %
\circ}}} g_i \right)^{\prime }\right\|_{\mathcal{V}_p(I)},
\end{equation*}
\begin{eqnarray*}
&\leq& \left[ 2^{1/p} \| f \|_{\mathcal{V}_p \left( \overset{i=n}{\underset{%
i=1}{{\Large \circ}}} g_i(I) \right)} \right] \times \left[ 2^{\frac{n-1}{p}%
} \| g_n^{\prime }\|_{\mathcal{V}_p(I)} \times \prod_{k=1}^{n-1} \|
g_k^{\prime }\|_{\mathcal{V}_p \left( \overset{n}{\underset{\ell=k+1}{\circ}}
g_\ell(I) \right)} \right] \\
&=& 2^{\frac{n}{p}} \| f \|_{\mathcal{V}_p \left( \overset{i=n}{\underset{i=1%
}{{\Large \circ}}} g_i(I) \right)} \times \| g_n^{\prime }\|_{\mathcal{V}%
_p(I)} \times \prod_{k=1}^{n-1} \| g_k^{\prime }\|_{\mathcal{V}_p \left( 
\overset{n}{\underset{\ell=k+1}{\circ}} g_\ell(I) \right)}.
\end{eqnarray*}
\end{description}
\end{proof}

\begin{example}\text{ \ }
\label{2.16}
\vspace{-20pt}
\begin{itemize}
\item For $n = 2$, let $g_1 \circ g_2 = h$, then we have

\begin{description}
\item[(i)] $\left\| (g_1 \circ g_2)^{\prime }\right\|_{\mathcal{V}_p(I)} =
\left\| g_2^{\prime }\times (g_1^{\prime }\circ g_2) \right\|_{\mathcal{V}%
_p(I)} = \left\| (g_1^{\prime }\circ g_2) \times g_2^{\prime }\right\|_{%
\mathcal{V}_p(I)} \leq 2^{1/p} \| g_1^{\prime }\|_{\mathcal{V}_p(g_2(I))} \|
g_2^{\prime }\|_{\mathcal{V}_p(I)}$.

\item[(ii)] $\displaystyle \left\| f \circ (g_1 \circ g_2) \times h
\right\|_{\mathcal{V}_p(I)} \leq 2^{1/p} \| f \|_{\mathcal{V}_p(g_1 \circ
g_2(I))} \| h \|_{\mathcal{V}_p(I)} = 2^{1/p} \| f \|_{\mathcal{V}_p(g_1
\circ g_2(I))} \| (g_1 \circ g_2)^{\prime }\|_{\mathcal{V}_p(I)}$. 
\begin{eqnarray*}
&\leq& \left[ 2^{1/p} \| f \|_{\mathcal{V}_p(g_1 \circ g_2(I))} \right]
\cdot \left[ 2^{1/p} \| g_1^{\prime }\|_{\mathcal{V}_p(g_2(I))} \|
g_2^{\prime }\|_{\mathcal{V}_p(I)} \right] \\
&\leq& 2^{2/p} \| f \|_{\mathcal{V}_p(g_1 \circ g_2(I))} \| g_1^{\prime }\|_{%
\mathcal{V}_p(g_2(I))} \| g_2^{\prime }\|_{\mathcal{V}_p(I)},
\end{eqnarray*}
and we have $$2^{2/p} \| f \|_{\mathcal{V}_p \left( \overset{i=2}{\underset{%
i=1}{{\Large \circ}}} g_i(I) \right)} \| g_2^{\prime }\|_{\mathcal{V}_p(I)}
\prod_{k=1}^1 \| g_k^{\prime }\|_{\mathcal{V}_p \left( \overset{2}{\underset{%
\ell=k+1}{\circ}} g_\ell(I) \right)} = 2^{2/p} \| f \|_{\mathcal{V}_p(g_1
\circ g_2(I))} \| g_2^{\prime }\|_{\mathcal{V}_p(I)} \| g_1^{\prime }\|_{%
\mathcal{V}_p(g_2(I))}$$
\end{description}

\item For $n = 3$, let $g_1 \circ g_2 \circ g_3 = h$, then we have

\begin{description}
\item[(i)] 
\begin{eqnarray*}
\left\| (g_1 \circ g_2 \circ g_3)^{\prime }\right\|_{\mathcal{V}_p(I)} &=&
\left\| ((g_1 \circ g_2) \circ g_3)^{\prime }\right\|_{\mathcal{V}_p(I)} =
\left\| g_3^{\prime }\times (g_1 \circ g_2)^{\prime }\circ g_3 \right\|_{%
\mathcal{V}_p(I)} \\
&=& \left\| [ (g_1 \circ g_2)^{\prime }\circ g_3 ] \times
g_3^{\prime }\right\|_{\mathcal{V}_p(I)} \\
&\leq& 2^{1/p} \left\| (g_1 \circ g_2)^{\prime }\right\|_{\mathcal{V}%
_p(g_3(I))} \| g_3^{\prime }\|_{\mathcal{V}_p(I)} \\
&\leq& 2^{1/p} \left[ 2^{1/p} \| g_1^{\prime }\|_{\mathcal{V}%
_p(g_2(g_3(I)))} \| g_2^{\prime }\|_{\mathcal{V}_p(g_3(I))} \right] \|
g_3^{\prime }\|_{\mathcal{V}_p(I)} \\
&=& 2^{2/p} \| g_1^{\prime }\|_{\mathcal{V}_p(g_2(g_3(I)))} \| g_2^{\prime
}\|_{\mathcal{V}_p(g_3(I))} \| g_3^{\prime }\|_{\mathcal{V}_p(I)},
\end{eqnarray*}
and we have $2^{\frac{3-1}{p}} \| g_3^{\prime }\|_{\mathcal{V}_p(I)}
\prod_{k=1}^{3-1} \| g_k^{\prime }\|_{\mathcal{V}_p \left( \overset{3}{%
\underset{\ell=k+1}{\circ}} g_\ell(I) \right)} = 2^{2/p} \| g_3^{\prime }\|_{%
\mathcal{V}_p(I)} \prod_{k=1}^2 \| g_k^{\prime }\|_{\mathcal{V}_p \left( 
\overset{3}{\underset{\ell=k+1}{\circ}} g_\ell(I) \right)}$ 
\begin{equation*}
= 2^{2/p} \| g_3^{\prime }\|_{\mathcal{V}_p(I)} \| g_1^{\prime }\|_{\mathcal{%
V}_p \left( \overset{3}{\underset{\ell=2}{\circ}} g_\ell(I) \right)} \|
g_2^{\prime }\|_{\mathcal{V}_p \left( \overset{3}{\underset{\ell=3}{\circ}}
g_\ell(I) \right)} = 2^{2/p} \| g_3^{\prime }\|_{\mathcal{V}_p(I)} \|
g_1^{\prime }\|_{\mathcal{V}_p(g_2 \circ g_3(I))} \| g_2^{\prime }\|_{%
\mathcal{V}_p(g_3(I))}.
\end{equation*}

\item[(ii)] $\left\| f \circ (g_1 \circ g_2 \circ g_3) \times h \right\|_{%
\mathcal{V}_p(I)} \leq 2^{1/p} \| f \|_{\mathcal{V}_p(g_1 \circ g_2 \circ
g_3(I))} \| h \|_{\mathcal{V}_p(I)}$ 
\begin{eqnarray*}
&=& 2^{1/p} \| f \|_{\mathcal{V}_p(g_1 \circ g_2 \circ g_3(I))} \| (g_1
\circ g_2 \circ g_3)^{\prime }\|_{\mathcal{V}_p(I)} \\
&\leq& \left[ 2^{1/p} \| f \|_{\mathcal{V}_p(g_1 \circ g_2 \circ g_3(I))} %
\right] \left[ 2^{2/p} \| g_3^{\prime }\|_{\mathcal{V}_p(I)} \| g_1^{\prime
}\|_{\mathcal{V}_p(g_2 \circ g_3(I))} \| g_2^{\prime }\|_{\mathcal{V}%
_p(g_3(I))} \right] \\
&\leq& 2^{3/p} \| f \|_{\mathcal{V}_p(g_1 \circ g_2 \circ g_3(I))} \|
g_3^{\prime }\|_{\mathcal{V}_p(I)} \| g_1^{\prime }\|_{\mathcal{V}_p(g_2
\circ g_3(I))} \| g_2^{\prime }\|_{\mathcal{V}_p(g_3(I))}.
\end{eqnarray*}
On the other hand, since $\prod_{k=1}^2 \| g_k^{\prime }\|_{\mathcal{V}_p
\left( \overset{3}{\underset{\ell=k+1}{\circ}} g_\ell(I) \right)} = \|
g_1^{\prime }\|_{\mathcal{V}_p \left( \overset{3}{\underset{\ell=1+1}{\circ}}
g_\ell(I) \right)} \| g_2^{\prime }\|_{\mathcal{V}_p \left( \overset{3}{%
\underset{\ell=2+1}{\circ}} g_\ell(I) \right)}$, then 

\begin{equation*}
\hspace{-1in}
2^{3/p} \| f \|_{\mathcal{V}_p \left( \overset{i=3}{\underset{i=1}{{\Large %
\circ}}} g_i(I) \right)} \| g_3^{\prime }\|_{\mathcal{V}_p(I)} \prod_{k=1}^2
\| g_k^{\prime }\|_{\mathcal{V}_p \left( \overset{3}{\underset{\ell=k+1}{%
\circ}} g_\ell(I) \right)} = 2^{3/p} \| f \|_{\mathcal{V}_p(g_1 \circ g_2
\circ g_3(I))} \| g_3^{\prime }\|_{\mathcal{V}_p(I)} \| g_1^{\prime }\|_{%
\mathcal{V}_p(g_2 \circ g_3(I))} \| g_2^{\prime }\|_{\mathcal{V}_p(g_3(I))}.
\end{equation*}
\end{description}
\end{itemize}
\end{example}

\chapter{\textbf{Composition of Operators in Homogeneous Besov Spaces} $\dot{%
B}_{p}^{s,q}\left( \mathbb{R}^{n}\right) $}

\begin{quote}
In this chapter, we first recall some basic notions, then define the norms
of homogeneous Besov spaces and their realizations, as well as functional
calculus by presenting two fundamental theorems (\ref{3.12}) and (\ref{3.13}%
). Finally, we give an extension of Peetre's Theorem (\ref{3.17}) along with
its proof by introducing certain functional spaces based on the works of ~%
\cite{Bo-Cr-Si-1}, ~\cite{Bo-Cr-Si-2}.
\end{quote}

\begin{definition}
\label{3.1} 
Let $\left( X,\left\Vert .\right\Vert _{X}\right) ,\ \left(Y,\left\Vert
.\right\Vert _{Y}\right)$ be two normed vector spaces. We denote by $%
X\hookrightarrow Y$ the injection of $X$ into $Y$, which is defined by the
following conditions:

\begin{itemize}
\item[(i)] $X$ is a vector subspace of $Y$.

\item[(ii)] The identity map $J:X\mathbf{\longrightarrow } Y$, such that $%
J\left( x\right) =x$ for all $x\in X$, is continuous.
\end{itemize}
\end{definition}

\begin{quote}
Being linear, the continuity of the identity operator is equivalent to the
existence of a constant $C>0$ such that 
\begin{equation}
\left\Vert f\right\Vert _{Y}\leq C\left\Vert f\right\Vert _{X}\ ,\ \ \ \ \ \ 
\text{for all } f\in X,  \label{[3.1]}
\end{equation}
- If (\ref{3.1}) holds, then we say that $X$ embeds into $Y$. We say that
the space $X$ is embedded in $Y$ if there exists a continuous operator $U :X%
\mathbf{\longrightarrow }Y$ such that 
\begin{equation*}
\left\Vert U \left( f\right) \right\Vert _{Y}\leq C\left\Vert f\right\Vert
_{X}\ ,\ \ \ \ \ \ \text{for all } f\in X.
\end{equation*}
\end{quote}

\begin{definition}[~\cite{Pee}]
\label{3.2}  
Let $A_{0}$ and $A_{1}$ be two compatible quasi-Banach spaces embedded in a
Hausdorff locally convex space. Set $A=\left(A_{0},A_{1}\right)$, such that 
\begin{eqnarray*}
\ \ \ \sum \left( A\right) &=&A_{0}+A_{1}\ ,\ \left\Vert a\right\Vert
_{A_{0}+A_{1}}\ =\inf_{a=a_{0}+a_{1}}\left( \left\Vert a_{0}\right\Vert
_{A_{0}}+\left\Vert a_{1}\right\Vert _{A_{1}}\right) \\
&&\text{and} \\
\Delta \left( A\right) &=&A_{0}\cap A_{1},\ \left\Vert a\right\Vert
_{A_{0}\cap A_{1}}\ =\max_{a\in A_{0}\cap A_{1}} \left( \left\Vert
a\right\Vert _{A_{0}},\left\Vert a\right\Vert _{A_{1}}\right) ,
\end{eqnarray*}
and we define for each $t>0$, the Peetre functionals 
\begin{eqnarray*}
\ K\left( t,a\right) &=&\inf_{a=a_{0}+a_{1}}\left( \left\Vert
a_{0}\right\Vert _{A_{0}}+t\left\Vert a_{1}\right\Vert _{A_{1}}\right) , \ 
\text{for } a\in \sum \left( A\right) \\
&&\text{and} \\
\ J\left( t,a\right) &=&\max \left( \left\Vert a\right\Vert _{A_{0}}\ ,\
t\left\Vert a\right\Vert _{A_{1}}\right) , \ \text{for } a\in \Delta
\left(A\right) .
\end{eqnarray*}
\end{definition}

\begin{quote}
The function $t$ $\mathbf{\mapsto }$ $K\left( t,a\right) $ is positive,
increasing, and concave. For each $0<\theta <1,\ 0<p\leq \infty $, we define
the interpolation space $\left( A\right) _{\theta ,p}$ of the spaces $A_{0}$
and $A_{1}$ by 
\begin{equation*}
\left( A\right) _{\theta ,p}=\left( A_{0}\ ,\ A_{1}\right) _{\theta ,p}\
=\left\{ a\in \sum \left( A\right) \,:\quad \left\Vert a\right\Vert _{\left(
A\right) _{\theta ,p}}<\infty \ \ \ \right\} ,
\end{equation*}
and we define its quasi-norm by 
\begin{eqnarray*}
\ \left\Vert a\right\Vert _{\left( A\right) _{\theta ,p}} &=&\left(
\int_{0}^{\infty }\left( t^{-\theta }K\left( t,a\right) \right) ^{p} \frac{\,%
\mathrm{d}{t}}{t}\right) ^{\frac{1}{p}}\ ,\text{ if }1\leq p<+\infty \\
&&\text{and} \\
\left\Vert a\right\Vert _{\left( A\right) _{\theta ,\infty }}
&=&\sup_{0<t<\infty }t^{-\theta }K\left( t,a\right) .
\end{eqnarray*}
The quasi-norm of $\left( A\right) _{\theta ,p}$ can be replaced by the norm 
\begin{equation*}
\ \ \displaystyle\left\Vert a\right\Vert _{A_{0}+A_{1}}+\left(
\int_{0}^{\infty }\left( t^{-\theta }K\left( t,a\right) \right) ^{p} \frac{\,%
\mathrm{d}{t}}{t}\right) ^{\frac{1}{p}}.
\end{equation*}
\end{quote}
\bigskip \bigskip
\begin{theorem}[~\cite{Pee}]
\label{3.3}  
Let $\left( A_{0},A_{1}\right)$ and $\left( B_{0},B_{1}\right)$ be two
compatible couples of Banach spaces, where the bounded linear operator $T$
embeds $A_{i}$ into $B_{i}$ $\left( i=0,1\right)$, such that 
\begin{equation*}
\ \ \ \ \ \ \ \ \ \ \ \ \ \ \ \ \ \displaystyle \left\Vert Tf\right\Vert
_{B_{0}}\leq M_{0}\left\Vert f\right\Vert _{A_{0}}\,,\ \text{and } %
\displaystyle \left\Vert Tf\right\Vert _{B_{1}}\leq M_{1}\left\Vert
f\right\Vert _{A_{1}},
\end{equation*}
then for all $0<\theta <1,\ 0<p\leq +\infty$, the operator $T$ embeds the
interpolation space $\left( A_{0},A_{1}\right) _{\theta ,p}$ into the
interpolation space $\left( B_{0},B_{1}\right) _{\theta ,p}$, such that 
\begin{equation*}
\ \ \ \ \ \ \ \ \ \ \ \ \ \ \ \ \ \ \displaystyle\left\Vert Tf\right\Vert
_{\left( B_{0},B_{1}\right) _{\theta ,p}}\leq M_{\theta }\left\Vert
f\right\Vert _{\left( A_{0},A_{1}\right) _{\theta ,p}},\ \text{and } %
\displaystyle M_{\theta }\leq M_{0}^{1-\theta }M_{1}^{\theta }.
\end{equation*}
\end{theorem}

\begin{proof}

From this main Theorem \ref{3.3}, all classical theorems such as
Riesz-Thorin or Marcinkiewicz interpolation theorems, and classical
inequalities such as Young's and Bernstein's inequalities follow. For the
proof, see ~\cite{KUF}, as well as Peetre's monograph ~\cite{Pee}, Chap 1.
\end{proof}

\section{\textbf{Composition Problem in Homogeneous Besov Spaces} $\dot{B}%
_{p}^{s,q}\left( \mathbb{R}^{n}\right) $}

\begin{quote}
In this section, we study functional calculus in homogeneous Besov spaces
and introduce the spaces $\mathcal{\dot{B}}_{p}^{s,q}\left(\mathbb{R}%
^{n}\right)$ defined via the set of distributions tending to $0$ at infinity 
$\widetilde{C}_{0}\left(\mathbb{R}^{n}\right)$, which greatly facilitates
functional calculus by allowing the truncation of polynomials.
\end{quote}

\subsection{\textbf{Properties of Homogeneous Besov Spaces}}

\begin{quote}
Let $\psi$ be an infinitely differentiable, even, and positive function,
with support a compact subset of $\mathbb{R}^{n}\backslash \{0\}$, and such
that 
\begin{equation}
\sum_{j\in \mathbb{Z}}\psi \left( 2^{j}\xi \right) =1,\ \text{for all } \xi
\neq 0\,.\   \label{[3.2]}
\end{equation}

The operator $Q_{j}:\mathcal{S}^{\prime }\left( \mathbb{R}^{n}\right) 
\mathbf{\longrightarrow }\mathcal{S}^{\prime }\left( \mathbb{R}^{n}\right)$
is defined by the identity 
\begin{equation*}
\displaystyle\ \ \widehat{Q_{j}f}\left( \xi \right) =\psi \left( 2^{-j}\xi
\right) \hat{f}\ \left( \xi \right) ,\ \left( j\in \mathbb{Z}\right).
\end{equation*}
\end{quote}

\begin{definition}
\label{3.4}
Let $s\in\mathbb{R},\ p,q\in \lbrack 1,+\infty ]$. Then the homogeneous
Besov space $\dot{\tilde{B}}_{p}^{s,q}\left(\mathbb{R}^{n}\right)$ is the
set of equivalence classes of distributions $f\in \mathcal{S}^{\prime }\left(%
\mathbb{R}^{n}\right) /\mathcal{P}_{\infty }( \mathbb{R}^{n})$ such that 
\begin{equation}
\left\Vert f\right\Vert _{\dot{\tilde{B}}_{p}^{s,q}\left(\mathbb{R}%
^{n}\right) }=\left( \sum_{j\in\mathbb{Z}}\left( 2^{js}\left\Vert
Q_{j}f\right\Vert _{p}\right) ^{q}\right)^{1/q}<+\infty\,.  \label{[3.3]}
\end{equation}
\end{definition}
The norm (\ref{3.3}) makes $\dot{\tilde{B}}_{p}^{s,q}\left( \mathbb{R}%
^{n}\right)$ a homogeneous Banach space, satisfying homogeneity properties
such that for all $\lambda >0$, and for all $a\in \mathbb{R}^{n}$, we have:
\begin{eqnarray}
\left\Vert \tau _{a}f\right\Vert _{\overset{\cdot }{ \tilde{B}}%
_{p}^{s,q}\left( \mathbb{R}^{n}\right) } &=&\left\Vert f\right\Vert _{%
\overset{\cdot }{\tilde{B}}_{p}^{s,q}\left( \mathbb{R} ^{n}\right) },\ \ \ \
\   \label{[3.4]} \\
&&  \notag \\
c_{1}\left\Vert f\right\Vert _{\dot{\tilde{B}}_{p}^{s,q}\left( \mathbb{R}%
^{n}\right) } &\leq &\lambda ^{\left( n/p\right) -s}\left\Vert f\left(
\lambda \left( \cdot \right) \right) \right\Vert _{\overset{\cdot }{\tilde{B}%
}_{p}^{s,q}\left( \mathbb{R}^{n}\right) }\leq c_{2}\left\Vert f\right\Vert
_{ \overset{\cdot }{\tilde{B}}_{p}^{s,q}\left( \mathbb{R}^{n}\right) }.
\label{[3.5]}
\end{eqnarray}

\begin{quote}
We can replace the norm $\left\Vert -\right\Vert _{\dot{B}_{p}^{s,q}\left( 
\mathbb{R}^{n}\right) }$ by an equivalent norm $\left\Vert -\right\Vert
^{\prime }$, which satisfies (\ref{3.4}) and improves (\ref{3.5}), by
replacing the discrete partition (\ref{3.2}) with a continuous partition
such that 
\begin{equation*}
\lambda ^{\left( n/p\right) -s}\left\Vert f\left( \lambda \left( \cdot
\right) \right) \right\Vert ^{\prime }=\left\Vert f\right\Vert ^{\prime },\
\ \ \ \text{for all } \lambda >0.\ \ \ \ \ \ 
\end{equation*}
\end{quote}

\begin{definition}
\label{3.5} 
Let $p,q\in \left[ 1,+\infty \right]$. If $s\in \left] 0,1\right]$, then we
denote by $\dot{B}_{p}^{s,q}\left(\mathbb{R}^{n}\right)$ the set of tempered
distributions $f$ in $L_{\ell oc}^{p} \ \left(\mathbb{R}^{n}\right)$ such
that $\left\Vert f\right\Vert _{\dot{B}_{p}^{s,q}\left(\mathbb{R}^{n}\right)
}<+\infty$, with 
\begin{eqnarray*}
\left\Vert f\right\Vert _{\dot{B}_{p}^{s,q}\left(\mathbb{R}^{n}\right) }
&\sim &\left( \int_{\mathbb{R} ^{n}}\frac{\omega ^{q}\left( h\right) }{%
\left\vert h\right\vert ^{sq}}\frac{\,\mathrm{d}{h}}{\left\vert h\right\vert
^{n}}\right) ^{\frac{1}{q}}\ \,, \ \ \ \text{if } s\neq 1 \\
&& \\
&\sim &\left( \int_{\mathbb{R}^{n}}\left( \frac{1}{\left\vert h\right\vert
^{s}}\left( \int_{\mathbb{R}^{n}}\left\vert f\left( x+h\right) -f\left(
x\right) \right\vert ^{p}\,\mathrm{d}{x}\right) ^{1/p}\right) ^{q}\frac{\,%
\mathrm{d}{h}}{\left\vert h\right\vert ^{n}} \right) ^{1/q}\ , \\
&& \\
\text{and } \ \ \left\Vert f\right\Vert _{\dot{B}_{p}^{1,q}\left(\mathbb{R}%
^{n}\right) } &\sim &\left( \int_{\mathbb{R} ^{n}}\left( \frac{1}{\left\vert
h\right\vert }\left( \int_{\mathbb{R}^{n}}\left\vert f\left( x+h\right)
+f\left( x-h\right) -2f\left( x\right) \right\vert ^{p}\,\mathrm{\ d }{x}%
\right) ^{\frac{1}{p}}\right) ^{q}\frac{\,\mathrm{d}{h}}{\left\vert
h\right\vert ^{n}} \right) ^{\frac{1}{q}}\ \ ,
\end{eqnarray*}

with the usual modifications for $p=+\infty$ or $q=+\infty$, \newline
where $\omega \left( h\right) =\left\Vert f\left( x+h\right)-f\left(x\right)
\right\Vert _{p}$ denotes the modulus of continuity. \newline
If $s\in \left] 1,2\right]$, then we denote by $\dot{B}_{p}^{s,q}\left(%
\mathbb{R}^{n}\right)$ the set of tempered distributions $f$ in $L_{\ell
oc}^{p}\left(\mathbb{R}^{n}\right)$ such that $\partial _{j}f\in\dot{B}%
_{p}^{s-1,q}\left(\mathbb{\ R}^{n}\right)$ for all $j=1,....,n$, and we
define its semi-norm by 
\begin{equation*}
\ \displaystyle\left\Vert f\right\Vert _{\dot{B}_{p}^{s,q}\left(\mathbb{R}%
^{n}\right) }\sim \sum_{j=1}^{n}\left\Vert \partial _{j}f\right\Vert _{ \dot{%
B}_{p}^{s-1,q}\left(\mathbb{R}^{n}\right) }.\ \ \ \ 
\end{equation*}
\end{definition}

\begin{quote}
We distinguish between the homogeneous Besov space $\dot{\tilde{B}}%
_{p}^{s,q}\left(\mathbb{R}^{n}\right)$ included in $\mathcal{S}^{\prime
}\left(\mathbb{R}^{n}\right) /\mathcal{P}_{\infty }(\mathbb{R}^{n})$ from
Definition \ref{3.4} and the space $\dot{B}_{p}^{s,q}\left(\mathbb{R}%
^{n}\right)$ from Definition \ref{3.5} included in $L_{\ell oc}^{p} \ \left(%
\mathbb{R}^{n}\right)$.
\end{quote}

\begin{lemma}[~\cite{Bo-Cr-Si-1}]
\label{3.6}  
Let $f\in L_{\ell oc}^{p}\left(\mathbb{R}^{n}\right) ,\ a\in\mathbb{R}^{n},\
\lambda>0$. Then the space $\dot{B}_{p}^{s,q}\left(\mathbb{R}^{n}\right)$
satisfies the following homogeneity properties 
\begin{equation*}
\displaystyle\left\Vert f\left( \cdot -a\right) \right\Vert _{\dot{B}%
_{p}^{s,q}\left(\mathbb{R}^{n}\right) }=\lambda ^{\left(
n/p\right)-s}\left\Vert f\left( \lambda \left( \cdot \right) \right)
\right\Vert _{ \dot{B}_{p}^{s,q}\left(\mathbb{R}^{n}\right) }=\left\Vert
f\right\Vert _{\dot{B}_{p}^{s,q}\left(\mathbb{R}^{n}\right) },\ \ \ 
\end{equation*}
moreover, $\left\Vert f\right\Vert _{\dot{B}_{p}^{s,q}\left(\mathbb{R}%
^{n}\right) }=0$ if and only if $f\in \mathcal{P}_{\left[ s\right] }\left(%
\mathbb{R}^{n}\right)$.
\end{lemma}

\begin{proof}

$-$ For $s\in \left] 0,1\right]$, take $x-a=y$, so $\,\mathrm{d}{x}=\,%
\mathrm{d}{y}$, 
\begin{equation*}
\text{hence} \ \ \ \ \ \ \ \ \ \ \ \ \ \ \ \ \left\Vert f\left(\cdot
-a\right) \right\Vert _{\dot{B}_{p}^{s,q}\left(\mathbb{R}^{n}\right)
}=\left\Vert f\right\Vert _{\dot{B}_{p}^{s,q}\left(\mathbb{R}^{n}\right) }.
\end{equation*}
$-$ For $s\in \left] 1,2\right]$, we have 
\begin{equation*}
\ \left\Vert f\left( \cdot -a\right) \right\Vert _{\dot{B}_{p}^{s,q}\left( 
\mathbb{R}^{n}\right) }=\sum_{j=1}^{n}\left\Vert \partial _{j}f\left(
\cdot-a\right) \right\Vert_{\dot{B}_{p}^{s-1,q}\left(\mathbb{R} ^{n}\right)
}=\sum_{j=1}^{n}\left\Vert \partial _{j}f\right\Vert _{ \dot{B}%
_{p}^{s-1,q}\left(\mathbb{R}^{n}\right) }=\left\Vert f\right\Vert _{\dot{B}%
_{p}^{s,q}\left(\mathbb{R}^{n}\right) } ,
\end{equation*}
and similarly by induction, for all $s>0$. \newline
\ \ \ \ \ \ \ $-$ Set $\lambda\,.\,x=y$ and $\lambda\,.\,h=t$, then 
\begin{equation*}
\,\mathrm{d}{y}\ =\lambda^{n}\,.\,\,\mathrm{d}{x} \,, \ \,\mathrm{d}{t}\
=\lambda ^{n}\,.\,\,\mathrm{d}{h}\, , \ \left\vert t \right\vert =\left\vert
\lambda \right\vert ^{n}\,.\,\left\vert h\right\vert \,,
\end{equation*}
applying Definition \ref{3.5} we obtain 
\begin{eqnarray*}
\left\Vert f\left( \lambda \left( \cdot \right) \right) \right\Vert _{\dot{B}%
_{p}^{s,q}\left(\mathbb{R}^{n}\right) } &=&\left( \int_{\mathbb{R}
^{n}}\left( \frac{1}{\left\vert \lambda \right\vert ^{-s}\left\vert
t\right\vert ^{s}}\left( \int_{\mathbb{R}^{n}}\lambda ^{-n}\left\vert
f\left( y+t\right) -f\left( y\right)\right\vert ^{p}\, \mathrm{d}{y}\right)
^{1/p}\right) ^{q}\frac{\lambda ^{-n}\,\mathrm{d}{t}}{ \left\vert \lambda
\right\vert ^{-n}\left\vert t\right\vert ^{n}}\right) ^{1/q}\ \  \\
&& \\
&=&\lambda ^{-\left( n/p\right) +s}\left\Vert f\right\Vert _{\dot{B}%
_{p}^{s,q}\left(\mathbb{R}^{n}\right) },\ 
\end{eqnarray*}

\begin{itemize}
\item[-] If $\left\Vert f\right\Vert _{\dot{B}_{p}^{1,q}\left(\mathbb{R}
^{n}\right) }=0$, then 
\begin{equation*}
f\left( x+h\right) +f\left( x-h\right) -2f\left( x\right) =0,
\end{equation*}
hence $\omega _{ p}^{2}\left( h,f\right) =0$, and so $f$ is a polynomial of
first degree.

\item[-] If $\left\Vert f\right\Vert _{\dot{B}_{p}^{s,q}\left(\mathbb{R}
^{n}\right) }=0$, then 
\begin{equation*}
\left\Vert \partial _{j}f\right\Vert _{\dot{B}_{p}^{s-1,q}\left(\mathbb{R}
^{n}\right) }=0,\text{ }\left( j=1,..n\right) ,
\end{equation*}
hence by induction $\partial _{j}f$ is a polynomial of degree $\left[ s-1%
\right] =$ $\left[ s\right] -1$, and so $f$ is a polynomial of degree $\left[
s\right]$.
\end{itemize}
\end{proof}

\begin{proposition}[~\cite{Bo-Cr-Si-1} ]
\label{3.7} 
Let $s\in\mathbb{R},\ p,q\in \lbrack 1,+\infty ]$. Then the quotient space $%
\left. \dot{B}_{p}^{s,q}\left(\mathbb{R}^{n}\right)\right/ \mathcal{P}_{%
\left[ s\right] }\ \left(\mathbb{R}^{n}\right)$ can be identified with the
homogeneous Besov space, which we can denote by $\dot{\tilde{B}}%
_{p}^{s,q}\left(\mathbb{R}^{n}\right)$, and can be taken as the set of
equivalence classes of tempered distributions modulo polynomials $f\in 
\mathcal{S}^{\prime }\left(\mathbb{R}^{n}\right) /\mathcal{P}_{\infty }\left(%
\mathbb{R} ^{n}\right)$ such that $\partial _{j}f\in \ \dot{\tilde{B}}%
_{p}^{s-1,q}\left( \mathbb{R}^{n}\right)$ for all $j=1,,..n$. Moreover, the
following expression is an equivalent norm for $\dot{\tilde{B}}%
_{p}^{s,q}\left(\mathbb{R}^{n}\right)$ 
\begin{equation*}
\ \left\Vert f\right\Vert _{\dot{\tilde{B}}_{p}^{s,q}\left(\mathbb{R}%
^{n}\right) }\approx \sum_{j=1}^{n}\left\Vert\partial_{j}f\right\Vert _{ 
\dot{\tilde{B}}_{p}^{s-1,q}\left(\mathbb{R}^{n}\right) }.
\end{equation*}
\end{proposition}

\begin{proof}

For the proof, see $\ ~\cite{Bo-Me}$.
\end{proof}

\subsection{\textbf{Realizations of Homogeneous Besov Spaces}}

\begin{definition}
\label{3.8} 
Let $\sigma :\dot{B}_{p}^{s,q}\left( \mathbb{R}^{n}\right) \mathbf{\
\longrightarrow }\left. \mathcal{S}^{\prime }\left( \mathbb{R}^{n}\right)
\right/ \mathcal{P}_{m}\left( \mathbb{R}^{n}\right)$ be a continuous linear
map such that $\sigma \left( f\right) =\left[ f\right]$ is the equivalence
class of $f$ modulo $\mathcal{P}_{m}\left( \mathbb{R}^{n}\right)$. Then for
all $f\in \dot{B}_{p}^{s,q}\left( \mathbb{R}^{n}\right)$, we say that $%
\sigma $ is a \textbf{realization} modulo $\mathcal{P}_{m}\left( \mathbb{R}%
^{n}\right)$ of $\dot{B}_{p}^{s,q}\left( \mathbb{R}^{n}\right)$. It is a
linear isomorphism from $\dot{B}_{p}^{s,q}\left( \mathbb{R}^{n}\right)$ onto
its image, such that the space $\sigma \left( \dot{B} _{p}^{s,q}\left( 
\mathbb{R}^{n}\right) \right)$ endowed with the norm 
\begin{equation*}
\ \left\Vert \sigma \left( f\right) \right\Vert =\left\Vert f\right\Vert _{ 
\dot{B}_{p}^{s,q}\left( \mathbb{R}^{n}\right) }\ 
\end{equation*}
becomes a Banach space.
\end{definition}

If $f\in \left. \mathcal{S}^{\prime }\left(\mathbb{R}^{n}\right) \right/ 
\mathcal{P}_{\infty }\left(\mathbb{R}^{n}\right)$ and if the series $%
\sum_{j\in\mathbb{Z}}Q_{j}f$ converges in $\left. \mathcal{S}^{\prime }\left(%
\mathbb{R}^{n}\right) \right/ \mathcal{P}_{m}\left(\mathbb{R} ^{n}\right)$,
then we have

\begin{center}
$\displaystyle\ \sigma _{m}\left( f\right) =$ $\sum_{j\in\mathbb{Z}
}Q_{j}f\in \left. \mathcal{S}^{\prime }\left(\mathbb{R}^{n}\right) \right/ 
\mathcal{P}_{m}\left(\mathbb{R}^{n}\right) .$
\end{center}

\begin{definition}
\label{3.9} 
We say that a tempered distribution $f\in\mathcal{S}^{\prime }\left(\mathbb{R%
}^{n}\right)$ \textbf{tends to $0$ at infinity} if we have 
\begin{equation*}
\ \lim_{\lambda \rightarrow 0}f\left( \frac{\cdot }{\lambda } \right)=0,\ 
\text{in } \mathcal{S}^{\prime }\left(\mathbb{R}^{n}\right) .
\end{equation*}
The set of such distributions is denoted by $\widetilde{C} _{0}\left(\mathbb{%
R}^{n}\right) .$
\end{definition}

\begin{quote}
If $C\left(\mathbb{R}\right)$ denotes the set of real-valued continuous
functions and $C_{b}\left(\mathbb{R}\right)$ the Banach space of bounded
functions in $C\left(\mathbb{R}\right)$ endowed with the sup norm, and $%
C_{0}\left(\mathbb{R}\right)$ the Banach subspace of functions in $%
C_{b}\left(\mathbb{R} \right)$ with limit zero at infinity, then the
following distributions tend to $0$ at infinity.
\end{quote}

\begin{itemize}
\item[--] Functions belonging to $C_{0}\left(\mathbb{R}^{n}\right)$ or to $%
L^{p}(\mathbb{R}^{n})$, $1\leq p<+\infty$.

\item[--] Bounded Borel measures.

\item[--] Derivatives of bounded continuous functions.

\item[--] Derivatives of distributions belonging to $\ \widetilde{C}%
_{0}\left(\mathbb{R}^{n}\right)$.
\end{itemize}

\begin{proposition}[~\cite{Bo-Me}]
 \label{3.10}
Let $1\leq p<\infty$, such that 
\begin{equation*}
(\displaystyle0<s<1+ \frac{1}{p} \quad \text{and} \quad 1\leq q\leq +\infty
)\quad \text{or} \quad ( \displaystyle s=1+\frac{1}{p} \quad \text{and}\quad
q=1){}.
\end{equation*}
If we denote by $\mathcal{\dot{B}}_{p}^{s,q}\left(\mathbb{R}^{n}\right)$ the
set of tempered distributions $f$ such that 
\begin{equation*}
\left[ f\right] \in\dot{B}_{p}^{s,q}\left(\mathbb{R}^{n}\right) \ \text{and}
\ \ \partial _{j}f\in\widetilde{C}_{0}, \text{ for } j=1,....n,
\end{equation*}
then every element of $\dot{B}_{p}^{s,q}\left(\mathbb{R}^{n}\right)$ admits
a representative in $\mathcal{\ \dot{B}}_{p}^{s,q}\left(\mathbb{R}%
^{n}\right) $ unique up to addition of a constant, and the space $\mathcal{%
\dot{B}}_{p}^{s,q}\left(\mathbb{R }^{n}\right)$ thus defined can be endowed
with the semi-norm $\left\Vert -\right\Vert _{ \dot{B}_{p}^{s,q}\left(%
\mathbb{R}^{n}\right)}$.
\end{proposition}

\subsection{\textbf{Statement of Results}}

In this subsection, we compare the different spaces $B_{p}^{s,q}$, $\dot{%
\tilde{B}}_{p}^{s,q}\approx \dot{B}_{p}^{s,q}$, $\mathcal{\dot{B}}_{p}^{s,q}$%
, concerning functional calculus to solve the composition problem by
specifying the conditions related to each of them and presenting two
comparative theorems \ref{3.12} and \ref{3.13}.

\begin{definition}[~\cite{Bo-Me} ]
\label{3.11} 
Let $p\in \lbrack 1,+\infty \lbrack$ and $J$ an interval of $\mathbb{R}$.
Then we denote by $\mathcal{U}_{p}\left(J\right)$ the set of measurable
functions $f:\mathbb{R}\mathbf{\longrightarrow }\mathbb{R}$ such that 
\begin{equation*}
\sup_{\left\vert h\right\vert \leq t}\ \left\vert f\left(x+h\right)-f\left(
x\right) \right\vert \ \text{is measurable on } J, \ \text{for all } t>0\,,
\end{equation*}
endowed with the norm 
\begin{equation*}
\left\Vert f\right\Vert _{\mathcal{U}_{p}\left(J\right)}^{p}=\sup_{t\
>0}t^{-1}\int_{J}\ \ \sup_{\left\vert h\right\vert \leq t}\left\vert f\left(
x+h\right)-f\left( x\right)\right\vert ^{p}\,\mathrm{d}{\ x}<+\infty .
\end{equation*}
\end{definition}

\begin{quote}
We say that a continuous function $f$ belongs to $U_{p}^{1}\left( \mathbb{R}%
\right)$ if there exists a bounded Borel function $h\in \mathcal{U}_{p}\left(%
\mathbb{R}\right)$ such that
\end{quote}
\begin{equation*}
\displaystyle\ f\left(x\right) -f\left( 0\right)=\int_{0}^{x}h\left(
t\right) \,\mathrm{d}{t},\ \ \ \ \ \ \text{for all } x\in\mathbb{R},
\end{equation*}
and for every bounded Borel measurable function $h$ of class $\mathcal{U}%
_{p}\left(\mathbb{R}\right)$, we endow $U_{p}^{1}\left( \mathbb{R}\right)$
with the semi-norm 
\begin{equation*}
\left\Vert f\right\Vert _{U_{p}^{1}}=\inf \left\{ \sup_{\mathbb{R}
}\left\vert h\right\vert +\left\Vert h\right\Vert _{\mathcal{U}_{p}},\
h=f^{\prime }\ \left( p.p\right) \right\} \sim \left\vert
f\left(0\right)\right\vert +\left\Vert f^{\prime }\right\Vert _{\infty
}+\left\Vert f^{\prime }\right\Vert _{U_{p}\left(\mathbb{R}\right) }.
\end{equation*}
And just as we defined $BV_{p}(\mathbb{R})$ from $\mathcal{V}_{p}\left(%
\mathbb{R} \right)$ (cf. Definition \ref{1.19}), we also define $U_{p}\left(%
\mathbb{R}\right)$ from $\mathcal{U}_{p}\left(\mathbb{R}\right)$, taking
functions that are almost everywhere equal to at least one element of $%
\mathcal{U}_{p}\left(\mathbb{R}\right)$, and we endow $U_{p}\left( \mathbb{R}
\right)$ with the semi-norm 
\begin{equation*}
\displaystyle\left\Vert f\right\Vert _{U_{p}\left(\mathbb{R}\right) }=\inf
\left\{ \ \left\Vert g\right\Vert _{\mathcal{U}_{p}\left(\mathbb{R}\right)
}\,;\quad g\in \mathcal{U}_{p}\left(\mathbb{R}\right)\,, \quad g=f\ \ \
\left( p.p\ \right) \right\}.
\end{equation*}
\begin{theorem}[~\cite{Bo-Cr-Si-1}]
\label{3.12}  
Let 
\begin{equation*}
p\in ]1,+\infty \lbrack\,,\ s\in ]0,1+\frac{1}{p}[\,, \ q\in \lbrack
1,+\infty ]\,, \ f\in U_{p}^{1}\left(\mathbb{R}\right).
\end{equation*}

\begin{itemize}
\item[(i)] If $f(0)=0$, then $\displaystyle T_{f}\left( B_{p}^{1+\left(
1/p\right),1}\left(\mathbb{R}^{n}\right)\right) \subseteq B_{p}^{1+\left(
1/p\right) ,\infty }\left(\mathbb{R}^{n}\right)$, 
\begin{equation*}
\ \text{and } \left\Vert f\circ g\right\Vert _{B_{p}^{1+\left(
1/p\right),\infty}\left(\mathbb{R}^{n}\right) }\leq c\left\Vert f\right\Vert
_{U_{p}^{1}\left(\mathbb{R}\right) }\left\Vert g\right\Vert
_{B_{p}^{1+\left( 1/p\right) ,1}\left(\mathbb{R}^{n}\right) } ,\ \text{for
all } g\in B_{p}^{1+\left( 1/p\right) ,1}\left(\mathbb{R} ^{n}\right).
\end{equation*}

\item[(ii)] If $f(0)=0$, then $T_{f}\left( B_{p}^{s,q}\left( \mathbb{R}%
^{n}\right) \right) \subseteq B_{p}^{s,q}\left(\mathbb{R} ^{n}\right)$, 
\begin{equation*}
\text{and } \left\Vert f\circ g\right\Vert _{B_{p}^{s,q}\left(\mathbb{R}
^{n}\right)}\leq c\left\Vert f\right\Vert _{U_{p}^{1}\left(\mathbb{R}
\right)}\left\Vert g\right\Vert _{B_{p}^{s,q}\left(\mathbb{R}^{n}\right) }
,\ \ \ \ \ \text{for all } g\in B_{p}^{s,q}\left(\mathbb{R}^{n}\right).
\end{equation*}

\item[(iii)] If $g\in \dot{B}_{p}^{1+\frac{1}{p},1}\left(\mathbb{R}
\right)\,,$ then $f\circ g\in \dot{B}_{p}^{1+\frac{1}{p},\infty }\left(%
\mathbb{R}\right) ,\ $ 
\begin{equation*}
\text{and } \left\Vert f\circ g\right\Vert _{\dot{B}_{p}^{1+\frac{1}{p}
,\infty}\left( \mathbb{R}\right) }\leq c_{p}\left\Vert f\right\Vert
_{U_{p}^{1}\left( \mathbb{R}\right) }(\left\Vert g\right\Vert _{\dot{B}
_{p}^{1+\frac{1}{p} ,1}\left( \mathbb{R}\right) }+\left\vert L_{g^{\prime
}}\right\vert ),\ \ \left( c_{p}>0\right) .
\end{equation*}

\item[(iv)] If $g\in BV_{p}^{1}\left(\mathbb{R}\right) ,$ then $f\circ g\in 
\dot{B}_{p}^{1+\frac{1}{p},\infty }\left(\mathbb{R} \right) ,$ 
\begin{equation*}
\text{and } \ \ \left\Vert f\circ g\right\Vert _{\dot{B}_{p}^{1+\frac{1}{p}
,\infty}\left(\mathbb{R}\right) }\leq c_{p}\left\Vert f\right\Vert
_{U_{p}^{1}\left(\mathbb{R}\right) } \left\Vert g^{\prime }\right\Vert
_{BV_{p}\left(\mathbb{R} \right) },\ \ \ \left( c_{p}>0\right).
\end{equation*}
\end{itemize}
\end{theorem}
\begin{proof}
For the proof, see ~\cite{Bo-Cr-Si-1}. Theorem \ref{3.12} can be generalized
to spaces $B_{p}^{s,q}\left(\mathbb{R}^{n}\right)$, but we do not know a
generalization to the n-dimensional case for spaces $BV_{p}^{1}$.
\end{proof}

\begin{theorem}[~\cite{Bo-Me}]
\label{3.13}  
Let 
\begin{equation*}
1\leq p<\infty ,\ q\in \left[ 1,+\infty \right] ,\ 0<s<1+\frac{1}{p} \ \ 
\text{and} \ \ f\in U_{p}^{1}\left(\mathbb{R}\right).
\end{equation*}

\begin{itemize}
\item[(i)] If $g\in \mathcal{\dot{B}}_{p}^{1+\frac{1}{p},1}\left(\mathbb{R }%
^{n}\right)\,,$ then we have 
\begin{eqnarray*}
\ \ \left[ f{\large \circ }g\right] &\in & \dot{B}_{p}^{1+\frac{1}{p}
,\infty}\left(\mathbb{R}^{n}\right) ,\ \partial _{j}\left( f{\large \circ }
g\right)\in \widetilde{C}_{0}\left(\mathbb{R}^{n}\right) ,\ \left(
j=1,...,n\right) , \\
&&\ \ \ \ \  \\
\text{and } \ \ \ \left\Vert f\ {\large \circ }g\right\Vert _{ \dot{B}%
_{p}^{1+\frac{1}{ p},\infty }\left(\mathbb{R}^{n}\right) } &\leq
&c\left\Vert f\right\Vert _{U_{p}^{1}}\left\Vert g\right\Vert _{ \dot{B}%
_{p}^{1+\frac{1}{p},1}\left( \mathbb{R}^{n}\right) },\ \ \ \left( c>0\right).
\end{eqnarray*}

\item[(ii)] If $g\in \mathcal{\dot{B}}_{p}^{s,q}\left(\mathbb{R} ^{n}\right) 
$ , then we have 
\begin{equation*}
f\circ g\in \mathcal{\dot{B}}_{p}^{s,q}\left(\mathbb{R}^{n}\right) , \ \text{%
and } \ \left\Vert f\ {\large \circ }g\right\Vert_{ \dot{B}_{p}^{s,q}\left(%
\mathbb{R}^{n}\right) }\leq c\left\Vert f\right\Vert _{U_{p}^{1}}\left\Vert
g\right\Vert _{ \dot{B}_{p}^{s,q}\left(\mathbb{R}^{n}\right) },\ \ \ \ \ \ \
\ \left( c>0\right).
\end{equation*}
\end{itemize}
\end{theorem}

\begin{proof}

For the proof, see ~\cite{Bo-Me}, page 13.
\end{proof}

\section{\textbf{Peetre's Theorem}}

\begin{lemma}[~\cite{GIR}]
\label{3.14} 
Let $\phi :\left[ a,b\right] \mathbf{\longrightarrow }\mathbb{R}$ be an
absolutely continuous function.\\ Then: 
\begin{eqnarray*}
\nu _{1}\left( \phi ,\left[ a,b\right] \right)&=&\int_{a}^{b}\left \vert
\phi ^{\prime }\left( x\right) \right\vert \,\mathrm{d}{x}\ ,\ \text{ } 
\text{where } \ \ \text{ } \nu _{ 1}\ (f,I) =\sup_{\left\{ t_{i}\right\}
\subset I} \ \sum_{k=1}^{N}\left\vert f\left( t_{k}\right) - f\left( t_{k-
1}\right) \right\vert .
\end{eqnarray*}
\end{lemma}

\begin{theorem}[~\cite{Pee}]
\label{3.15}  
Let $1\leq p_{0}\,,\,p_{1}\leq \infty\quad,\quad \ 0<q_{0}\,,\,q_{1}\leq
\infty\,,\quad 0<\theta <1$.

\begin{itemize}
\item[(i)] If $\displaystyle s=\left( 1-\theta \right) s_{0}+\theta s_{1}\,,
\ \frac{ 1}{p}=\frac{1-\theta }{p_{0}}+\frac{\theta }{p_{1}}\,, \ \frac{1}{q}%
=\frac{1-\theta }{q_{0}}+\frac{\theta }{q_{1}}\,, $ then 
\begin{equation*}
B_{p}^{s,\min \left( q,p\right) }\left(\mathbb{R}^{n}\right)
\hookrightarrow\left( B_{p_{0}}^{s_{0},q_{0}}\left( \mathbb{R}
^{n}\right)\,,\,B_{p_{1}}^{s_{1},q_{1}}\left(\mathbb{R}^{n}\right) \right)
_{\theta,p}\hookrightarrow B_{p}^{s,\max \left(q,p\right) }\left(\mathbb{R}
^{n}\right).
\end{equation*}

\item[(ii)] If $s=\left( 1-\theta \right) s_{0}+\theta s_{1}\,,\ 0<\theta <1
\,,\ r\in\mathbb{R}\,,$ then 
\begin{equation*}
\left( B_{p}^{s_{0},q_{0}}\left(\mathbb{R}^{n}\right)\,,
\,B_{p}^{s_{1},q_{1}}\left(\mathbb{R}^{n}\right) \right)
_{\theta,r}=B_{p}^{s,r}\ \left(\mathbb{R}^{n}\right).
\end{equation*}

\item[(iii)] If $\displaystyle\frac{1}{q}=\frac{1-\theta }{q_{0}}+ \frac{%
\theta }{q_{1}}\,,$ then 
\begin{equation*}
\ \ \left( B_{p}^{s_{0},q_{0}}\left(\mathbb{R}^{n}\right)\,,
\,B_{p}^{s_{1},q_{1}}\left(\mathbb{R}^{n}\right) \right)
_{\theta,q}=B_{p}^{s,q}\ \left(\mathbb{R}^{n}\right).
\end{equation*}
\end{itemize}
\end{theorem}

\begin{proof}

For the proof, see Peetre's monograph, ~\cite{Pee} Chap 5, P107. Theorem \ref%
{3.15} remains true for homogeneous Besov spaces.
\end{proof}

\begin{proposition}[~\cite{Bo-Cr-Si-1}]\label{3.16}\text{ \ } \vspace{-20pt}
\begin{itemize}
\item[(i)] Let $p\in ]1,+\infty \lbrack$. If $g$ is an element of $\dot{B}%
_{p}^{\frac{1}{p},1}\left(\mathbb{R}\right)$,\\ then $g$ is a continuous
function and the limit $L_{g}=\lim_{x \rightarrow \infty }g\left(x\right)$
exists in $\mathbb{R}$. \newline
In particular, each element of $\dot{B}_{p}^{\frac{1}{p} ,1}\left(\mathbb{R}%
\right)$ is congruent modulo $\mathcal{P}_{0}$ to exactly one element of $%
C_{0}(\mathbb{R})$.

\item[(ii)] Let $p\in \lbrack 1,+\infty \lbrack$. Then the space $\dot{B}%
_{p}^{\frac{1}{p},1}\left(\mathbb{R}\right) \cap C_{0}(\mathbb{R})$, endowed
with the restriction of the semi-norm $\left\Vert \cdot\right\Vert _{\dot{B}%
_{p}^{\frac{1}{p},1}\left(\mathbb{R} \right) }$, is a Banach space isometric
to $\dot{\tilde{B}}_{p}^{1/p,1}\left(\mathbb{R}\right)$.
\end{itemize}
\end{proposition}

\begin{theorem}[~\cite{Bo-Cr-Si-1}]
\label{3.17} 
Let $p\in ]1,+\infty \lbrack$. Then the following continuous inclusions
hold: 
\begin{eqnarray*}
\ \dot{\tilde{B}}_{p}^{1/p{\small ,}1}\left(\mathbb{R}\right)
&\hookrightarrow &\left( L^{\infty }\left(\mathbb{R}\right) ,BV_{1}\left( 
\mathbb{R}\right) \right) _{1/p,p}\ =\left( BV_{\infty }\left(\mathbb{R}
\right) ,BV_{1}\left(\mathbb{R}\right) \right) _{1/p,p} \\
&& \\
&\hookrightarrow &BV_{p}\left(\mathbb{R}\right) \hookrightarrow U_{p}\left( 
\mathbb{R}\right) \hookrightarrow \dot{\tilde{B}}_{p}^{1/p,\ \infty}\left( 
\mathbb{R}\right) .
\end{eqnarray*}
\end{theorem}

\begin{proof}

\begin{itemize}
\item[(1)] Applying Theorem \ref{3.15}, (i) with 
\begin{equation*}
\theta =s=1/p,\ p_{0}=\infty ,\ s_{0}=0,\ s_{1}=q_{0}=q_{1}=p_{1}=1\,,\qquad 
\text{then}
\end{equation*}
\begin{equation*}
\dot{\tilde{B}}_{p}^{1/p{\small ,}1}\left(\mathbb{R}\right) \hookrightarrow
( \dot{\tilde{B}}_{\infty }^{0{\small ,}1}\left(\mathbb{R}\right) ,\dot{ 
\tilde{ B}}_{1}^{1{\small ,}1}\left(\mathbb{R}\right)
)_{1/p,p}\hookrightarrow \dot{ \tilde{B}}_{p}^{1/p{\small ,}p}\left(\mathbb{R%
}\right) ,
\end{equation*}
and according to Proposition \ref{3.16}, (ii) we have 
\begin{eqnarray*}
\dot{B}_{p}^{1/p{\small ,}1}\left(\mathbb{R}\right) \cap C_{0}\left(\mathbb{%
\ R }\right) &\approx &\dot{\tilde{B}}_{p}^{1/p{\small ,}1}\left(\mathbb{R}
\right) \ \ \text{and} \ \ \dot{B}_{\infty }^{0,1}\left(\mathbb{R} \right)
\cap C_{b}\left(\mathbb{R}\right) \approx \dot{\tilde{B}}_{\infty }^{0%
{\small ,} 1}\left( \mathbb{R}\right) , \\
&& \\
\text{hence } \ \ \ \ \ \ \ \ \ \ \ \ \ \ \ \ \ \ \ \ \ \dot{B}_{p}^{1/p%
{\small ,}1}\left(\mathbb{R}\right) \cap C_{0}\left(\mathbb{\ R }\right)
&\hookrightarrow &\left( \dot{B}_{\infty }^{0,1}\left(\mathbb{R} \right)
\cap C_{b}\left(\mathbb{R}\right) ,\dot{B}_{1}^{1,1}\left(\mathbb{R} \right)
\cap C_{0}\left(\mathbb{R} \right) \right) _{1/p,p}.
\end{eqnarray*}
-- Let us prove the injections 
\begin{equation*}
\dot{\tilde{B}}_{1}^{1{\small ,} 1}\left(\mathbb{R}\right) \hookrightarrow 
\dot{W}^{1,1}\left(\mathbb{R}\right) \hookrightarrow BV_{1}\left(\mathbb{R}
\right).
\end{equation*}
-- If $f\in \dot{\tilde{B}}_{1}^{1{\small ,}1}\left(\mathbb{R}\right) $,
then using Bernstein's inequality we obtain 
\begin{eqnarray*}
\left\Vert f^{\prime }\right\Vert _{1} &\leq &cR\left\Vert f\right\Vert
_{1}=cR\left\Vert \sum_{j\in\mathbb{Z} }Q_{j}f\right\Vert _{1} \leq cR\sum_{j\in\mathbb{Z}}2^{j}\left\Vert Q_{j}f\right\Vert
_{1}=C\left\Vert f\right\Vert _{\dot{\tilde{B}}_{1}^{1{\small ,}1}}, \\
\\
\text{where } \ \ \ \ \ \ \ \ \ \ \ \ supp(\hat{f}) &\subset &\left\{ \xi
\in \mathbb{R}^{n}:\left\vert \xi \right\vert \leq R\right\},
\end{eqnarray*}
hence $\ \ \left\Vert f\right\Vert _{\dot{W}^{1,1}\left(\mathbb{R} \right)
}\leq C\left\Vert f\right\Vert _{\dot{\tilde{B}}_{1}^{1{\small ,} 1}\left(%
\mathbb{R}\right) }$, and so $\dot{\tilde{B}}_{1}^{1{\small ,}1}\left(%
\mathbb{R}\right) \hookrightarrow \dot{W}^{1,1}\left( \mathbb{R}\right)$. --
If $f\in \dot{W}^{1,1}$ $\left(\mathbb{R} \right) $, then $f$ is absolutely
continuous, and by Lemma \ref{3.14} we have 
\begin{equation*}
\ \ \ \ \ \left\Vert f\right\Vert _{BV_{1}\left(\mathbb{R}\right) }=\nu
_{1}\left( \tilde{f},\left(\mathbb{R} \right) \right) +\sup_{x\in \mathbb{R}%
}\left\vert \tilde{f}\left( x\right) \right\vert \leq 2\nu _{1}\left( \tilde{%
f},\mathbb{R}\right) \leq 2\int_{\mathbb{R} }\left\vert \tilde{f}^{\prime
}\left( x\right) \right\vert \,\mathrm{d}{x}\ \leq 2\left\Vert f\right\Vert
_{\dot{W}^{1,1} \left(\mathbb{R}\right) }, \ 
\end{equation*}
hence $\dot{W}^{1,1}\left(\mathbb{R}\right) \hookrightarrow BV_{1}\left(%
\mathbb{R}\right)$, and therefore $\dot{\tilde{B}}_{1}^{1{\small ,}1}\left(%
\mathbb{R}\right) \hookrightarrow BV_{1}\left(\mathbb{R} \right)$. -- We
also prove that $\dot{\tilde{B}}_{\infty }^{0,1}\left(\mathbb{R} \right)
\hookrightarrow L^{\infty }$ $\left(\mathbb{R}\right)$, because we have 
\begin{equation*}
\ \left\Vert f\right\Vert _{\infty }=\left\Vert \sum_{j\in\mathbb{Z}
}Q_{j}f\right\Vert _{\infty }\leq \ \sum_{j\in\mathbb{Z}}\left\Vert
Q_{j}f\right\Vert _{\infty }=\left\Vert f\right\Vert _{\overset{ \cdot }{ 
\tilde{B}}_{\infty }^{0,1}\left(\mathbb{R}\right) }\ ,
\end{equation*}
-- applying the interpolation Theorem \ref{3.3} we obtain 
\begin{eqnarray*}
\text{ \ \ \ }(\dot{\tilde{B}}_{\infty }^{0{\small ,}1}\left(\mathbb{R}
\right) ,\dot{\tilde{B}}_{1}^{1{\small ,}1}\left(\mathbb{R}\right) )_{1/p,p}
&\hookrightarrow &\left( L^{\infty }\left(\mathbb{R} \right) ,BV_{1}\left( 
\mathbb{R}\right) \right) _{1/p,p}, \\
&& \\
\text{hence } \ \ \ \ \ \ \ \ \ \ \ \ \ \ \ \ \ \ \ \ \ \ \ \ \ \ \text{ } 
\dot{\tilde{B}} _{p}^{1/p{\small ,}1}\left(\mathbb{R}\right)
&\hookrightarrow &\left( L^{\infty }\left(\mathbb{R}\right) ,BV_{1}\left( 
\mathbb{R}\right) \right) _{1/p,p}.
\end{eqnarray*}

\item[(2)] If $E$ is the closure of $BV_{1}\left(\mathbb{R}\right)$ with
respect to $\ L^{\infty }\left(\mathbb{R}\right)$, then 
\begin{equation*}
\left( L^{\infty }\left(\mathbb{R}\right) ,BV_{1}\left(\mathbb{R} \right)
\right) _{1/p,p}=\left( E\ ,BV_{1}\left(\mathbb{R}\right) \right) _{1/p,p},
\end{equation*}
and since $E$ $\hookrightarrow BV_{\infty }\left(\mathbb{R}\right)$, then 
\begin{eqnarray*}
\left( E,BV_{1}\left(\mathbb{R}\right) \right) _{1/p,p} &\hookrightarrow
&\left( BV_{\infty }\left(\mathbb{R} \right) ,BV_{1}\left(\mathbb{R}\right)
\right) _{1/p,p}, \\
&& \\
\text{hence } \ \ \ \ \ \ \ \ \ \text{\ \ \ \ \ \ \ \ \ }\left( L^{\infty
}\left(\mathbb{R}\right) ,BV_{1}\left(\mathbb{R}\right) \right) _{1/p,p}
&\hookrightarrow &\left( BV_{\infty }\left(\mathbb{R}\right) ,BV_{1}\left(%
\mathbb{R}\right) \right) _{1/p,p}.
\end{eqnarray*}

\item[(3)] We define the norms of the space $l_{N}^{\ p}$ of finite
sequences $\left\{ t_{i}\right\} _{1\leq i\leq N}\subset\mathbb{R}$ by 
\begin{equation*}
\left\Vert t_{i}\right\Vert _{l_{N}^{\ p}}=(\sum_{i=1}^{i=N}\left\vert
t_{i}\right\vert ^{p})^{\frac{1}{p} }\ \ \text{and} \ \ \left\Vert
t_{i}\right\Vert _{l_{N}^{\infty }}=\sup_{1\leq i\leq N}\left\vert
t_{i}\right\vert .
\end{equation*}
Consider $t_{0}<t_{1}<...<t_{N}$, and let the continuous function $U$
defined by 
\begin{eqnarray*}
U\,:\, \ \ BV_{\infty }\left(\mathbb{R}\right) \ \ &\mathbf{\longrightarrow }%
&\ l\ _{N}^{\infty }\left(\mathbb{R} \right) \\
&& \\
f\ \ &\mathbf{\longmapsto }&\ U\left( f\right) =\left( \ \tilde{f}\ \left(
t_{k}\right) -\tilde{f}\ (t_{k\ -1})\right) _{1\ \leq \ k\ \leq \ N}\ ,
\end{eqnarray*}
we then have 
\begin{eqnarray*}
\left\Vert {\small U}\left( f\right) \right\Vert _{l_{ N}^{\ {\large \infty }
}} &=&\sup_{1\leq i\leq N}\left\vert \ \tilde{f}\ \left( t_{i}\right) -%
\tilde{f}\ (t_{i\ -1})\right\vert \leq \sup_{1\leq i\leq N}\ \left(
\left\vert \tilde{f}\ \left( t_{i}\right) \right\vert +\left\vert \tilde{f}\
(t_{i\ -1})\right\vert \right) \\
&& \\
&\leq &2\sup \left\vert \tilde{f}\right\vert \ =2\left\Vert f\right\Vert
_{BV_{\infty }\left(\mathbb{R}\right) },\  \\
&& \\
\ \ \ \ \ \ \ \ \ \ \ \ \text{and } \ \ \ \ \ \ \ \ \ \ \ \ \ \ \ \ \ \
\left\Vert {\small U}\left( f\right) \right\Vert _{l_{ N}^{\ {\large 1}}}\
&=&\sum_{k=1}^{k=N}\left\vert \ \tilde{f}\ \left( t_{k}\right) - \tilde{f}\
(t_{k\ -1})\right\vert \leq \nu _{1}\left( \tilde{f} \right) \leq \left\Vert
f\right\Vert _{BV_{{\large 1}}\left(\mathbb{R} \right) }.
\end{eqnarray*}
We have $(L^{\infty }(X),L^{1}(X))_{\frac{1}{p}\,,\,p}$ $=$ $L^{p}(X)$, and
taking as a particular case the spaces $l_{ N}^{\ p}$, we obtain by the
interpolation Theorem \ref{3.3} 
\begin{equation*}
\left\Vert {\small U}\left( f\right) \right\Vert _{l_{N}^{\ p}}\ {\small %
\leq } c_{p}\left\Vert f\right\Vert _{\left( BV_{\infty }\left(\mathbb{R}
\right) ,BV_{{\large 1}}\left( \mathbb{R}\right) \right) _{\frac{1}{p}
\,,\,p}}\ ,
\end{equation*}
and since $\displaystyle\left\Vert f\right\Vert _{BV_{p}\left(\mathbb{R}
\right) }$ $=\inf_{a}\left\{ a>0:\left\Vert {\small U(}\tilde{f}
)\right\Vert _{l_{N}^{\ p}}\leq a\right\} \,, $ then 
\begin{equation*}
\left\Vert f\right\Vert _{BV_{p}\left(\mathbb{R}\right) }\leq
c_{p}\left\Vert f\right\Vert _{\left( BV_{\infty }\left( \mathbb{R}\right)
,BV_{{\large 1}}\left(\mathbb{R}\right) \right) _{\frac{1}{p}\,,\,p}},
\end{equation*}
which gives $\left( BV_{\infty }\left(\mathbb{R}\right) ,BV_{{\large 1 }%
}\left(\mathbb{R} \right) \right) _{\frac{1}{p}\,,\,p}\hookrightarrow
BV_{p}\left(\mathbb{R}\right)$.

\item[(4)] 
\begin{eqnarray*}
\text{We have } \int_{\mathbb{R} }\ \sup_{\left\vert h\right\vert \leq
t}\left\vert f\left( x+h\right) -f\left( x\right) \right\vert ^{p}\,\mathrm{d%
}{x} &=&\sum_{m\in\mathbb{Z} }\int_{mt}^{\left( m+1\right)
t}\sup_{\left\vert h\right\vert \leq t} \left\vert f\left( x+h\right)
-f\left( x\right) \right\vert ^{p}\,\mathrm{d}{x} \\
&& \\
&=&\int_{0}^{t}\ \sum_{m \in\mathbb{Z} }\ \sup_{ \left\vert h\right\vert
\leq t}\left\vert f\left( y+mt+h\right) -f\left( y+mt\right) \right\vert
^{p}\,\mathrm{d}{y}\,.
\end{eqnarray*}
For each $m\in\mathbb{Z}$, $y\in\mathbb{R}$, there exists $h_{m}(y)\in
\lbrack -t,t]$, such that 
\begin{equation*}
\left\vert f\left( y+mt+h_{m}\left( y\right) -f\left( y+mt\right) \right)
\right\vert ^{p}\geq \ \sup_{\left\vert h\right\vert \leq t} \left\vert
f\left( y+mt+h\right) -f\left( y+mt\right) \right\vert ^{p}-\frac{ {\LARGE %
\varepsilon }}{4}2^{-\left\vert m\right\vert }\ .
\end{equation*}
Since the family of intervals $\left\{ \left] y+kt,y+kt+h\left( y\right) %
\right[ \right\}$ is disjoint, then 
\begin{equation*}
\hspace{-1in}\int_{0}^{t}\sum_{m \in\mathbb{Z}}\sup_{ \left\vert
h\right\vert {\small \leq }t}\left\vert f\left(y+mt+h\right) -f\left(
y+mt\right) \right\vert ^{p}\,\mathrm{d}{y} \leq \int_{0}^{t}\sum_{m \in%
\mathbb{Z}}\left\{ \left\vert f\left( y+mt+h_{m}\left( y\right) \right)
-f\left(y+mt\right) \right\vert ^{p}+\frac{{\LARGE \varepsilon }}{4}%
2^{-\left\vert m\right\vert }\right\} \, \mathrm{d}{y}
\end{equation*}
\begin{equation*}
\leq \int_{0}^{t}\left( 2\nu _{p}^{p}\left( f\right) +\varepsilon \right) \,%
\mathrm{d}{y}=t\left( \left( 2\nu _{p}^{p}\left( f\right) +\varepsilon
\right) \right)\,.
\end{equation*}
Since $\varepsilon >0$ is arbitrary, we obtain the inequality $\displaystyle%
\left\Vert f\right\Vert _{\mathcal{U}_{p}\left(\mathbb{R}\right) }\leq
2^{1/p}\left\Vert f\right\Vert _{\mathcal{V }_{p}\left(\mathbb{R}\right) \ }$%
, which gives $$\mathcal{V}_{p}\ \left(\mathbb{R}\right) \hookrightarrow 
\mathcal{U}_{p}\left( \mathbb{R} \right),$$ and we deduce that $$%
BV_{p}\left(\mathbb{R}\right) \hookrightarrow U_{p}\ \left(\mathbb{R}
\right). $$

\item[(5)] We have for all $g=f$, (a.e.) 
\begin{eqnarray*}
\left\Vert f\right\Vert _{\dot{\tilde{B}}_{p}^{1/p,\infty }\left( \mathbb{R}
\right) } &=&\sup_{\left\vert h\right\vert \ \geq 1}\ \left[ h^{- \frac{1}{p}%
}\left\Vert \tau _{-h}f-f\right\Vert _{L^{p}}\frac{1}{\left\vert
h\right\vert }\right] , \\
&& \\
&\leq &\left[ \sup_{t\ >0}\ t^{-1}\int_{\mathbb{R}}\ \left( \sup_{\left\vert
h\right\vert \leq t}\ \left\vert g\left( x+h\right) -g\left( x\right)
\right\vert ^{p}\right) \,\mathrm{d}{x}\right] ^{\frac{1}{ p }}=\left\Vert
g\right\Vert _{\mathcal{U}_{p}\left( \mathbb{R}\right) }, \ \text{ } \\
&& \\
\text{hence } \ \ \ \ \ \ \ \ \ \ \ \ \ \ \ \ \ \ \ \ \left\Vert
f\right\Vert _{\dot{\tilde{B}}_{p}^{1/p,\infty }\left( \mathbb{R}\right) }
&\leq &\inf \ \left\{ \ \left\Vert g\right\Vert _{\mathcal{U}_{p}\left( 
\mathbb{R}\right) }\,;\quad g\in \mathcal{U}_{p}\left( \mathbb{R}%
\right)\,,\, \ g=f\ \ \ \left( p.p\ \right) \right\} =\left\Vert
f\right\Vert _{U_{p}\left( \mathbb{R }\right) },\text{ } \\
&& \\
\text{therefore } \ \ \ \ \ \ \ \ \ \ \ \ \ \ \ \ \ \ \ \ \left\Vert
f\right\Vert _{\dot{\tilde{B}}_{p}^{1/p,\infty }\left( \mathbb{R}\right) }
&\leq &\left\Vert f\right\Vert _{U_{p}\left( \mathbb{R}\right) }, \\
&& \\
\text{which gives } \ \ \ \ \ \ \ \ \ \ \ \ \ \ \ \ U_{p}\left( \mathbb{R}
\right) &\hookrightarrow &\dot{\tilde{B}}_{p}^{1/p,\infty }\left( \mathbb{R}
\right) .
\end{eqnarray*}
\end{itemize}
\end{proof}

\setstretch{1.5}

\chapter{\textbf{Composition of Operators in $BV_{p}^{%
			\alpha }\left( I\right) $}}
 In this chapter, we give
the definitions of the spaces $\mathcal{V} _{p}^{\alpha }\left( I\right) ,$ $%
BV_{p}^{\alpha }\left( I\right) ,$ of functions defined on intervals $I$ of $%
\mathbb{R}$ with bounded $p$-variation of order $\alpha$, and the space of
their equivalence classes with respect to almost everywhere equality. Then
we study the particular cases $\displaystyle \alpha =1-\frac{1}{p},\ \left(
p\geq 1\right) $ which give the spaces $W^{1,p}\ (I)$. We also give a
generalization of Peetre's Theorem \ref{3.17} to the spaces $BV_{p}^{\alpha
}\left( I\right)$,\,$\ 0\leq\alpha <1\,,\ $ $1<p<\infty \,$. Finally, we
present some examples of composition operators defined in certain functional
spaces.

\section{\textbf{Basic Notions}}

In this section, we present some basic notions on the spaces $\mathcal{V}_{
p}^{\alpha }\left( I\right)$ of functions with bounded $p$-variation of
order $\alpha$ where $p\geq 1,\ \alpha \geq 0$, as well as the space of
their primitives $BV_{ p}^{\alpha }\left( I\right)$, in order to generalize
some results from previous chapters.

\subsection{\textbf{The Spaces }$\mathcal{V}_{ p}^{\alpha %
}\left(I\right) $}

\begin{definition}
\label{4.1} 
Let $p\in \left[ 1,+\infty \right]$. We define the norm of the space $%
l_{N}^{\ p}$ of finite real sequences $\left\{ t_{i}\right\} _{1\leq i\leq
N} $ by 
\begin{equation*}
\left\{ 
\begin{array}{c}
\left\Vert t_{i}\right\Vert _{l_{N}^{\ p}}=\left( \sum_{i=1}^{i=N}\left\vert
t_{i}\right\vert ^{p}\right) ^{1/p},\ \ \ \ \ \ \ \ \text{if } \ p\neq \infty
\\ 
\\ 
\ \left\Vert t_{i}\right\Vert _{l_{N}^{\ \infty }}=\sup_{1\leq i\leq
N}\left\vert t_{i}\right\vert \ \ \ \ ,\ \ \ \ \ \ \ \ \ \ \text{if } \
p=\infty%
\end{array}
\right.
\end{equation*}
\end{definition}

\begin{definition}
\label{4.2} 
Let $p\in \left[ 1,\ +\ \infty \ \right[ ,\ \alpha \geq 0$. Then a function $%
f\ :I\ \mathbf{\longrightarrow }\mathbb{R}$ is said to be of \textbf{bounded 
$p$-variation of order $\alpha$}, or briefly of $p$-$v$. $b$ of order $%
\alpha $, if for all finite, strictly increasing real sequences $t_{ 0}< t_{
1} <......<t_{ N}$ in $I$, there exists $c\ >\ 0$ such that 
\begin{equation}
\ \ \ \ \ \ \ \ \ \ \ \ \ \ \ \ \ \ \ \ \ \ \ \ \ \ \left\Vert \frac{
f\left( t_{k}\right) - f\left( t_{k- 1}\right) }{\left( t_{k}- t_{k-
1}\right) ^{\alpha }}\right\Vert _{l_{N}^{\ p}}\leq c\ <\infty \ \ \ \ \ \ \
\ \ \ \ \ \ \ \ \ \ \ \ \ \ \ \ \ \ \ \ \ \ \ \ \ \ \ \ \ \hspace*{\fill}
\label{[4.1]}
\end{equation}
or equivalently, if the number $\displaystyle\sup_{\left\{
t_{k}\right\}_{1\leq k\leq n,t_{k-1}<t_{k}}\subset I}\ \ \left[
\sum_{k=1}^{N}\left\vert \frac{f\left(t_{k}\right) - f\left( t_{k- 1}\right) 
}{\left( t_{k}- t_{k- 1}\right) ^{\alpha }}\right\vert ^{p}\right] \qquad 
\text{is finite}\,.$ \newline
\newline
The set of these functions is denoted by $\mathcal{V}_{ p}^{\alpha }\left(
I\right) ,\ \left( \mathcal{V}_{ p}^{\alpha }\ \text{if } I=\mathbb{R}
\right)\,,$ and the minimum of such constants $c$ is denoted by $\nu_{
p}^{\alpha }\ (f,\ I)$,$\ \left( \ \nu _{ p}^{\alpha }\ \text{if } I=\mathbb{%
R }\right)$: 
\begin{equation*}
\ \nu _{ p}^{\alpha }\ (f,\ I)=\inf_{c}\left\{ c\ >\
0:\sum_{k=1}^{N}\left\vert \frac{f\left( t_{k}\right) - f\left(
t_{k-1}\right) }{\left( t_{k}- t_{k- 1}\right) ^{\alpha }}\right\vert ^{
p}\leq c^{ p},\ \text{for all } \left\{ t_{k}\right\} _{1\leq k\leq
n,t_{k-1}<t_{k}}\subset I\right\}.
\end{equation*}
\end{definition}

\begin{remark}\text{ \ }
\label{4.3}
\vspace{-20pt}
\begin{itemize}
\item[--] We agree to take $\mathcal{V}_{ p}^{0}\left( I\right) = \mathcal{V}%
_{ p}\left( I\right)$.

\item[--] Condition (\ref{4.1}) is equivalent to the fact that for every
family of disjoint intervals $\left[ a_{k},b_{k}\right]$, $\ a_{k}\neq b_{k}$
of $I$, we have $\displaystyle\sum_{k=1}^{N}\left\vert \frac{f\left(
a_{k}\right) -f\left( b_{k}\right) }{\left( a_{k}-b_{k}\right) ^{\alpha }}%
\right\vert ^{p}\leq c<\infty ,\ \left( c\ >\ 0\right)$.

\item[--] $\mathcal{V}_{1}^{\alpha }\left( I\right)$ or simply $\mathcal{V}%
^{\alpha }\left( I\right)$ is called the space of functions of bounded
variation of order $\alpha$ on $I$, and $\mathcal{V}_{\infty }^{\alpha
}\left( I\right)$ is a Banach space for the norm 
\begin{eqnarray*}
\ \ \left\Vert f\right\Vert _{\mathcal{V}_{\infty }^{\alpha }\left( I\right)
} &=&\nu _{\infty }^{\alpha }\ (f,\ I)=\ \sup_{x\in I\ /\left\{ 0\right\}
}\left\vert \frac{f\left( x\right) }{x^{\alpha }}\right\vert , \\
\\
\text{and we have } \ \ \ \ \ \ \ \ \ \ \ \ \ \ \ \ \ \ \mathcal{V}_{\infty
}^{\alpha }\left( I\right) &=&Lip_{\alpha }\left( I\right) \text{,\ for }
0\leq \alpha <1.
\end{eqnarray*}
\end{itemize}
\end{remark}

\begin{proposition}
\label{4.4} 
Let $x\,,\,\ y\ \in I\,,\,\ p\ \in \ [\ 1\,,\,\ +\ \infty \ [\,,\,\ \text{%
and } \ 0\leq \alpha <1$. Then each element of $\mathcal{V}_{p}^{\alpha
}\left( I\right)$ is a Lipschitz function of order $\alpha$, and $\mathcal{V}%
_{p}^{\alpha }\left( I\right)$ becomes a Banach space when endowed with the
following norm: 
\begin{equation}
\displaystyle\left\Vert f\right\Vert _{\mathcal{V}_{p}^{\alpha }\left(
I\right) }=\sup_{x\in I}\left\vert f\left( x\right) \right\vert +\nu
_{p}^{\alpha }\ (f,\ I)\,.  \label{[4.2]}
\end{equation}
\end{proposition}

\begin{proof}

If $f\in \mathcal{V}_{p}^{\alpha }\left( I\right)$, then $f$ is $\nu
_{p}^{\alpha }$-Lipschitz of order $\alpha$ because 
\begin{equation*}
\ \left|\frac{f(x)-f(y)}{\left( x-y\right) ^{\alpha }}\right|\leq \nu
_{p}^{\alpha }(f,I)\ \ \ \ \ \ \ \ \text{for all } x,y\in I.\ \ \ \ \ \ 
\end{equation*}
and the norm (\ref{4.2}) satisfies all conditions making $\mathcal{V}%
_{p}^{\alpha }\left( I\right)$ a Banach space.
\end{proof}

\begin{proposition}
\label{4.5} 
The space $\mathcal{V}_{p}^{\alpha }\left( I\right)$ is a Banach algebra for
pointwise multiplication of functions, 
\begin{equation*}
\text{such that } \ \ \ \left\Vert f.g\right\Vert _{ \mathcal{V}_{p}^{\alpha
}\left( I\right)}\leq \left\Vert f\right\Vert _{ \mathcal{V}_{p}^{\alpha
}\left( I\right)}\left\Vert g\right\Vert _{ \mathcal{V} _{p}^{\alpha }\left(
I\right) }\ \ \ \ \ \ \ \ \ \ \ \ \text{for all } f, g\in \mathcal{V}
_{p}^{\alpha }\left(I\right) .
\end{equation*}
\end{proposition}

\begin{proof}

Set $x_{0}<x_{1}<...<x_{n}\ $\ a finite sequence in $I$ \ and $f,g\in 
\mathcal{V}_{p}\left( I\right).$ Since $p\geq 1$\ and by using Minkowski's inequality we obtain:
\begin{equation*}
\hspace*{-0.7in}\left( \sum\limits_{j=1}^{N}\left\vert \frac{
fg(x_{j})-fg(x_{j-1})}{\left( x_{j}-x_{j-1}\right) ^{\alpha }}\right\vert
^{p}\right) ^{\frac{1}{p}}\leq \left( \sum\limits_{j=1}^{N}\left\vert
f(x_{j})\left( \frac{g\left( x_{j}\right) -g(x_{j-1})}{\left(
x_{j}-x_{j-1}\right) ^{\alpha }}\right) \right\vert ^{p}\right)
^{1/p}+\left( \sum\limits_{j=1}^{N}\left\vert g(x_{j-1})\left( \frac{
f(x_{j})-f(x_{j-1})}{\left( x_{j}-x_{j-1}\right) ^{\alpha }}\right)
\right\vert ^{p}\right) ^{1/p}
\end{equation*}

\begin{eqnarray*}
	&\leq &\sup\limits_{I}\left\vert f\right\vert \,.\,\left(
	\sum\limits_{j=1}^{N}\left\vert \frac{g\left( x_{j}\right) -g(x_{j-1})}{
		\left( x_{j}-x_{j-1}\right) ^{\alpha }}\right\vert ^{p}\right) ^{\frac{1}{p}
	}+\sup\limits_{I}\left\vert g\right\vert .\,\left(
	\sum\limits_{j=1}^{N}\left\vert \frac{f\left( x_{j}\right) -f(x_{j-1})}{
		\left( x_{j}-x_{j-1}\right) ^{\alpha }}\right\vert ^{p}\right) ^{\frac{1}{p}
	} \text{\ \ \ \ \ \ \ \ \ \ \ \ \ \ \ \ \ \ \ \ \ \ \ \ \ \ \ \ \ \ \ \ \ \
		\ \ \ \ \ \ \ \ \ \ \ \ \ \ \ \ \ \ \ \ \ \ \ \ \ \ \ \ \ \ \ \ \ \ \ \ \ \
		\ \ \ \ \ \ \ \ \ \ \ \ \ \ \ \ \ \ \ \ \ \ \ \ \ \ \ \ \ \ \ \ \ \ \ \ \ \ }
	\\
	&& \\
	\ &\leq &\ \sup\limits_{I}\left\vert f\right\vert .\,\nu _{p}^{\alpha
	}(g,I)+\sup\limits_{I}\left\vert g\right\vert .\,\nu _{p}^{\alpha }(f,I), \\ \text{ \  }
	\end{eqnarray*}
hence

\begin{eqnarray*}
\nu _{p}^{\alpha }(f\text{ }.g,I &=&\sup\limits_{I}\left(
\sum\limits_{j=1}^{N}\left\vert \frac{fg(x_{j})-fg(x_{j-1})}{\left(
x_{j}-x_{j-1}\right) ^{\alpha }}\right\vert ^{p}\right) ^{\frac{1}{p}}\leq \
\sup\limits_{I}\left\vert f\right\vert .\,\nu _{p}^{\alpha
}(g,I)+\sup\limits_{I}\left\vert g\right\vert .\,\nu _{p}^{\alpha
}(f,I). \\
\end{eqnarray*}
We deduce that:
\begin{eqnarray*}
\left\Vert f\text{ }.\,g\right\Vert _{\mathcal{V}_{p}^{\alpha }\left(
I\right) } &=&\nu _{p}^{\alpha }(f.\,g,I)+\sup\limits_{I}\left\vert
f\right\vert \text{ }.\,\sup\limits_{I}\left\vert g\right\vert \leq \
\sup\limits_{I}\left\vert f\right\vert \text{ }.\,\nu _{p}^{\alpha
}(g,I)+\sup\limits_{I}\left\vert g\right\vert \text{ }.\,\nu _{p}^{\alpha
}(f,I)+\sup\limits_{I}\left\vert f\right\vert \text{ }.\,\sup\limits_{I}
\left\vert g\right\vert \\
&& \\
&\leq &\sup\limits_{I}\left\vert f\right\vert .\,\nu _{p}^{\alpha
}(g,I)+\sup\limits_{I}\left\vert g\right\vert .\,\nu _{p}^{\alpha
}(f,I)+\sup\limits_{I}\left\vert f\right\vert +\sup\limits_{I}\left\vert
g\right\vert +\underset{\geq 0}{\underbrace{\left[ \nu _{p}^{\alpha
}(g,I).\,\nu _{p}^{\alpha }(f,I)\right] }} \\
&& \\
&=&\left( \sup\limits_{I}\left\vert f\right\vert +\nu _{p}^{\alpha
}(f,I)\right) .\,\left( \sup\limits_{I}\left\vert g\right\vert +\nu
_{p}^{\alpha }(g,I)\right) =\left\Vert f\right\Vert _{\mathcal{V}
_{p}^{\alpha }\left( I\right) }\cdot \left\Vert g\right\Vert _{\mathcal{V}
_{p}^{\alpha }\left( I\right) }\,,
\end{eqnarray*}
\end{proof}

\subsection{\textbf{The Spaces }$BV_{p}^{\alpha }$$\left(I\right) $}

\begin{quote}
In this subsection, we
generalize the space of classes $BV_{p}$$\left( I\right)$ to the space of
classes $BV_{p}^{\alpha }$ $\left(I\right)$ such that $0\leq \alpha <1$.
Then we define the space of distributions $BV_{p}^{\alpha }$ for $\alpha
\geq 1$ as their primitives.
\end{quote}

\begin{definition}
\label{4.6} 
Let $p\ \in \lbrack 1,\ +\
\infty \ ],\ 0\leq \alpha <1$.\\ Then we denote by $\mathcal{BV}_{p}^{\alpha
}\left( I\right)$ the set of functions $f: \mathbb{R}\mathbf{\longrightarrow 
}\mathbb{R}$ such that there exists a function $g\in \mathcal{V}_{p}^{\alpha
}\left( I\right)$ which coincides with $f$ almost everywhere. 
\begin{equation*}
\ \mathcal{BV}_{p}^{\alpha }\left( I\right) =\left\{ f:I\mathbf{\
\longrightarrow }\mathbb{R} \quad ;\qquad \ \exists g\in \mathcal{V}
_{p}^{\alpha }\,,\,\text{such that}\quad f=g\ \left( p.p\right) \right\} ,
\end{equation*}
and we set 
\begin{equation*}
\varepsilon _{ p}^{\alpha }\left( f\right) =\inf \left\{ {\large \nu }_{
p}^{\alpha }\ \left( g\right) \quad ;\qquad \ g \in \mathcal{V}_{p}^{\alpha
}\,,\, \text{such that} \quad \ \ g\ = f{\large \ \ }\left( p.p\right)
\right\} .
\end{equation*}
We denote by $BV_{ p}^{\alpha }\left( I\right)$ the quotient set with
respect to the equivalence relation "equality in $\mathcal{BV} _{p}^{\alpha
}\left( I\right)$ almost everywhere", such that 
\begin{eqnarray*}
\dot{f} &=&\left\{ g\in \mathcal{BV}_{p}^{\alpha }\left( I\right)\quad ;
\qquad \ g=f\ \ \left( p.p\right) \right\} , \\
&& \\
\text{and } \ \text{\ \ \ \ \ \ \ \ \ \ \ \ \ \ \ \ \ \ \ \ \ \ \ \ \ \ \ }
BV_{ p}^{\alpha }\left( I\right) &=&\left\{ \ \dot{f}\quad ;\qquad \ f\in 
\mathcal{BV}_{p}^{\alpha }\left( I\right) \right\} =\mathcal{BV}_{p}^{\alpha
}\left( I\right) /e.p.p.
\end{eqnarray*}
If $h\in BV_{ p}^{\alpha}\left( I\right) ,$ we denote by $\varepsilon
_{p}^{\alpha }\left( h\right) \,$ the number $\varepsilon _{ p}^{\alpha
}\left( f\right) \,$ for any representative $f$ of $h$.
\end{definition}

\begin{quote}
We agree to take $\ BV_{p}^{0}\left( I\right) =BV_{p}\left( I\right)$.
\end{quote}

\begin{proposition}
\label{4.7} 
Let $p\in \ [1,+\infty ], 0\leq
\alpha <1$. If $f\in BV_{p}^{\alpha }( \mathbb{R})$, then $f$ has a unique
normal representative $\tilde{f}\in \ \mathcal{V}_{p}^{\alpha }$, and we
have 
\begin{equation*}
\varepsilon _{p}^{\alpha }\left( f\right) =\nu _{p}^{\alpha }(\tilde{f})\ .
\end{equation*}
\end{proposition}

\begin{quote}
We will therefore consider $BV_{p}^{\alpha }(\mathbb{R})$ as a Banach space
of distributions endowed with the following norm 
\begin{eqnarray*}
\left\Vert f\right\Vert _{BV_{p}^{\alpha }\left(\mathbb{R}\right) }
&=&\varepsilon _{p}^{\alpha }\left( f\right) +\left\Vert \frac{f\left(
x\right) }{x^{\alpha }}\right\Vert _{\infty } \\
\\
&=&\nu _{p}^{\alpha }( \tilde{f} )+\sup_{x\in\mathbb{R}/\left\{ 0\right\}
}\left\vert \frac{\tilde{f} \left( x\right) }{x^{\alpha }} \right\vert \ \ .
\end{eqnarray*}
Note that for $p=\infty$, we obtain 
\begin{equation*}
\left\Vert f\right\Vert _{BV_{\infty }^{\alpha }\left( I\right) }\sim
\sup_{x\in I\ /\left\{ 0\right\} }\left\vert \frac{\tilde{f}\left( x\right) 
}{x^{\alpha }}\right\vert .
\end{equation*}
\end{quote}

\begin{definition}
\label{4.8}
Let $p\in \lbrack \ 1,\ +\
\infty \ ],\ \alpha \geq 1$. Then we say that a function $\ f\ :\ I \ 
\mathbf{\longrightarrow }\mathbb{R}$ belongs to $BV_{ p}^{\alpha
}\left(I\right)$ if there exist $c,x_{0}\ \in\mathbb{R}$ and $g\in
BV_{p}^{\alpha -1}\left( I\right)$ such that 
\begin{equation}
\ \text{\ \ \ \ \ \ \ \ \ \ \ \ \ \ \ \ \ \ \ \ \ \ \ \ \ \ \ \ \ \ \ \ }
f(x)=c+ \int_{x_{0}}^{x}g\left( t\right) \,\mathrm{d}{t}\ \ \ \ \ \ \ \ \ 
\text{for all } x\in I\hfill \, .  \label{[4.3]}
\end{equation}
\end{definition}

\begin{quote}
If (\ref{4.3}) holds, then $f$ is a Lipschitz continuous function, and we
endow $BV_{p}^{\alpha }\left( I\right)$ with the norm 
\begin{equation*}
\left\Vert f\right\Vert _{BV_{p}^{\alpha }\left( I\right)
}=|f(x_{0})|+\left\Vert f^{\prime }\right\Vert _{BV_{p}^{\alpha -1}\left(
I\right) \ },
\end{equation*}
for which $BV_{p}^{\alpha }\left( I\right)$ becomes a Banach space. For each
point $x_{0}$ in $I$, we obtain an equivalent norm.
\end{quote}

\section{\textbf{The Spaces}$\ W^{1,p}\left( \Omega \right) $, $\Omega $ 
\textbf{\ an open subset of }$\mathbb{R}^{n}$}

 In this section, we study
the space $W^{1,p}\left( \Omega \right)$, where $\Omega$ is an open subset
of $\mathbb{R}^{n}$, which proves to be very interesting due to its
regularity with respect to composition operators. See the works of H. Brezis
~\cite{Br-Mir}.

\subsection{\textbf{Properties of the Spaces }$W^{1,p}\left( \Omega\right) $}

\begin{definition}[~\cite{Po}]\text{ \ }
\label{4.9} 
\vspace{-20pt}
\begin{itemize}
\item[$\ast$] Let $\Omega$ be
an open subset of $\mathbb{R}^{n}$, and $1\leq p\leq \infty$. Then we define
the space $W^{1,p}\left( \Omega \right)$ by 
\begin{equation*}
W^{1,p}\left( \Omega \right) =\left\{ u\in L^{p}\left( \Omega \right) ,\frac{
\partial u}{\partial x_{i}}\left( \text{in the sense of distributions}
\right) \in L^{p}\left( \Omega \right) ,\text{ }i=1,..n\right\}.
\end{equation*}

\item[$\ast $] We set $H^{1}\left( \Omega \right) =W^{1,2}\left( \Omega
\right)$, which is endowed with the scalar product 
\begin{equation*}
\ \left\langle u,v\right\rangle _{H^{1}\left( \Omega \right) }=\left\langle
u,v\right\rangle _{L^{2}\left( \Omega \right) }+ \sum_{i=1}^{n}\left\langle 
\frac{\partial u}{\partial x_{i}},\frac{ \partial v}{\partial x_{i}}%
\right\rangle _{L^{2}\left( \Omega \right) },
\end{equation*}
and the associated norm is 
\begin{equation*}
\left\Vert u\right\Vert _{H^{1}\left( \Omega \right) }=\left( \left\Vert
u\right\Vert _{L^{2}\left( \Omega \right) }^{2}+ \sum_{i=1}^{n}\left\Vert 
\frac{ \partial u}{ \partial x_{i}}\right\Vert _{L^{2}\left( \Omega \right)
}^{2}\right) ^{\frac{ 1}{2}}.
\end{equation*}

\item[$\ast$] The space $\displaystyle W^{1,p} \left( \Omega \right)$ is
endowed with the norm 
\begin{equation*}
\left\Vert u\right\Vert _{W^{1,p}\left( \Omega \right) }=\left\Vert
u\right\Vert _{L^{p}\left( \Omega \right) }+ \sum_{i=1}^{n}\left\Vert \frac{%
\partial u}{\partial x_{i} }\right\Vert _{L^{p}\left( \Omega \right) }.
\end{equation*}

\item[$\ast$] If $\displaystyle 1<p<+\infty \,$, then 
\begin{equation*}
\left\Vert u\right\Vert _{W^{1,p}\left( \Omega \right) }\sim \left[
\left\Vert u\right\Vert _{L^{p}\left( \Omega \right)
}^{p}+\sum_{i=1}^{n}\left\Vert \frac{\partial u}{\partial x_{i}} \right\Vert
_{L^{p}\left( \Omega \right) }^{p}\right] ^{1/p}\ .
\end{equation*}

\item[$\ast$] If we denote $\displaystyle g_{i}= \frac{\partial u}{\partial
x_{i}}\,$, then $u\in W^{1,p}\left( \Omega \right)$ if and only if 
\begin{equation*}
\forall \varphi \in \mathcal{D}\left( \Omega \right) \,,\,\forall
i=1,...n\quad ;\qquad \text{ \ \ \ }\int_{\Omega }u\frac{\partial \varphi }{%
\partial x_{i}}=-\int_{\Omega }g_{i}\varphi \ \ ,\ \ \ .
\end{equation*}
\end{itemize}
\end{definition}


\begin{definition}[~\cite{Po}]\text{  }
\label{4.10} 
\vspace{-20pt}
\begin{itemize}
\item[$\ast$]  Let $\Omega$
be an open subset of $\mathbb{R}^{n}$, $1 \leq p \leq \infty$, $m \geq 2$,
then we define the Sobolev space $W^{m,p}\left(\mathbb{R}^{n}\right)$, by 
\begin{equation*}
W^{m,p}\left( \Omega \right) = \left\{ u \in W^{m-1,p}\left( \Omega \right);
\qquad \frac{\partial u}{\partial x_{i}} \in W^{m-1,p}\left( \Omega \right), 
\text{ for all } i=1,\dots,N \right\}.
\end{equation*}

\item[$\ast$] We set $H^{m}\left( \Omega \right) = W^{m,2}\left( \Omega
\right)$ which becomes a Hilbert space, if equipped with the inner product 
\begin{equation*}
\left\langle u, v \right\rangle_{H^{m}\left( \Omega \right)} =
\sum\limits_{0 \leq \left\vert \alpha \right\vert \leq m} \left\langle
D^{\alpha}u, D^{\alpha}v \right\rangle_{L^{2}\left( \Omega \right)}.
\end{equation*}

\item[$\ast$] The space $W^{m,p}\left( \Omega \right)$ is equipped with the
norm 
\begin{equation*}
\left\Vert u \right\Vert_{W^{m,p}} = \sum\limits_{0 \leq \left\vert \alpha
\right\vert \leq m} \left\Vert D^{\alpha}u \right\Vert_{L^{p}}.
\end{equation*}

\item[$\ast$] If we denote $\displaystyle g_{\alpha} = D^{\alpha}u = \frac{%
D^{\left\vert \alpha \right\vert}}{\partial x_{1}^{\alpha_{1}} \cdots
\partial x_{n}^{\alpha_{n}}}$, then for all $\displaystyle u \in
L^{p}\left(\Omega \right)$ 
\begin{equation*}
u \in W^{m,p}\left( \Omega \right), \quad \text{if and only if} \qquad
\forall \alpha \in \mathbb{N}^{n}, \text{ } \left\vert \alpha \right\vert
\leq m, \quad \exists g_{\alpha} \in L^{p}\left( \Omega \right), \text{ such
that}
\end{equation*}
\begin{equation*}
\int_{\Omega} u D^{\alpha} \varphi = \left( -1 \right)^{\left\vert \alpha
\right\vert} \int_{\Omega} g_{\alpha} \varphi, \quad \text{for all } \varphi
\in \mathcal{D}\left( \Omega \right).
\end{equation*}

\item[$\ast$] If $W_{0}^{1,p}\left( \Omega \right)$ denotes the closure of $%
\mathcal{D}\left(\mathbb{R}^{n}\right)$ in $W^{1,p}\left( \Omega \right)$,
then we have the following inequality called Poincare 's inequality 
\begin{eqnarray*}
\forall u & \in & W_{0}^{1,p}\left( \Omega \right); \qquad \int_{\Omega}
\left\vert u \right\vert^{p} \leq c\left( \Omega \right)^{p} \int_{\Omega}
\left\Vert \nabla u \right\Vert^{p}, \quad c\left( \Omega \right) > 0, \quad
p \geq 1. \\
&& \\
\text{Where } \nabla u &=& \left( \frac{\partial u}{\partial x_{1}}, \dots, 
\frac{\partial u}{\partial x_{k}}, \dots, \frac{\partial u}{\partial x_{n}}
\right), \qquad \frac{\partial u}{\partial x_{k}} \in L^{p}\left( \Omega
\right), \quad k = 1, \dots, n.
\end{eqnarray*}
\end{itemize}
\end{definition}

\begin{proposition}[ ~\cite{GIR} ]
\label{4.11}
 Let $I = \left[ a, b \right]$
such that $a, b \in \mathbb{R}$, and consider a function $f: I
\longrightarrow \mathbb{R}$, and a partition $\mathcal{P} \subset I$ such
that 
\begin{equation*}
\mathcal{P} = \left\{ x_{0}, x_{1}, x_{2}, \dots, x_{n} \right\}, \text{
where } a = x_{0} < x_{1} < x_{2} < \dots < x_{n} = b,
\end{equation*}
if we define the sums $S\left( f, \mathcal{P} \right), s\left( f, \mathcal{P}
\right), V\left( f, \mathcal{P} \right)$ by 
\begin{eqnarray*}
S\left( f, \mathcal{P} \right) &=& \sum\limits_{k=1}^{n} \sup \left\{
f\left( x \right); x_{k-1} \leq x \leq x_{k} \right\} \cdot \left( x_{k} -
x_{k-1} \right) \\
\\
s\left( f, \mathcal{P} \right) &=& \sum\limits_{k=1}^{n} \inf \left\{
f\left( x \right); x_{k-1} \leq x \leq x_{k} \right\} \cdot \left( x_{k} -
x_{k-1} \right) \\
\\
V\left( f, \mathcal{P} \right) &=& \sum\limits_{k=1}^{n} \left\vert f\left(
x_{k} \right) - f\left( x_{k-1} \right) \right\vert \text{ and } \nu \left(
f, \left[ a, b \right] \right) = \sup \left\{ V\left( f, \mathcal{P}
\right); \mathcal{P} \subset \left[ a, b \right] \right\},
\end{eqnarray*}
\begin{equation*}
\text{then } \int_{a}^{b} f\left( x \right) \, \mathrm{d}x = \inf \left\{
S\left( f, \mathcal{P} \right); \mathcal{P} \subset \left[ a, b \right]
\right\} = \sup \left\{ s\left( f, \mathcal{P} \right); \mathcal{P} \subset %
\left[ a, b \right] \right\},
\end{equation*}
$\ \ \ $\ \ \ \ \ \ \ \ \ \ \ \ \ \ \ \ \ \ \ \ \ \ \ \ \ \ \ \ \ \ \ \ 
\begin{equation*}
\text{and if $f$ is absolutely continuous then } \nu \left( f, \left[ a, b %
\right] \right) = \int_{a}^{b} \left\vert f^{\prime }\left( x \right)
\right\vert \, \mathrm{d}x \text{ \ \ \ \ \ \ \ \ \ \ \ \ \ \ \ \ \ \ \ \ \
\ \ \ \ \ \ \ \ \ \ \ \ \ \ \ }
\end{equation*}
\end{proposition}

\begin{proposition}[~\cite{Br-Mir}]
\label{4.12}  
Let $1 < p < \infty$, $s \in \mathbb{N}$, $0 < \sigma \leq s$, $\Omega
\subset \mathbb{R}^{n}$, then we can find equivalent norms for the space $%
W^{s,p}\left(\Omega \right)$ defined by 
\begin{equation*}
\left\Vert f \right\Vert_{W^{s,p}\left( \Omega \right)} = \left\Vert f
\right\Vert_{L^{p}\left( \Omega \right)} + \sup\limits_{h \in \mathbb{R}%
^{n}, h \neq 0} \frac{\left\Vert \Delta_{h}^{s} f \right\Vert_{L^{p}\left(
\Omega \right)}}{\left\vert h \right\vert^{\sigma}}.
\end{equation*}
\end{proposition}

\begin{lemma}[~\cite{GIR}]
 \label{4.13} 
Let $1 < p < \infty$, then for every interval $I$ of $\mathbb{R}$, we have 
\begin{equation*}
W^{1,p}\left( I \right) = \mathcal{V}_{p}^{1 - \frac{1}{p}}\left( I \right).
\end{equation*}
\end{lemma}

\begin{proof}

\begin{description}
\item[--] If $f \in \mathcal{V}_{p}^{1 - \frac{1}{p}} \left( \left[ a, b %
\right] \right)$, then for any partition $\mathcal{P} = \left\{ a \leq
x_{k-1} < x_{k} \leq \dots b \right\}$, 
\begin{equation*}
\text{we have } \left( \sum\limits_{k \geq 0} \left\vert x_{k} - x_{k-1}
\right\vert \cdot \left( \frac{\left\vert f\left( x_{k} \right) - f\left(
x_{k-1} \right) \right\vert}{\left\vert x_{k} - x_{k-1} \right\vert}
\right)^{p} \right)^{1/p} \leq c < \infty, \text{ \ \ \ \ \ \ \ \ \ \ \ \ \
\ \ \ \ \ \ \ \ \ \ \ \ \ \ \ \ \ \ \ \ \ \ \ \ \ \ \ \ \ \ \ \ \ \ \ \ \ \
\ \ \ \ \ \ \ \ \ \ \ \ }
\end{equation*}
fix the difference $\left\vert x_{k} - x_{k-1} \right\vert = h$, and
consider the function $g_{h}$, defined by $\displaystyle g_{h}\left( x
\right) = \left( \left\vert f\left( x+h \right) - f\left( x \right)
\right\vert \right)^{p}$, then according to Proposition \ref{4.11} we have 
\begin{equation*}
\int_{a}^{b} g_{h}\left( x \right) \, \mathrm{d}x = \inf\limits_{\mathcal{P}
\subset \left[ a, b \right]} \left\{ \sum\limits_{k=1}^{n} \sup\limits_{x}
\left\{ g_{h}\left( x \right); x_{k-1} \leq x \leq x_{k} \right\} \cdot
\left( x_{k} - x_{k-1} \right) \right\}, \text{ hence }
\end{equation*}

$\displaystyle \int_{a}^{b} \left( \left\vert f\left( x+h \right) - f\left(
x \right) \right\vert \right)^{p} \leq \sum\limits_{k \geq 0} \left\vert
x_{k} - x_{k-1} \right\vert \left( \left\vert f\left( x_{k} \right) -
f\left( x_{k-1} \right) \right\vert \right)^{p},$ 

dividing both sides by $\left\vert x_{k} - x_{k-1} \right\vert^{p}$, we
obtain 
\begin{equation*}
\left[ \frac{\displaystyle \int_{a}^{b} \left( \left\vert f\left( x+h
\right) - f\left( x \right) \right\vert \right)^{p} \, \mathrm{d}x}{%
\left\vert h \right\vert^{p}} \right]^{\frac{1}{p}} \leq \left[
\sum\limits_{k \geq 0} \left\vert x_{k} - x_{k-1} \right\vert \left( \frac{%
\left\vert f\left( x_{k} \right) - f\left( x_{k-1} \right) \right\vert}{%
\left\vert x_{k} - x_{k-1} \right\vert} \right)^{p} \ \right]^{\frac{1}{p}}
\leq \left\Vert f \right\Vert_{\mathcal{V}_{p}^{1 - \frac{1}{p}}},
\end{equation*}
and thus $\mathcal{V}_{p}^{1 - \frac{1}{p}}\left( I \right) \subset
W^{1,p}\left( I \right)$.

\item[--] Consider a sequence of disjoint intervals $\left[ a_{k}, b_{k} %
\right]$ of $I$, we can always choose a subsequence of intervals such that $%
\displaystyle \left[ a_{k}^{\prime }, b_{k}^{\prime }\right] \subset \left[
a_{k}, b_{k} \right]$, $\left\vert b_{k}^{\prime }- a_{k}^{\prime
}\right\vert = \left\vert h \right\vert \neq 0$.

If $f \in W^{1,p}\left( I \right)$, then according to Proposition \ref{4.12} 
\begin{eqnarray*}
\frac{\left[ \displaystyle \sum\limits_{k \geq 0} \left( \left\vert f\left(
b_{k} \right) - f\left( a_{k} \right) \right\vert \right)^{p} \right]^{1/p}}{%
\left\vert b_{k} - a_{k} \right\vert^{1 - \frac{1}{p}}} &=& \frac{ \left( %
\displaystyle \sum\limits_{k \geq 0} \left( \left\vert f\left( a_{k} + h
\right) - f\left( a_{k} \right) \right\vert \right)^{p} \right)^{1/p}}{%
\left\vert h \right\vert^{1 - \frac{1}{p}}} \\
&\leq& \displaystyle \frac{ \left( \displaystyle \int_{a}^{b} \left(
\left\vert f\left( x + h \right) - f\left( x \right) \right\vert \right)^{p}
\, \mathrm{d}x \right)^{1/p}}{\left\vert h \right\vert^{1 - \frac{1}{p}}} = 
\frac{\left\Vert \Delta_{h}^{1} f \right\Vert_{L^{p}}}{\left\vert h
\right\vert^{\sigma}} \leq c < \infty,
\end{eqnarray*}
with $\displaystyle 0 < \sigma = 1 - \frac{1}{p} < 1$ because $p > 1$, hence 
$\displaystyle W^{1,p}\left( I \right) \subset \mathcal{V}_{p}^{1 - \frac{1}{%
p}}\left( I \right)$.
\end{description}
\end{proof}

\subsection{\textbf{Composition  }$W^{1,p}\left( \Omega \right)$%
}

\begin{corollary}[~\cite{Po}]
\label{4.14} 
 Let $G \in C^{1}\left( 
\mathbb{R} \right)$, $G\left( 0 \right) = 0$, $1 \leq p < \infty$, $u \in
W^{1,p}\left( I \right)$, then 
\begin{equation*}
G \circ u \in W^{1,p}\left( I \right) \quad \text{and} \quad \left( G \circ
u \right)^{\prime }= \left( G^{\prime }\circ u \right) u^{\prime }.
\end{equation*}
\end{corollary}

\begin{proof}

By the mean value theorem we have 
\begin{equation*}
\frac{\left\vert G\left( s \right) - G\left( 0 \right) \right\vert}{%
\left\vert s - 0 \right\vert} \leq 1_{\text{supp } G}\left( s \right) \cdot
\left\Vert G^{\prime }\right\Vert_{\infty}.
\end{equation*}
If $G\left( s \right) = 0$, the result is immediate; otherwise $1_{\text{%
supp } G}\left( s \right) = 1$,\\ and since $G \in C^{1}\left( \mathbb{R}
\right)$ then $\left\Vert G^{\prime }\right\Vert_{\infty} \leq c < \infty$,
and taking $M = \left\Vert u \right\Vert_{L^{\infty}}$, so for every $M > 0$%
, we have 
\begin{equation*}
\forall s \in \left[ -M, +M \right], \quad \left\vert G\left( s \right)
\right\vert \leq c \left\vert s \right\vert \leq c \left\Vert u
\right\Vert_{L^{\infty}},
\end{equation*}
but $\left\Vert u \right\Vert_{L^{\infty}} \leq c^{\prime }\left\Vert u
\right\Vert_{W^{1,p}}$, $c^{\prime }> 0$, hence 
\begin{equation*}
\left\vert G \circ u \right\vert \leq c \left\Vert u
\right\Vert_{L^{\infty}} \leq c^{\prime \prime }\left\Vert u
\right\Vert_{W^{1,p}} < \infty,
\end{equation*}
and since $u \in L^{p}\left( I \right)$, $u^{\prime p}\left( I \right)$,
then 
\begin{equation*}
G \circ u \in L^{p}\left( I \right), \quad \left( G \circ u^{\prime }\right)
u^{\prime p}\left( I \right),
\end{equation*}
and since $1 \leq p < \infty$, then there exists a sequence $u_{n}$ in $%
\mathcal{D}\left( \mathbb{R} \right)$ such that $u_{n} \longrightarrow u$ in 
$W^{1,p}(I)$ and in $L^{\infty}\left( I \right)$, and thus 
\begin{eqnarray*}
G \circ u_{n} &\longrightarrow& G \circ u \in L^{\infty}\left( I \right) \\
&& \\
\text{and } (G^{\prime }\circ u_{n}) u_{n}^{\prime }&\longrightarrow&
(G^{\prime }\circ u) u^{\prime p}\left( I \right),
\end{eqnarray*}
\begin{equation*}
\text{but we have } \qquad \forall \varphi \in \mathcal{D}\left( I \right),
\quad \int_{I} (G \circ u_{n}) \varphi^{\prime }= -\int_{I} (G^{\prime
}\circ u_{n}) u_{n}^{\prime }\varphi,
\end{equation*}
hence the result, because it suffices to take $\displaystyle \varphi\left( t
\right) = e^{-t} \in \mathcal{D}\left( I \right)$.
\end{proof}

\begin{lemma}[~\cite{Po}]
\label{4.15} 
 Let $\Omega$ be an open
subset of $\mathbb{R}^{n}$, $1 \leq p < \infty$, $G \in C^{1}\left( \mathbb{R%
} \right)$, such that 
\begin{equation*}
G\left( 0 \right) = 0, \quad \text{and} \quad \left\vert G^{\prime }\left( s
\right) \right\vert \leq M, \quad \left( M > 0 \right) \text{ for all } s
\in \mathbb{R},
\end{equation*}
\begin{equation*}
\text{we have if } u \in W^{1,p}\left( \Omega \right), \text{ then } G \circ
u \in W^{1,p}\left( \Omega \right) \quad \text{and} \quad \frac{\partial}{%
\partial x_{i}} \left( G \circ u \right) = \left( G^{\prime }\circ u \right) 
\frac{\partial u}{\partial x_{i}}.
\end{equation*}
\end{lemma}

\begin{proof}

The same steps as in the proof of Corollary \ref{4.14} give us Lemma \ref%
{4.15}.
\end{proof}
\bigskip \bigskip
\begin{proposition}[~\cite{Br-Mir}]
\label{4.16} 
Let $\Omega$ be an open subset of $\mathbb{R}^{n}$, $s > 1$, $1 < p < \infty$%
, $sp = n$, $k = \left[ s \right] + 1$, such that 
\begin{equation*}
G \in C^{k}\left( \mathbb{R} \right), \quad G(0) = 0 \quad \text{and} \quad
D^{j}G \in L^{\infty}\left( \mathbb{R} \right), \quad \text{for all } j \leq
k,
\end{equation*}
\begin{equation*}
\text{we have, } \text{ if } u \in W^{s,p}\left( \Omega \right) \quad \text{%
then} \quad G \circ u \in W^{s,p}\left( \Omega \right). \text{ \ \ \ \ \ \ \
\ \ \ \ \ \ \ \ \ \ \ \ \ \ \ \ \ \ \ \ \ \ \ \ \ \ \ \ \ \ \ \ \ \ \ \ \ \
\ \ \ \ \ \ \ \ \ \ \ \ \ \ \ \ }
\end{equation*}
\end{proposition}

\begin{quote}
This Proposition \ref{4.16} is a generalization of Lemma \ref{4.15},\\ but it
is not verified for $p = 1$, see ~\cite{Br-Mir}.
\end{quote}

\begin{theorem}[~\cite{Bo-Mo-Si}]
\label{4.17}
Let $f: \mathbb{R} \longrightarrow \mathbb{R}$ be a Borel measurable
function such that $f\left( 0 \right) = 0$.

\begin{itemize}
\item[(i)] If $1 \leq p \leq \infty$, then $T_{f}$ operates on $%
W^{1,p}\left( \mathbb{R} \right)$ if and only if $f$ is a continuous and
locally Lipschitz function.

\item[(ii)] If $1 \leq p \leq \infty$ and $m \geq 2$, then 
\begin{equation*}
T_{f} \text{ operates on } W^{m,p}\left( \mathbb{R} \right) \text{ if and
only if } f \in W_{\ell oc}^{m,p}\left( \mathbb{R} \right).
\end{equation*}
\end{itemize}
\end{theorem}

\begin{quote}
It follows that if the composition operator $T_{f}$ operates on a Sobolev
space then it must be bounded. This Theorem \ref{4.17} is due to Marcus and
Mizel, ~\cite{Bo-Mo-Si}.
\end{quote}

\section{\textbf{Peetre's theorem  } $BV_{p}^{\alpha}
\left( I \right)$}

\begin{theorem}
\label{4.18}
Let $p \in ]1, +\infty[$, $0 \leq \alpha < 1$, then we have the following
continuous injections which are verified 
\begin{eqnarray*}
\dot{\tilde{B}}_{p}^{1/p, 1}\left( \mathbb{R} \right) &\hookrightarrow&
\left( L^{\infty}\left( \mathbb{R} \right), BV_{1}^{\alpha}\left( \mathbb{R}
\right) \right)_{1/p, p} = \left( BV_{\infty}^{\alpha}\left( \mathbb{R}
\right), BV_{1}^{\alpha}\left( \mathbb{R} \right) \right)_{1/p, p} \\
&& \\
&\hookrightarrow& BV_{p}^{\alpha}\left( \mathbb{R} \right) \hookrightarrow
U_{p}\left( \mathbb{R} \right) \hookrightarrow \dot{\tilde{B}}_{p}^{1/p,
\infty}\left( \mathbb{R} \right).
\end{eqnarray*}
\end{theorem}

\begin{proof}
{\ \ \ } 
\vspace{-20pt}
\begin{itemize}
\item[1)] As we saw in the proof of Theorem \ref{3.17} we have 
\begin{equation*}
\dot{B}_{p}^{1/p, 1}\left( \mathbb{R} \right) \cap C_{0}\left( \mathbb{R}
\right) \hookrightarrow \left( \dot{B}_{\infty}^{0, 1}\left( \mathbb{R}
\right) \cap C_{b}\left( \mathbb{R} \right), \dot{B}_{1}^{1, 1}\left( 
\mathbb{R} \right) \cap C_{0}\left( \mathbb{R} \right) \right)_{1/p, p}.
\end{equation*}

\begin{itemize}
\item Let us prove the injections $\dot{\tilde{B}}_{1}^{1, 1}
\hookrightarrow \dot{W}^{1,1} \hookrightarrow BV_{1}^{\alpha}$.

\begin{itemize}
\item The injection $\dot{\tilde{B}}_{1}^{1, 1} \hookrightarrow \dot{W}%
^{1,1} $ is verified (Theorem \ref{3.17}).

\item Since $\displaystyle \nu_{1}^{\alpha} (f, \mathbb{R}) =
\sup\limits_{\left\{ t_{0} < t_{1} < \dots < t_{k} \dots < t_{n} \right\}
\subset \mathbb{R}, n \in \mathbb{N}} \left[ \sum\limits_{k=1}^{n} \frac{%
\left\vert f\left( t_{k} \right) - f\left( t_{k-1} \right) \right\vert}{%
\left\vert t_{k} - t_{k-1} \right\vert^{\alpha}} \right]$, then 
\begin{eqnarray*}
\left\Vert f \right\Vert_{BV_{1}^{\alpha}\left( \mathbb{R} \right)} &=&
\nu_{1}^{\alpha}\left( \tilde{f} \right) + \sup\limits_{x \in \mathbb{R}
\setminus \left\{ 0 \right\}} \left\vert \frac{\tilde{f}^{\prime }\left( x
\right)}{x^{\alpha}} \right\vert \leq 2 \nu_{1}^{\alpha}\left( \tilde{f}
\right), \text{ because } 0 \leq \alpha < 1 \\
&& \\
&\leq& 2 \int_{\mathbb{R} \setminus \left\{ 0 \right\}} \left\vert \frac{%
\tilde{f}^{\prime }\left( x \right)}{x^{\alpha}} \right\vert \, \mathrm{d}x
\sim 2 \left\Vert f \right\Vert_{\dot{W}^{1,1}\left( \mathbb{R} \right)}, \\
&& \\
\text{hence } \dot{W}^{1,1} &\hookrightarrow& BV_{1}^{\alpha}, \text{ and
thus } \dot{\tilde{B}}_{1}^{1, 1} \hookrightarrow BV_{1}^{\alpha}.
\end{eqnarray*}

\item We also prove that $\dot{\tilde{B}}_{\infty}^{0,1} \hookrightarrow
L^{\infty}$ (Theorem \ref{3.17}).
\end{itemize}

\item Then applying the interpolation Theorem \ref{3.3} we obtain 
\begin{equation*}
\left( \dot{\tilde{B}}_{\infty}^{0,1}, \dot{\tilde{B}}_{1}^{1,1}
\right)_{1/p, p} = \dot{\tilde{B}}_{p}^{1/p,1} \hookrightarrow \left(
L^{\infty}, BV_{1}^{\alpha} \right)_{1/p, p}.
\end{equation*}
\end{itemize}

\item[(2)] If $E_{\alpha}$ is the closure of $BV_{1}^{\alpha}\left( \mathbb{R%
} \right)$ with respect to $L^{\infty}\left( \mathbb{R} \right)$, then 
\begin{equation*}
\left( L^{\infty}\left( \mathbb{R} \right), BV_{1}^{\alpha}\left( \mathbb{R}
\right) \right)_{1/p, p} = \left( E_{\alpha}, BV_{1}^{\alpha}\left( \mathbb{R%
} \right) \right)_{1/p, p}
\end{equation*}
and since $E_{\alpha} \hookrightarrow BV_{\infty}^{\alpha}\left( \mathbb{R}
\right)$, then 
\begin{eqnarray*}
\left( E_{\alpha}, BV_{1}^{\alpha}\left( \mathbb{R} \right) \right)_{1/p, p}
&\hookrightarrow& \left( BV_{\infty}^{\alpha}\left( \mathbb{R} \right),
BV_{1}^{\alpha}\left( \mathbb{R} \right) \right)_{1/p, p}, \\
&& \\
\text{hence } \left( L^{\infty}\left( \mathbb{R} \right),
BV_{1}^{\alpha}\left( \mathbb{R} \right) \right)_{1/p, p} &\hookrightarrow&
\left( BV_{\infty}^{\alpha}\left( \mathbb{R} \right), BV_{1}^{\alpha}\left( 
\mathbb{R} \right) \right)_{1/p, p}.
\end{eqnarray*}

\item[(3)] We define the norms of the space $l_{N}^{p}$ of sequences $%
\left\{ t_{i} \right\}_{1 \leq i \leq N}$ by 
\begin{equation*}
\left\Vert t_{i} \right\Vert_{l_{N}^{p}} = \left( \sum\limits_{i=1}^{N}
\left\vert t_{i} \right\vert^{p} \right)^{1/p} \text{ and } \left\Vert t_{i}
\right\Vert_{l_{N}^{\infty}} = \sup\limits_{1 \leq i \leq N} \left\vert
t_{i} \right\vert.
\end{equation*}
Fix a finite strictly increasing real sequence $\left\{ t_{i} \right\}_{0
\leq i \leq N}$ such that $t_{0} < t_{1} < \dots < t_{N}$, and associate for
every $0 \leq \alpha < 1$, a function $U_{\alpha}$, defined by 
\begin{eqnarray*}
U_{\alpha}: BV_{\infty}^{\alpha}\left( \mathbb{R} \right) &\longrightarrow&
l_{N}^{\infty} \\
&& \\
f &\longmapsto& \left( \frac{\tilde{f}\left( t_{i} \right) - \tilde{f}\left(
t_{i-1} \right)}{\left\vert t_{i} - t_{i-1} \right\vert^{\alpha}} \right)_{1
\leq i \leq N},
\end{eqnarray*}
\begin{eqnarray*}
\text{thus } \left\Vert U_{\alpha}\left( f \right)
\right\Vert_{l_{N}^{\infty}} &=& \sup\limits_{t_{k} \in \left\{ t_{i}
\right\}} \left\vert \frac{\tilde{f}\left( t_{k} \right) - \tilde{f}\left(
t_{k-1} \right)}{\left\vert t_{k} - t_{k-1} \right\vert^{\alpha}}
\right\vert \leq \sup\limits_{t_{k} \in \left\{ t_{i} \right\}} \left(
\left\vert \frac{\tilde{f}\left( t_{k} \right)}{\left\vert t_{k} - t_{k-1}
\right\vert^{\alpha}} \right\vert + \left\vert \frac{\tilde{f}\left( t_{k-1}
\right)}{\left\vert t_{k} - t_{k-1} \right\vert^{\alpha}} \right\vert \right)
\\
\\
&\leq& 2 \sup\limits_{t \in \mathbb{R} \setminus \left\{ 0 \right\}}
\left\vert \frac{\tilde{f}\left( t \right)}{t^{\alpha}} \right\vert = 2
\left\Vert f \right\Vert_{BV_{\infty}^{\alpha}\left( \mathbb{R} \right)}, \\
\\
\text{hence } \left\Vert U_{\alpha}\left( f \right)
\right\Vert_{l_{N}^{\infty}} &\leq& 2 \left\Vert f
\right\Vert_{BV_{\infty}^{\alpha}\left( \mathbb{R} \right)}, \\
\\
\text{and thus } BV_{\infty}^{\alpha}\left( \mathbb{R} \right)
&\hookrightarrow& l_{N}^{\infty}.
\end{eqnarray*}
On the other hand we have 
\begin{eqnarray*}
\left\Vert U_{\alpha}\left( f \right) \right\Vert_{l_{N}^{1}} &=&
\sum\limits_{k=1}^{N} \left\vert \frac{ \tilde{f}\left( t_{k} \right) - 
\tilde{f}\left( t_{k-1} \right)}{\left\vert t_{k} - t_{k-1}
\right\vert^{\alpha}} \right\vert \leq \nu_{1}^{\alpha}(\tilde{f}) + \sup
\left\vert \tilde{f} \right\vert = \left\Vert f
\right\Vert_{BV_{1}^{\alpha}\left( \mathbb{R} \right)}, \\
\\
\text{hence } \left\Vert U_{\alpha}\left( f \right) \right\Vert_{l_{N}^{1}}
&\leq& \left\Vert f \right\Vert_{BV_{1}^{\alpha}\left( \mathbb{R} \right)},
\\
\\
\text{and thus } BV_{1}^{\alpha}\left( \mathbb{R} \right) &\hookrightarrow&
l_{N}^{1}.
\end{eqnarray*}
We have $(L^{\infty}(X), L^{1}(X))_{\frac{1}{p}, p} = L^{p}(X)$, and by the
interpolation Theorem \ref{3.3} we obtain 
\begin{equation*}
\left\Vert U_{\alpha}\left( f \right) \right\Vert_{l_{N}^{p}} \leq c_{p,
\alpha} \left\Vert f \right\Vert_{\left( BV_{\infty}^{\alpha}\left( \mathbb{R%
} \right), BV_{1}^{\alpha}\left( \mathbb{R} \right) \right)_{\frac{1}{p},
p}}, \quad c_{p, \alpha} > 0.
\end{equation*}
Since $\left\Vert f \right\Vert_{BV_{p}^{\alpha}\left( \mathbb{R} \right)} =
\inf\limits_{a} \left\{ a > 0: \left\Vert U_{\alpha}\left( f \right)
\right\Vert_{l_{N}^{p}} \leq a \right\}$, then 
\begin{eqnarray*}
\left\Vert f \right\Vert_{BV_{p}^{\alpha}\left( \mathbb{R} \right)} &\leq&
c_{p, \alpha} \left\Vert f \right\Vert_{\left( BV_{\infty}^{\alpha}\left( 
\mathbb{R} \right), BV_{1}^{\alpha}\left( \mathbb{R} \right) \right)_{\frac{1%
}{p}, p}}, \quad c_{p, \alpha} > 0, \\
&& \\
\text{which gives } \left( BV_{\infty}^{\alpha}\left( \mathbb{R} \right),
BV_{1}^{\alpha}\left( \mathbb{R} \right) \right)_{\frac{1}{p}, p}
&\hookrightarrow& BV_{p}^{\alpha}\left( \mathbb{R} \right).
\end{eqnarray*}

\item[(4)] According to the proof of Theorem \ref{3.17}, and since $0 \leq
\alpha < 1$, then we have 
\begin{equation*}
\left\Vert f \right\Vert_{\mathcal{U}_{p}\left( \mathbb{R} \right)} \leq
2^{1/p} \left\Vert f \right\Vert_{\mathcal{V}_{p}\left( \mathbb{R} \right)}
\leq 2^{1/p} \left\Vert f \right\Vert_{\mathcal{V}_{p}^{\alpha}\left( 
\mathbb{R} \right)},
\end{equation*}
\begin{eqnarray*}
\text{hence } \mathcal{V}_{p}^{\alpha}\left( \mathbb{R} \right)
&\hookrightarrow& \mathcal{V}_{p}\left( \mathbb{R} \right) \hookrightarrow 
\mathcal{U}_{p}\left( \mathbb{R} \right), \\
&& \\
\text{and we deduce that } BV_{p}^{\alpha}\left( \mathbb{R} \right)
&\hookrightarrow& BV_{p}\left( \mathbb{R} \right) \hookrightarrow
U_{p}\left( \mathbb{R} \right), \\
&& \\
\text{ which gives } BV_{p}^{\alpha} &\hookrightarrow& U_{p}.
\end{eqnarray*}

\item[(5)] The continuous inclusion $U_{p} \hookrightarrow \dot{\tilde{B}}%
_{p}^{1/p, \infty}$ was proved in the proof of Theorem \ref{3.17}.
\end{itemize}
\end{proof}

\section{\textbf{Examples}}

\begin{quote}
The following examples are taken from the works of: ~\cite{Bo-Cr-Si-1,Bo-Cr-Si-2,KAT,SIC,Bo-Mo-Si,DAH,IGA,Bo-Cr,Bo-Me}.\\
We denote by $\mathcal{E}_{p}$ the closure of $BV_{p}\left( \mathbb{R}
\right) \cap C^{\infty}\left( \mathbb{R} \right)$ in $BV_{p}\left( \mathbb{R}
\right)$ and by $\mathcal{E}_{p}^{1}$ the closure of $BV_{p}^{1}\left( 
\mathbb{R} \right) \cap C^{\infty}\left( \mathbb{R} \right)$ in $%
BV_{p}^{1}\left( \mathbb{R} \right)$.
\end{quote}

\makeatletter\let\orig@addcontentsline 
\addcontentsline
\renewcommand{\addcontentsline}[3]{}   \makeatother
\subsection{\textbf{Example 1 }~\cite{KAT,Bo-Mo-Si,SIC}}

\begin{itemize}
\item ~\cite{KAT} Let $1 < p < \infty$, $\displaystyle 1 < s < 1 + 1/p$, $1
\leq q \leq \infty$, and let $F$ be a real variable function, Lipschitz, $%
F\left( 0 \right) = 0$, $F \in \dot{B}_{p}^{1 + \left( 1/p \right),
\infty}\left( \mathbb{R} \right)$, then we have

$$F\left( B_{p}^{s, q}\left( \mathbb{R}^{n} \right) \right) \subset
B_{p}^{s, q}\left( \mathbb{R}^{n} \right)$$
and
$$\exists C \left( n, s, p, q \right) > 0: \left\Vert F\left( f \right)
\right\Vert_{B_{p}^{s, q}\left( \mathbb{R}^{n} \right)} \leq C \max \left(
\left\Vert F^{\prime }\right\Vert_{\infty}, \left\Vert F \right\Vert_{\dot{B}%
_{p}^{1 + \left( 1/p \right), \infty}} \right) \cdot \left\Vert f
\right\Vert_{B_{p}^{s, q}\left( \mathbb{R}^{n} \right)}$$

\item ~\cite{KAT} Let the nonlinear operator $F_{\mu}: f \longmapsto
\left\vert f \right\vert^{\mu}, \mu > 0$,

\begin{itemize}
\item If $\displaystyle \mu > 1$, $1 \leq p \leq \infty$, $1 \leq q \leq
\infty$, $0 < s < \mu + \frac{1}{p}$, then there exists a constant $$C\left(
s, p, q, n, \mu \right) > 0,$$ such that:

\begin{itemize}
\item $\displaystyle\forall f \in B_{p}^{s, q}\left( \mathbb{R}%
^{n} \right) \cap L^{\infty}\left( \mathbb{R}^{n} \right),\quad \left\Vert F_{\mu}\left( f \right)
\right\Vert_{B_{p}^{s, q}\left( \mathbb{R}^{n} \right)} \leq C \left\Vert f
\right\Vert_{B_{p}^{s, q}\left( \mathbb{R}^{n} \right)} \left\Vert f
\right\Vert_{\infty}^{\mu - 1}$

\item $\displaystyle \forall g \in B_{p}^{\mu + \left(
	1/p \right), \infty}\left( \mathbb{R} \right) ,\quad \left\Vert \left\vert g \right\vert^{\mu}
\right\Vert_{B_{p}^{\mu + \left( 1/p \right), \infty}\left( \mathbb{R}
\right)} \leq C \left\Vert g \right\Vert_{B_{p}^{\mu + \left( 1/p \right),
\infty}\left( \mathbb{R} \right)}^{\mu}$
\end{itemize}

\item If $\mu > 1$, $1 < p < \infty$, $m < n/p$, then we have an equivalence
between

\begin{itemize}
\item $\displaystyle \left\{ \left\vert f \right\vert^{\mu}: f \in W^{m, t}
\right\} \subset W^{m, p}$, and

\item $\displaystyle m < \mu + \frac{1}{p}, \quad \frac{n \mu}{m \left( \mu
- 1 \right) + \frac{n}{p}} \leq t \leq p \mu$,
\end{itemize}
\end{itemize}

the result remains true if we replace $F_{\mu}: f \longmapsto \left\vert f
\right\vert^{\mu}$, by 
\begin{equation*}
\tilde{F}_{\mu}: f \longmapsto f \left\vert f \right\vert^{\mu - 1}, \quad 
\text{or by } \bar{F}_{\mu}: f \longmapsto \left( \max \left( f, 0 \right)
\right)^{\mu}.
\end{equation*}

\item If $m \in \mathbb{N}$, $1 < p < \infty$, $m < n/p$, $k \in \mathbb{N}$%
, $k \geq 2$, then we have an equivalence between

 $\displaystyle \left\{ f^{k}: f \in W^{m, t} \right\} \subset W^{m, p}$\quad and\quad $\displaystyle \frac{n k}{m \left( k - 1 \right) + \frac{n}{p}} \leq t
\leq p k$.

\end{itemize}

\begin{quote}
The operator $f \longmapsto f \left\vert f \right\vert^{\mu}$ plays an
essential role in the study of the Cauchy initial value problem for
nonlinear Schrudinger partial differential equations.
\end{quote}

\subsection{\textbf{Example 2 }~\cite{Bo-Cr-Si-1,Bo-Cr-Si-2}}

\begin{itemize}
\item Let $p \geq 1$, and consider the absolute value operator $T_{F_{1}} =
T_{\left\vert \cdot \right\vert}$ such that 
\begin{equation*}
T_{F_{1}}\left( g \right) = F_{1}\left( g \right) = \left\vert g \right\vert,
\end{equation*}
then $\left\vert \cdot \right\vert$ belongs to $BV_{p}^{1}\left( \mathbb{R}
\right) \setminus \mathcal{E}_{p, \ell oc}^{1}$, and it verifies

\begin{itemize}
\item $T_{\left\vert \cdot \right\vert}$ operates on $BV_{p}^{1}\left( I
\right)$, such that 
\begin{equation*}
\left\Vert \left\vert g \right\vert \right\Vert_{BV_{p}^{1}\left( I \right)}
\leq c \left\Vert g \right\Vert_{BV_{p}^{1}\left( I \right)}, \quad (c > 0)
\end{equation*}

\item $\displaystyle T_{\left\vert \cdot \right\vert}$ sends $\displaystyle %
B_{p}^{1 + 1/p, 1}\left( \mathbb{R}^{n} \right)$ to $\displaystyle B_{p}^{1
+ 1/p, \infty}\left( \mathbb{R}^{n} \right)$ such that 
\begin{equation*}
\left\Vert \left\vert g \right\vert \right\Vert_{B_{p}^{1 + 1/p,
\infty}\left( \mathbb{R}^{n} \right)} \leq c \left\Vert g
\right\Vert_{B_{p}^{1 + 1/p, 1}\left( \mathbb{R}^{n} \right)}, \quad (c > 0)
\end{equation*}

\item The operator $T_{\left\vert \cdot \right\vert}$ is not continuous from 
$B_{p}^{1 + 1/p, \infty}\left( \mathbb{R}^{n} \right)$ to $B_{p}^{1 + 1/p,
\infty}\left( \mathbb{R}^{n} \right)$.

\item If $0 < s < 1 + (1/p)$, $1 \leq q \leq \infty$, it operates on $%
B_{p}^{s, q}\left( \mathbb{R}^{n} \right)$, such that 
\begin{equation*}
\left\Vert \left\vert g \right\vert \right\Vert_{B_{p}^{s, q}\left( \mathbb{R%
}^{n} \right)} \leq c \left\Vert g \right\Vert_{B_{p}^{s, q}\left( \mathbb{R}%
^{n} \right)}, \quad (c > 0)
\end{equation*}

\item If $0 < s < \infty$, $1 \leq q < \infty$, it operates continuously on $%
B_{p}^{s, q}\left( \mathbb{R}^{n} \right)$.
\end{itemize}

\item Consider the family of functions 
\begin{equation*}
u_{\alpha}(x) = |x + \alpha| - |\alpha|, \quad x, \alpha \in \mathbb{R},
\end{equation*}
then 
\begin{equation*}
\left\Vert u_{\alpha}\left( g \right) \right\Vert_{BV_{p}^{1}\left( \mathbb{R%
} \right)} = \left\Vert |g + \alpha| - |\alpha|
\right\Vert_{BV_{p}^{1}\left( \mathbb{R} \right)} \leq c_{p} \left\Vert g
\right\Vert_{BV_{p}^{1}\left( \mathbb{R} \right)}.
\end{equation*}
\end{itemize}

\subsection{\textbf{Example 3 }~\cite{Bo-Cr-Si-2}}

\begin{itemize}
\item Let $p \geq 1$, and the operator $T_{\psi}: g \longmapsto \psi \circ g$%
, such that

\begin{itemize}
\item $\displaystyle \rho \in \mathcal{D}\left( \mathbb{R} \right)$, $\text{%
supp } \rho \subseteq [-1/2, 1/2]$, $\rho\left( t \right) = 1$ in $[-1/e,
1/e]$,

\item $\displaystyle \psi: \mathbb{R} \longrightarrow \mathbb{R}$, and $%
\psi\left( t \right) = \left\vert t \right\vert \frac{\rho\left( t \right)}{%
\log \left\vert t \right\vert}$, if $t \neq 0$, $\psi\left( 0 \right) = 0$,
\end{itemize}

then the following assertions are verified

\begin{description}
\item[a)] The function $\psi$ belongs to $\mathcal{E}_{p}^{1}$.

\item[b)] $T_{\psi}$ operates on $BV_{p}^{1}\left( I \right)$ such that $%
\left\Vert \psi \circ g \right\Vert_{BV_{p}^{1}\left( I \right)} \leq c
\left\Vert g \right\Vert_{BV_{p}^{1}\left( I \right)}, \quad (c > 0)$.

\item[c)] $T_{\psi}$ sends $B_{p}^{1 + 1/p, 1}\left( \mathbb{R}^{n} \right)$
to $B_{p}^{1 + 1/p, \infty}\left( \mathbb{R}^{n} \right)$ such that 
\begin{equation*}
\left\Vert \psi \circ g \right\Vert_{B_{p}^{1 + 1/p, \infty}\left( \mathbb{R}%
^{n} \right)} \leq c \left\Vert g \right\Vert_{B_{p}^{1 + 1/p, 1}\left( 
\mathbb{R}^{n} \right)}, \quad (c > 0).
\end{equation*}

\item[d)] $T_{\psi}$ is continuous from $B_{p}^{1 + 1/p, 1}\left( \mathbb{R}%
^{n} \right)$ to $\displaystyle B_{p}^{1 + 1/p, \infty}\left( \mathbb{R}^{n}
\right)$.

\item[e)] If $\displaystyle 0 < s < 1 + (1/p)$, $\displaystyle 1 \leq q \leq
\infty$, then $T_{\psi}$ operates on $B_{p}^{s, q}\left( \mathbb{R}^{n}
\right)$, such that 
\begin{equation*}
\left\Vert \psi \circ g \right\Vert_{B_{p}^{s, q}\left( \mathbb{R}^{n}
\right)} \leq c \left\Vert g \right\Vert_{B_{p}^{s, q}\left( \mathbb{R}^{n}
\right)}, \quad (c > 0).
\end{equation*}

\item[f)] If $q \in [1, \infty[$, then the operator in (e) is continuous.
The above properties do not change if $\log \left\vert t \right\vert$ is
replaced by iterated logarithms like $\displaystyle \log \left\vert \log
\left\vert t \right\vert \right\vert$ or $\displaystyle \log \left\vert \log
\left\vert \log \left\vert t \right\vert \right\vert \right\vert$.
\end{description}
\end{itemize}

\subsection{\textbf{Example 4 }~\cite{Bo-Cr-Si-2}}

\begin{itemize}
\item Since the spaces $BV_{p}^{1}\left( \mathbb{R} \right)$ decrease with
respect to $p$, we look for functions that belong to $BV_{p_{0}}^{1}\left( 
\mathbb{R} \right)$ such that $p > p_{0}$, and considering thus the family
of functions $\psi_{\alpha, \beta}: \mathbb{R} \longrightarrow \mathbb{R}$,
defined by 
\begin{equation*}
\psi_{\alpha, \beta}(t) = |t|^{\alpha + 1} \rho\left( t \right) \sin \left(
\left\vert t \right\vert^{-\beta} \right) \text{ if } t \neq 0, \text{ and }
\psi_{\alpha, \beta}\left( 0 \right) = 0, \quad 0 < \beta < \alpha,
\end{equation*}
Then we obtain an equivalence between the following assertions (a), (b), (c)
such that
\end{itemize}

\begin{description}
\item[a)] $\displaystyle \frac{1}{p} < \frac{\alpha}{\beta} - 1$

\item[b)] $\displaystyle \text{The function } \psi_{\alpha, \beta} \text{
belongs to } BV_{p}^{1}\left( \mathbb{R} \right)$

\item[c)] $\displaystyle \text{The operator } T_{\psi_{\alpha, \beta}} \text{
sends } BV_{p}^{1}\left( \mathbb{R} \right) \text{ to } BV_{p}^{1}\left( 
\mathbb{R} \right)$,
\end{description}

\begin{quote}
and if $\displaystyle \frac{1}{p} < \frac{\alpha}{\beta} - 1$, then the
operator $\displaystyle T_{\psi_{\alpha, \beta}}$ satisfies properties
(c)-(f) of the third example if we replace $\psi$ by $\psi_{\alpha, \beta}$.
\end{quote}

\subsection{\textbf{Example 5 }~\cite{DAH,IGA}}

\begin{itemize}
\item For the problem of composition of operators we have
\end{itemize}

\begin{description}
\item[i)] A classical theorem due to B.E.J. Dahlberg (cf ~\cite{DAH}) for
the Sobolev space $W^{m, p}\left( \mathbb{R}^{n} \right)$ stating that 
\begin{equation*}
\text{If } 1 + \frac{1}{p} < m < \frac{n}{p}, \quad p \neq 0, \text{ or } 1
< p < \infty, \quad 2 \leq m < \frac{n}{p}, \text{ then }
\end{equation*}
\begin{eqnarray*}
T_{G}\left( W^{m, p}\left( \mathbb{R}^{n} \right) \right) &=& \left\{
G\left( f \right); f \in W^{m, p}\left( \mathbb{R}^{n} \right) \right\}
\subset W^{m, p}\left( \mathbb{R}^{n} \right), \\
&& \\
\text{implies } \qquad G\left( t \right) &=& c \cdot t, \quad c \in \mathbb{R%
}.
\end{eqnarray*}

\item[ii)] A classical result due to S. Igari (~\cite{IGA}) for the Besov
space $B_{p}^{s, q}\left( \mathbb{R}^{n} \right)$, which states that 
\begin{equation*}
\text{If } 1 \leq p < \infty, \quad 1 \leq q \leq \infty, \quad 0 < s < 1/p
\quad \text{ then }
\end{equation*}
\begin{equation*}
T_{G}\left( B_{p}^{s, q} \right) \subset B_{p}^{s, q}, \quad \text{if and
only if } G \text{ is Lipschitz and } G\left( 0 \right) = 0.
\end{equation*}
\end{description}

\subsection{\textbf{Example 6 }~\cite{Bo-Cr-Si-1,Bo-Cr}}

\begin{itemize}
\item Operators related to continuous Lipschitz functions satisfy the norm
inequality property on $W^{1, p}\left( \mathbb{R}^{n} \right)$ but not on $%
W^{2, p}\left( \mathbb{R}^{n} \right)$, this restriction is due to the
following result

\begin{proposition}[~\cite{Br-Mir}]
\label{4.19} 
Let $\displaystyle 1 < p \leq +\infty$, $\displaystyle s > 1 + \left( 1/p
\right)$ and $N$ a norm on $\mathcal{D}\left( \mathbb{R}^{n} \right)$. 
\newline
If $E$ is a normed space such that $\displaystyle \mathcal{D}\left( \mathbb{R%
}^{n} \right) \subset E \subset W_{\ell oc}^{1, 1}\left( \mathbb{R}^{n}
\right)$ 
\begin{equation*}
\sup_{h \neq 0} \left\vert h \right\vert^{1 - s} \left( \int_{\mathbb{R}%
^{n}} \left\vert \frac{\partial g}{\partial x_{i}}\left( x + h \right) - 
\frac{\partial g}{\partial x_{i}}\left( x \right) \right\vert^{p} \, \mathrm{%
d}x \right)^{1/p} \leq A \left\Vert g \right\Vert_{E}, \quad (A > 0)
\end{equation*}
for all $g \in E$, and $i = 1, \dots, n$, and if there exists a continuously
differentiable function $f: \mathbb{R} \longrightarrow \mathbb{R}$, and a
constant $B > 0$, such that $T_{f}$ injects $\mathcal{D}\left( \mathbb{R}%
^{n} \right)$ into $E$, and such that the following inequality is verified,
then $f$ must be an affine function. 
\begin{equation*}
\left\Vert f \circ g \right\Vert_{E} \leq B \left( N\left( g \right) + 1
\right) \quad \text{for all } g \in \mathcal{D}\left( \mathbb{R}^{n} \right).
\end{equation*}
\end{proposition}

\item Superposition operators related to affine functions trivially satisfy
the norm inequality properties with respect to function composition for
every normed functional space.

\item For the Sobolev space $W^{m, p}\left( \mathbb{R}^{n} \right)$, $m \in 
\mathbb{N}$, $1 \leq p \leq \infty$, nontrivial superposition operators that
satisfy the norm inequality property exist if and only if 
\begin{equation*}
(m = 0) \quad \text{or} \quad (m = 1) \quad \text{or} \quad (m = 2 \text{
and } p = 1).
\end{equation*}
$\bullet $ For every $\mu > 1$, $m < \mu + \frac{1}{p}$, $f \in W^{m,
p}\left( \mathbb{R}^{n} \right) \cap L^{\infty}\left( \mathbb{R}^{n} \right)$
we have 
\begin{equation*}
\left\Vert F_{\mu}\left( f \right) \right\Vert_{W^{m, p}\left( \mathbb{R}%
^{n} \right)} \leq c \left\Vert f \right\Vert_{W^{m, p}\left( \mathbb{R}^{n}
\right)} \left\Vert f \right\Vert_{\infty}^{\mu - 1}, \quad (c > 0).
\end{equation*}
The norm inequality property for the operator $\displaystyle F_{\mu} =
\left\vert \cdot \right\vert^{\mu}$ is not entirely verified on $W^{m,
p}\left( \mathbb{R}^{n} \right)$ but partially on $W^{m, p}\left( \mathbb{R}%
^{n} \right) \cap L^{\infty}\left( \mathbb{R}^{n} \right)$.

\item ~\cite{Bo-Cr} - Let $\displaystyle p \geq 1$, $m > \max \left( n/p, 1
\right)$, $m \in \mathbb{N}$, then 
\begin{equation*}
\text{the composition operator } T_{f} \text{ operates on } W^{m, p}\left( 
\mathbb{R}^{n} \right), \text{ if and only if } f \in W_{\ell oc}^{m,
p}\left( \mathbb{R}^{n} \right).
\end{equation*}
\end{itemize}

\subsection{\textbf{Example 7 }~\cite{Bo-Me}}

\begin{itemize}
\item We obtain a solution for the \textbf{P.S.O.}, with an additional norm
inequality property in the Sobolev spaces $W^{s, p}\left( \mathbb{R}^{n}
\right)$ by Theorem \ref{4.20}, by introducing the space of functions of
bounded variation $BV\left( \mathbb{R}^{n} \right)$ equipped with the
following seminorm:

$\displaystyle \left\Vert f \right\Vert_{BV\left( \mathbb{R}^{n} \right)}
\sim \nu \left( f \right) = \sum\limits_{j=1}^{n} \left\Vert \partial_{j} f
\right\Vert_{M} \sim \sup_{h \in \mathbb{R}^{n} \setminus \left\{ 0
\right\}} \frac{1}{\left\vert h \right\vert} \int_{\mathbb{R}^{n}}
\left\vert f\left( x + h \right) - f\left( x \right) \right\vert \, \mathrm{d%
}x < +\infty$,

where $\left\Vert g \right\Vert_{M}$ denotes the total variation of the
measure $g$, and the space $BH\left( \mathbb{R} \right)$ of distributions
whose derivative belongs to $BV\left( \mathbb{R} \right)$, equipped with the
seminorm 
\begin{equation*}
\left\Vert f \right\Vert_{BH} = \nu \left( f^{\prime }\right) + \left\Vert
f^{\prime }\right\Vert_{\infty}.
\end{equation*}

\begin{theorem}[~\cite{Bo-Me}]
\label{4.20} 
Let $\displaystyle 1 \leq p \leq \infty$, $0 < s < 1 + \left( 1/p \right)$,
then every function $f \in BH\left( \mathbb{R} \right)$, $f\left( 0 \right)
= 0$, operates on $W^{s, p}\left( \mathbb{R}^{n} \right)$, moreover there
exists $c\left( s, p, n \right) > 0$, such that:
\begin{equation*}
 \text{For all } g
\in W^{s, p}\left( \mathbb{R}^{n} \right),\quad \left\Vert f \circ g \right\Vert_{W^{s, p}\left( \mathbb{R}^{n} \right)}
\leq c\left( s, p, n \right) \left\Vert f \right\Vert_{BH} \left\Vert g
\right\Vert_{W^{s, p}\left( \mathbb{R}^{n} \right)}.
\end{equation*}
\end{theorem}
\end{itemize}

\subsection{\textbf{Example 8 }~\cite{Bo-Cr}}

\begin{itemize}
\item Every composition operator such that $T_{f}\left( g\right) =f\circ g$
and $f:\mathbb{R}\longrightarrow \mathbb{R}$, defined on $$B_{\infty
}^{s,\infty }\left( \mathbb{R}^{n}\right) =\mathcal{C}^{s}\left( \mathbb{R}%
^{n}\right),\, s>0$$ the Holder-Zygmund space, verifies 
\begin{equation*}
T_{f}\left( \mathcal{C}^{s}\left( \mathbb{R}^{n}\right) \right) \subset 
\mathcal{C}^{s}\left( \mathbb{R}^{n}\right)
\end{equation*}%
if and only if we have the following conditions
\end{itemize}

\begin{description}
\item[i)] $f$ is continuous and locally Lipschitz, for $0 < s < 1$.

\item[ii)] $f$ belongs locally to $\mathcal{C}^{s}\left( \mathbb{R}^{n}
\right)$, for $s > 1$.

\item[iii)] $f$ is continuous and locally Lipschitz and satisfies the
condition 
\begin{equation*}
f\left( x + t \right) + f\left( x - t \right) - 2f\left( x \right) = O\left( 
\frac{t}{\left\vert \log t \right\vert} \right), \quad \text{with } t
\longrightarrow 0^{+},
\end{equation*}
uniformly on each compact subset of $\mathbb{R}$, for $s = 1$.
\end{description}

 \makeatletter
\let\addcontentsline\orig@addcontentsline \makeatother

\begin{thebibliography}{99}
\bibitem{Bo-Cr} G.~Bourdaud and M.~Lanza de~Cristoforis, \emph{Functional
calculus in {H\"{o}lder}-{Zygmund} spaces}, Trans. Amer. Math. Soc. \textbf{%
354} (2002), 4109--4129.

\bibitem{Bo-Cr-Si-1} G.~Bourdaud, M.~Lanza de~Cristoforis, and W.~Sickel, 
\emph{Superposition operators and functions of bounded p-variation}, Rev.
Mat. Iberoamer (Feb., 2004), (to appear). Preprint 362. Institut de Math\'{e}%
matiques de Jussieu. Unit\'{e} Mixte de Recherche 7586. Universit\'{e} Paris
VI et Paris VII / CNRS. 

\bibitem{Bo-Cr-Si-2} G.~Bourdaud, Massimo~Lanza de~Cristoforis, and
W.~Sickel, \emph{Superposition operators and functions of bounded
p-variation {II}}, Nonlinear Analysis Series A \textbf{Volume. 62} (2005),
483--518.

\bibitem{Bo-Me} G.~Bourdaud and Y.~Meyer, \emph{Le calcul fonctionnel
sous-lin\'{e}aire dans les espaces de {Besov} homog\`{e}nes}, (1st November
2004).

\bibitem{Bo-Mo-Si} G.~Bourdaud, M.~Moussai, and W.~Sickel, \emph{{Towards} {%
Sharp} {Superposition} {Theorems} in {Besov} and {Lizorkin}-{Triebel} {Spaces%
}}, (March 29, 2006).

\bibitem{Br-Mir} H.~Brezis and P.~Mironescu, \emph{Composition in fractional 
{Sobolev} spaces}, Discrete and Continuous Dynamical Systems \textbf{Number
2. Volume 7} (April 2001), 241--246, Website: http://math.smsu.edu/journal.

\bibitem{DAH} B.~E.~J. Dahlberg, \emph{A note on {Sobolev} spaces}, Proc.
Symp. Pure Math. \textbf{N. 35 Vol. 1} (1979), 183--185, MR 81h:46030.

\bibitem{GIR} A.~Giroux, \emph{Notes de cours. {Mesure} et integration}, D%
\'{e}partement de Math\'{e}matiques et Statistique. Universit\'{e} de Montr%
\'{e}al, Mai 2004.

\bibitem{Hein} J.~Heinonen, \emph{{Lectures} on {Lipchitz} {Analysis}},
Lectures at the 14th Jyv\"{a}skyl\"{a} Summer School in August 2004.
Supported by NSF grant DMS 0353549 and DMS 0244421, April 28, 2005.

\bibitem{IGA} S.~Igari, \emph{Sur les fonctions qui op\`{e}rent sur l'espace 
$\widehat{A^{2}}$}, Ann. Inst. Fourier. Grenoble \textbf{15-21} (1965),
525--533.

\bibitem{KAT} D.~Kateb, \emph{On the boundedness of the mapping $f
\longmapsto |f|^{\mu}, \mu > 1,$ on {Besov} {Spaces}}, Math.Nachr \textbf{%
248-249} (2003), 110--128.

\bibitem{KUF} A.~Kufner, \emph{Some difference inequalities with weights and
interpolation}, L. E. Persson \textbf{Volume 1. Number 3} (1998), 437--444,
Mathematical inequalities Applications.

\bibitem{Mau} B.~Maurey and J.~P. Tacchi, \emph{{Ludwing} {Scheffer} et les
extensions du th\'{e}or\`{e}me des accroissements finis}, Version of
27/10/01.

\bibitem{Pee} J.~Peetre, \emph{{New} {Thoughts} on {Besov} {Spaces}},
Mathematics Series. No.1, Mathematics Department. Duke University. Durham.
N.C, 1976. MR \MRhref{57:1108}{57:1108}.

\bibitem{Po} C.~Portenier, \emph{Cours d'analyse fonctionnelle}, Version of
November 29, 2004.

\bibitem{SIC} W.~Sickel, \emph{Necessary conditions on composition operators
acting on {Sobolev} spaces of fractional order}, Forum Math.\,9 (1997),
267--302.

\bibitem{Trie} H.~Triebel, \emph{{Theory} of {Function} {Spaces} {III}},
Birkh\"{a}user Verlag, P. O. Box : 133. CH-4010 Basel. Switzerland. e-ISBN :
3-7643-7643-7582-5, 2006.

\bibitem{ULLR} D.~C. Ullrich, \emph{{Besov} spaces. {A} primer}, Department
of Mathematics Oklahoma State University. Stillwater OK 74078.
\end{thebibliography}
\end{document}